%% file: hinkelcubiclw.tex
\documentclass[12pt,notitlepage]{report}

\usepackage{verbatim}
\usepackage{amssymb,amsmath,amsthm}

\numberwithin{equation}{section}

\input{mathenv}

\input{syms}

\usepackage{graphicx}
\usepackage{psfrag}
\usepackage{afterpage}
\usepackage{subfigure}
\usepackage{bigfoot}
\usepackage{caption}
\usepackage{alltt}
\usepackage[hidelinks]{hyperref}
\usepackage{ifpdf}
\usepackage{seqsplit}
\usepackage{booktabs}
\usepackage{titlesec}
\titleformat{\chapter}{}{}{0em}{\bf\huge\thechapter.~}

\graphicspath{ {figures/} }

\makeatletter
\newcommand*{\toccontents}{\@starttoc{toc}}
\makeatother

\newcommand\blfootnote[1]{%
  \begingroup
  \renewcommand\thefootnote{}\footnote{#1}%
  \addtocounter{footnote}{-1}%
  \endgroup
}

\renewcommand{\arraystretch}{1.2}
\newcommand{\ra}[1]{\renewcommand{\arraystretch}{#1}}

\DeclareMathOperator{\disc}{disc}

\renewcommand\Re{\operatorname{Re}}
\renewcommand\Im{\operatorname{Im}}

\newcommand{\bz}{\begin{itemize}}
\newcommand{\ez}{\end{itemize}}
\newcommand{\ep}{\varepsilon}
\newcommand{\mb}{\mathbb}

\def\lb{\langle}
\def\rb{\rangle}

\def\Z{\mb{Z}}
\def\C{\mb{C}}

\def\R{\mb{R}}
\def\Q{\mb{Q}}
\def\N{\mb{N}}
\def\O{\mathcal{O}}

\begin{document}

\author{Dustin Hinkel}
\title{Constructing Simultaneous Diophantine Approximations of Certain Cubic Numbers}
\date{}
\maketitle

\blfootnote{{\em Date:} December~8, 2014}

\input{cubiclwabs.tex}

\toccontents



\input{cubiclwbody.tex}

\titleformat{\chapter}{}{}{0em}{\bf\huge}
\input{cubiclwbib.tex}

\end{document}

%% file: mathenv.tex
\usepackage{ifthen}

\theoremstyle{plain}

\newtheorem{cor}[equation]{Corollary}
\newtheorem{lem}[equation]{Lemma}

\newtheorem*{thm*}{Theorem}
\newtheorem*{cor*}{Corollary}
\newtheorem*{lem*}{Lemma}
\newtheorem*{prop*}{Proposition}
\newtheorem*{ax*}{Axiom}

\theoremstyle{definition}
\newtheorem{defn}[equation]{Definition}

\newtheorem*{defn*}{Definition}
\newtheorem*{problem*}{Problem}
\newtheorem{ex}[equation]{Example}

\newtheorem*{ex*}{Example}
\newtheorem*{exs*}{Examples}

\newtheorem{rem}[equation]{Remark}
\newtheorem*{rem*}{Remark}

\newtheorem*{aside*}{Aside}

\newtheorem*{intuition*}{Intuition}

\newtheorem*{notn*}{Notation}

\newtheorem*{conv*}{Convention}

\newtheorem*{interp*}{Interpretation}

\newtheorem*{mnem*}{Mnemonic}

\newtheorem*{exer*}{Exercise}

\newtheorem*{conjecture*}{Conjecture}

\newcommand{\beginex}{\begin{ex}$\rhd$ }
\newcommand{\mendex}{\hfill$\lhd$\end{ex}}

\newcounter{ppart}



%% file: syms.tex
\def\Z{\mathbb{Z}}

\newcommand{\dd}{\mathrm{d}}



%% file: cubiclwabs.tex
\begin{abstract}
For $K$ a cubic field with only one real embedding and $\alpha,\beta\in K$, we show how to construct an increasing sequence $\{m_n\}$ of positive integers and a subsequence $\{\psi_n\}$ such that (for some constructible constants $C_1,C_2>0$) 
$
\max\{\|m_n\alpha\|,\|m_n\beta\|\}<\frac{C_1}{m_n^{1/2}} 
$ 
and 
$
\|\psi_n\alpha\|<\frac{C_2}{\psi_n^{1/2}\log \psi_n} 
$ 
for all $n$.
As a consequence, we have 
$
\psi_n\|\psi_n\alpha\|\|\psi_n\beta\|<\frac{C_1 C_2}{\log \psi_n}, 
$ 
thus giving an effective proof of Littlewood's conjecture for the pair $(\alpha,\beta)$. Our proofs are elementary and use only standard results from algebraic number theory and the theory of continued fractions.

\end{abstract}

%% file: cubiclwbody.tex

\chapter{Introduction}
\label{chap:intro}
\input{cubiclwintro.tex}


\setlength{\parindent}{0cm}

\chapter{Main Results}
\label{chap:results}

\input{cubiclwmainresults.tex}



\section{Lemmas}
\label{sec:lemmas}

\input{cubiclwlemmas.tex}



\eject

\section{Theorem 1}
\label{sec:thm1}

\input{disthm1.tex}


\section{Theorem 2}
\label{sec:thm2}

\input{disthm2.tex}


\section{Theorem 3}
\label{sec:thm3}
\input{disthm3.tex}

\section{Further Questions}
\label{sec:furtherquestions}
\input{disfurtherquestions.tex}


\chapter{Littlewood's Conjecture (Examples)}
\label{sec:exs}
\input{disexamples.tex}


\appendix
\chapter{Commands and Algorithms}
\label{appendix}
\input{disappendix.tex}

%% file: cubiclwintro.tex
\section{Background and Context}

Littlewood conjectured that for real numbers $\alpha$ and $\beta$,
\begin{equation}\label{LWconj}
\liminf_{n\to\infty} n\|n\alpha\|\|n\beta\|=0,
\end{equation}
where $\|x\|=|x-\lfloor x+\frac{1}{2} \rfloor|$ denotes the
distance from $x$ the nearest integer. It is straightforward\footnote{See Remark~\ref{rem:1abdependent} for example.} to show this is true if $1,\alpha,\beta$ are $\Q$-linearly dependent. From the theory of continued fractions, we know that the conjecture is true if either $\alpha$ or $\beta$ has unbounded partial quotients in its continued fraction expansion. Although almost all real numbers have unbounded partial quotients (with respect to Lebesgue measure -- see Theorem 29 of \cite{bib:Kh}), we actually have very few examples\footnote{
Most of these involve values of the exponential function, for instance
 (see \cite{bib:Le1}, \cite{bib:Sh}, \cite{bib:Th}, \cite{bib:RS})
\begin{gather*}
e=[2;1,2,1,1,4,1,1,6,1,\dots]=[2;\overline{1,2n,1}]_{n=1}^\infty,\\
e^{1/k}=[1;\overline{(2n-1)k-1,1,1}]_{n=1}^\infty,\\
\coth(1/k)=\frac{e^{2/k}+1}{e^{2/k}-1}=[k;3k,5k,7k,9k,\dots]
\end{gather*}}
of common numbers which are known to have unbounded partial quotients. For instance, we don't know whether the partial quotients are bounded or unbounded for $\pi$ or for any specific non-quadratic algebraic numbers.

Cassels and Swinnerton-Dyer had the first major result with their 1955 paper \cite{bib:CaSw-Dy}. They showed that
 the conjecture holds for the pair $(\alpha,\beta)$ when $\alpha,\beta\in\R$ are both in the same cubic field. Their proof involved showing that
\begin{equation*}\label{result:casw-dy}
\inf_{m,n\neq 0}|mn|\|m\alpha+n\beta\|=0,
\end{equation*}
and then showing that this implies \eqref{LWconj}.

Peck showed a slightly stronger result in 1961 with \cite{bib:Pe}. He showed that if $K\subset \R$ is an algebraic extension of $\Q$ of degree $n$, and if $\gamma_1,\dots,\gamma_n$ is a $\Q$-basis of $K$, then there exist infinitely many integer $n$-tuples $(q_1,\dots,q_n)$ with $q_1>0$ and $\gcd(q_1,\dots,q_n)=1$ such that
\begin{equation*}
 |q_1\gamma_j-q_j\gamma_1|<\frac{C_1}{q_1^{1/(n-1)}\log q_1^{1/(n-2)}}
\end{equation*}
(for $j=2,\dots,n-2$), and
\begin{equation*}
|q_1\gamma_n-q_n\gamma_1|<\frac{C_2}{q_1^{1/(n-1)}}.
\end{equation*}
In particular, if $[K:\Q]=3$ and $(\gamma_1,\gamma_2,\gamma_3)=(1,\alpha,\beta)$, then there are infinitely many triples $(q_1,q_2,q_3)$ with $q=q_1>0$ such that
\begin{gather}
\|q\alpha\|=|q\gamma_2-q_2|<\frac{C_1}{q^{1/2}\log q},\label{ineq:Peck1}\\
\|q\beta\|=|q\gamma_3-q_3|<\frac{C_2}{q^{1/2}},\label{ineq:Peck2}
\end{gather}
or
\begin{equation*}
q\|q\alpha\|\|q\beta\|<\frac{C_3}{\log q}.
\end{equation*}
This implies the results from \cite{bib:CaSw-Dy}.

Pollington and Velani showed \cite{bib:PoVe} that if $\alpha$ has bounded partial quotients, there exist uncountably many $\beta$ with bounded quotients such that
\begin{equation*}
q\|q\alpha\|\|q\beta\|<\frac{1}{\log q}
\end{equation*}
for infinitely many $q\in\N$.

In \cite{bib:deMa}, de Mathan showed how to construct, for a given quadratic $\alpha$, a $\Q$-linearly independent irrational $\beta$ with bounded partial quotients such that Littlewood's conjecture holds for the pair $(\alpha,\beta)$. This was the first explicit example of $\Q$-linearly independent pairs $(\alpha,\beta)$ with (known) bounded partial quotients satisfying Littlewood's conjecture.

Einsiedler, Katok, Lindenstrauss \cite{bib:EiKaLi} proved that the set of counterexamples to Littlewood's conjecture has Hausdorff dimension 0.

Adamczewski and Bugeaud \cite{bib:AdBu} showed how, given an $\alpha$ with bounded partial quotients, to construct uncountably many $\beta$ with bounded partial quotients such that the conjecture holds for $(\alpha,\beta)$.

Modifying the arguments of Peck in \cite{bib:Pe}, Bugeaud \cite{bib:Bu2} has recently shown, for $K$ a cubic field with only one real embedding, and for $1,\alpha,\beta$ a $\Q$-basis of $K$, how to construct a linear recursive sequence $\{q_n\}$ of positive integers eventually satisfying
\begin{equation*}
\max\{\|q_n\alpha\|,\|q_n\beta\|\}<\frac{C}{q_n^{1/2}}
\end{equation*}
for some $C$.

\section{Our Results}

The work of this paper is most similar to results in Peck \cite{bib:Pe} and Bugeaud \cite{bib:Bu2}. What differs is that we effectively construct sequences whose terms satisfy Peck's inequalities \eqref{ineq:Peck1} and \eqref{ineq:Peck2}. These inequalities motivate the following definition.

For the sake of stating our results more easily, we will call a sequence $\{s_n\}$ of positive integers a {\bf Peck sequence} for the pair $(\sigma,\tau)$ if there are constants $M_1,M_2$ and a subsequence $\{\psi_n\}$ of $\{s_n\}$ such that for all $n$
\begin{gather*}
\max\{\|s_n \sigma\|,\|s_n\tau\|\}<\frac{M_1}{s_n^{1/2}},\\
\psi_n\|\psi_n\sigma\|<\frac{M_2}{\psi_n^{1/2}\log \psi_n}.
\end{gather*}

In Theorem~1, we show that if $\theta$ is the only real root of an irreducible cubic of the form $x^3-Px-Q\in\Z[x]$, then we can construct a Peck sequence for the pair $(\theta,\theta^2)$. In Theorem~2, we show that if $\alpha$ is the only real root of an irreducible cubic in $\Q[x]$, then we can construct a Peck sequence for the pair $(\alpha,\alpha^2)$. Finally, we show in Theorem~3 that if $K$ is a real cubic field with only one real embedding, and if $\alpha,\beta\in K$, then we can construct a Peck sequence for the pair $(\alpha,\beta)$. As a consequence, we can construct $n$ for which the Littlewood product $n\|n\alpha\|\|n\beta\|$ is arbitrarily small, thus providing an effective proof of Littlewood's conjecture for the pair $(\alpha,\beta)$. Other than relying on Dirichlet's units theorem and the ability to produce a unit in a ring of cubic integers, our proofs are elementary and constructive.

Chapter~2 contains the proofs of our theorems. Chapter~3 is devoted to showing how to construct Peck sequences for several examples of $(\alpha,\beta)$. In the Appendix, we give some of the algorithms we've used in our constructions.

\section{Overview and Motivating Example}
\label{sec:overview}

This work arose from considering the cubic pell equation
\begin{equation}\label{eq:cubicpell}
x^3+my^3+m^2z^3-3mxyz=N_{K/\Q}(x+y\theta+z\theta^2)=1
\end{equation}
(see Chapter 7 of \cite{bib:Ba}), where $m\in\Z$ is not a perfect cube, $\theta=\sqrt[3]{m}$, and $K=\Q(\theta)$. (We assume $m>0$.) It is a tedious-but-not-difficult algebra exercise to verify that \eqref{eq:cubicpell} factors as
\begin{equation}\label{eq:normfac}
N_{K/\Q}(x+y\theta+z\theta^2)
=\frac{1}{2}
(x+y\theta+z\theta^2)
((x-y\theta)^2+(x-z\theta^2)^2+(y\theta-z\theta^2)^2).
\end{equation}
From this, we see that
\begin{equation}\label{ineq:origbound}
\max\{
|x-y\theta|,|x-z\theta^2|,|y\theta-z\theta^2|
\}
\leq
\left(
    \frac{2 N_{K/\Q}(x+y\theta+z\theta^2)}{x+y\theta+z\theta^2}
    \right)^{1/2}.
\end{equation}
So if $|N_{K/\Q}(x+y\theta+z\theta^2)|$ is ``small'' compared to $|x+y\theta+z\theta^2|$, then $|x-y\theta|$, $|y\theta-z\theta^2|$, and $|x-z\theta^2|$ will also be ``small''. In this case we would have
\begin{equation*}
\frac{y}{z}\approx \theta
\mbox{\quad\quad and \quad\quad}
\frac{x}{z}\approx \theta^2.
\end{equation*}

That is, we can find simultaneous Diophantine approximations to~$\theta$ and~$\theta^2$ by considering~$\zeta$ with small norm and large absolute value.\\

In particular, we consider units with large absolute value. Dirichlet's units theorem guarantees that we can produce infinite sequences $\{\lambda^n\}$ of distinct units (where we can assume $\lambda>1$), so we have $N_{K/\Q}(\lambda)=1$ and $\lambda^n\to\infty$. If it is the case that $\O_K=\Z[\theta]$,
and if we
define\footnote{We discuss how to compute these in Remark \ref{rem:kncomp} and in the Appendix.}
$a_n$, $b_n$, and $c_n$ as the coordinates of $\lambda^n$ in the basis $1, \theta, \theta^2$, then we will see (using \eqref{ineq:origbound}) that eventually
\begin{align*}
\|c_n\theta\|
=
\frac{|b_n\theta-c_n\theta^2|}{\theta}
<
\frac{1}{\theta}\cdot
\frac{\sqrt{2}}{\lambda^{n/2}}
<
\frac{C}{c_n^{1/2}},\\
\|c_n\theta^2\|
=
|a_n-c_n\theta^2|
<
\frac{\sqrt{2}}{\lambda^{n/2}}
<
\frac{C}{c_n^{1/2}}
\end{align*}
for some constant $C$ which is independent of $n$. In particular, the sequence
\begin{equation*}
\{c_n\|c_n\theta\|\|c_n\theta^2\|\}
\end{equation*}
of Littlewood products is bounded. If we could make either $c_n^{1/2}\|c_n\theta\|$ or $c_n^{1/2}\|c_n\theta^2\|$ arbitrarily small, then we could make $c_n\|c_n\theta\|\|c_n\theta^2\|$ arbitrarily small. That is, we would have a constructive proof of Littlewood's conjecture for the pair $(\sqrt[3]{m},\sqrt[3]{m^2})$.

\subsection{Example 1}
\label{intro:ex1}
Consider the pair $(\theta,\theta^2)$ for $\theta=\sqrt[3]{2}$.  Now $\lambda=1+\theta+\theta^2$ is a unit greater than 1 ($\lambda\approx 3.8473$), so we can use the sequence $\{\lambda^n\}$ to produce a sequence of simultaneous approximations to $(\theta,\theta^2)$. Define
$a_n$, $b_n$, and $c_n$ by
\begin{equation*}
a_n+b_n\theta+c_n\theta^2=\lambda^n=(1+\theta+\theta^2)^n.
\end{equation*}
The first few $\lambda^n$ are
\begin{gather*}
\lambda^1=1+\theta+\theta^2,\\
\lambda^2=5+4\theta+3\theta^2, \\
\lambda^3=19+15\theta+12\theta^2, \\
\lambda^4=73+58\theta+46\theta^2, \\
\lambda^5=281+223\theta+177\theta^2,
\end{gather*}
and first several $c_n$ are:
\begin{equation*}
1,\ 3,\ 12,\ 46,\ 177,\ 681,\ 2620,\ 10080,\ 38781,\ 149203,\ 574032,\ 2208486,\
8496757.
\end{equation*}

By the previous discussion, we have that $\{c_n^{1/2}\|c_n\theta\|\}$ and $\{c_n^{1/2}\|c_n\theta^2\|\}$, and therefore \begin{equation*}
\{c_n\|c_n\theta\|\|c_n\theta^2\|\},
\end{equation*}
are bounded. In Figure~\ref{fig:LWprodcn} we show $c_n\|c_n\theta\|\|c_n\theta^2\|$ for $n\leq 1000$. To provide some context, we also show $n\|n\theta\|\|n\theta^2\|$ for $n\leq 1000$ in Figure~\ref{fig:LWprodall}.

\begin{figure}[!hptb]
    \centering
    \includegraphics[width=.8\textwidth]{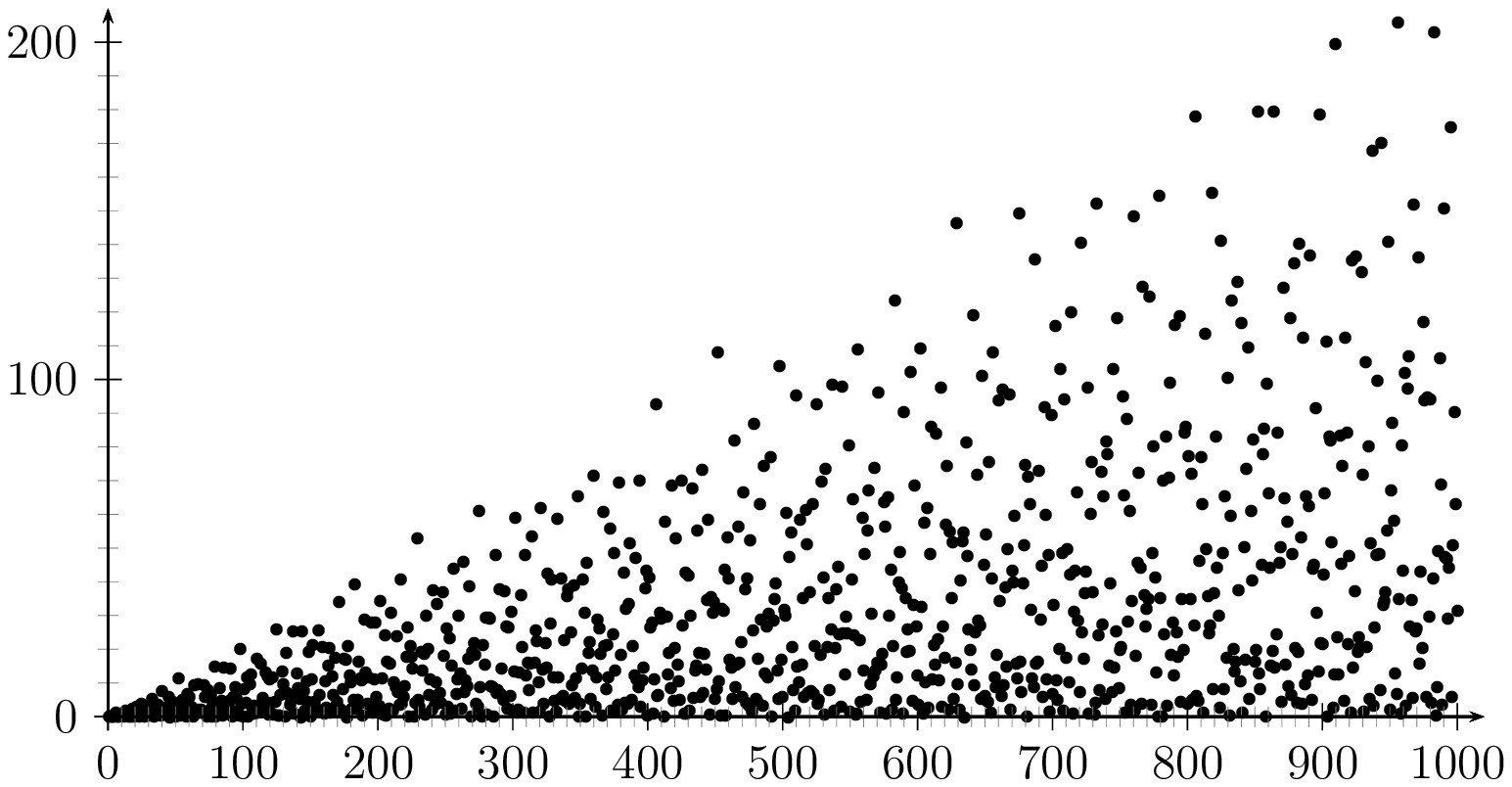}
    \caption{
    $n\|n\sqrt[3]{2}\|\|n\sqrt[3]{4}\|$ for $n\leq 1000$
    }
    \label{fig:LWprodall}
%
    \includegraphics[width=.8\textwidth]{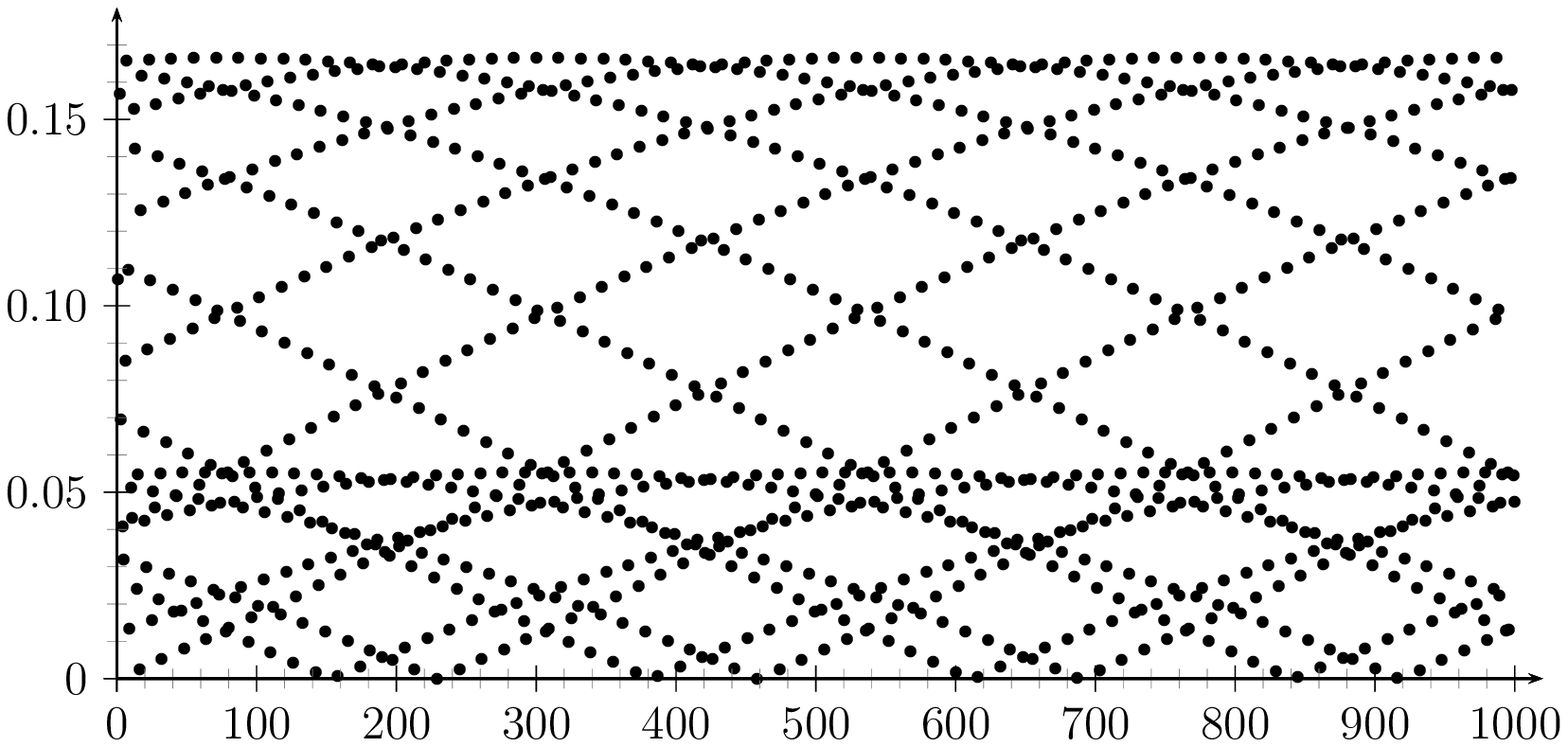}
    \caption{
    $c_n\|c_n\sqrt[3]{2}\|\|c_n\sqrt[3]{4}\|$ for $n\leq 1000$
    }
    \label{fig:LWprodcn}
\end{figure}

\begin{figure}[!hptb]
    \centering
    \includegraphics[width=.95\textwidth]{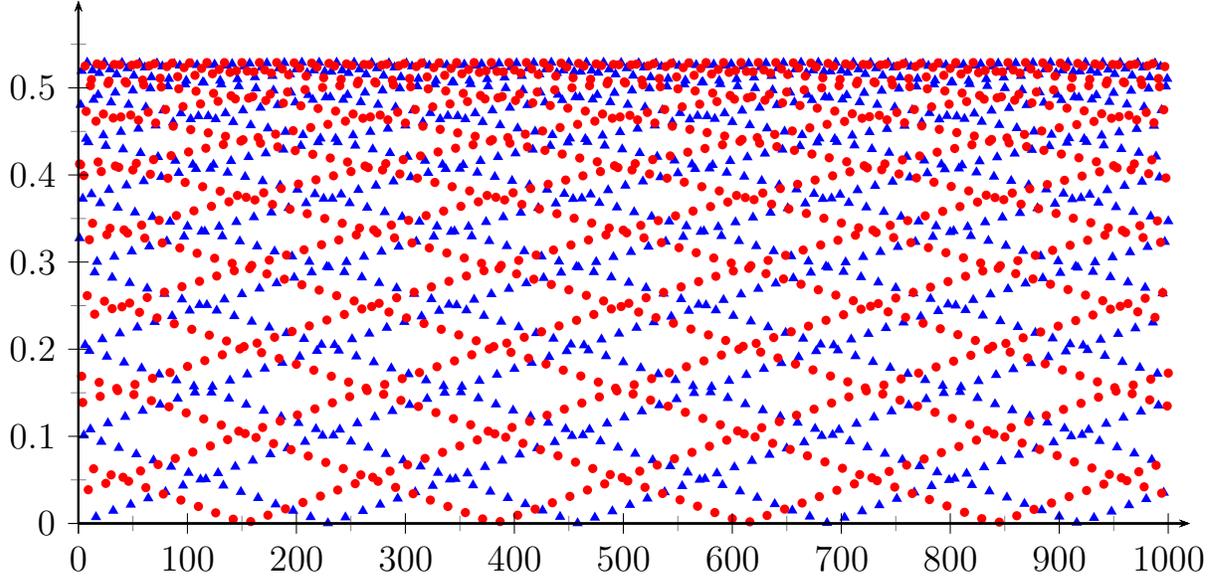}
    \caption{
    $\theta\cdot c_n^{1/2}\|c_n\sqrt[3]{2}\|$ (triangles) and $c_n^{1/2}\|c_n\sqrt[3]{4}\|$ (circles) for $n\leq 1000$
    }
    \label{fig:LWfactorsredblue}
\end{figure}


Figure~\ref{fig:LWfactorsredblue} shows $c_n^{1/2}\theta\|c_n\theta\|=c_n^{1/2}|b_n\theta-c_n\theta^2|$ and $c_n^{1/2}\|c_n\theta^2\|=c_n^{1/2}|a_n-c_n\theta^2|$ for $n\leq 1000$. Since the graphs are so similar, we consider the points $(c_n^{1/2}\|c_n\theta\|,c_n^{1/2}\|c_n\theta^2\|)$.

In Figure~\ref{fig:LWfactorspointsAbsVal} we have the first 500 points $(c_n^{1/2}\|c_n\theta\|,c_n^{1/2}\|c_n\theta^2\|)$. From this graph, we see that $c_n^{1/2}\|c_n\theta\|$ and $c_n^{1/2}\|c_n\theta^2\|$ almost completely determine each other. In order to uncover more information about this relationship, we use the following notation.

\begin{defn}
For $x\in\R$, we define\footnote{This is more commonly denoted by $\{\cdot\}$. We use $\lb\cdot\rb$ to avoid confusion with sequence/set notation.}
$\langle x\rangle=x-\left\lfloor x+\frac{1}{2}\right\rfloor$. (So $\|x\|=|\langle x\rangle|$.)
\end{defn}

\begin{figure}[!h ptb]
    \centering
    \includegraphics[width=.75\textwidth]{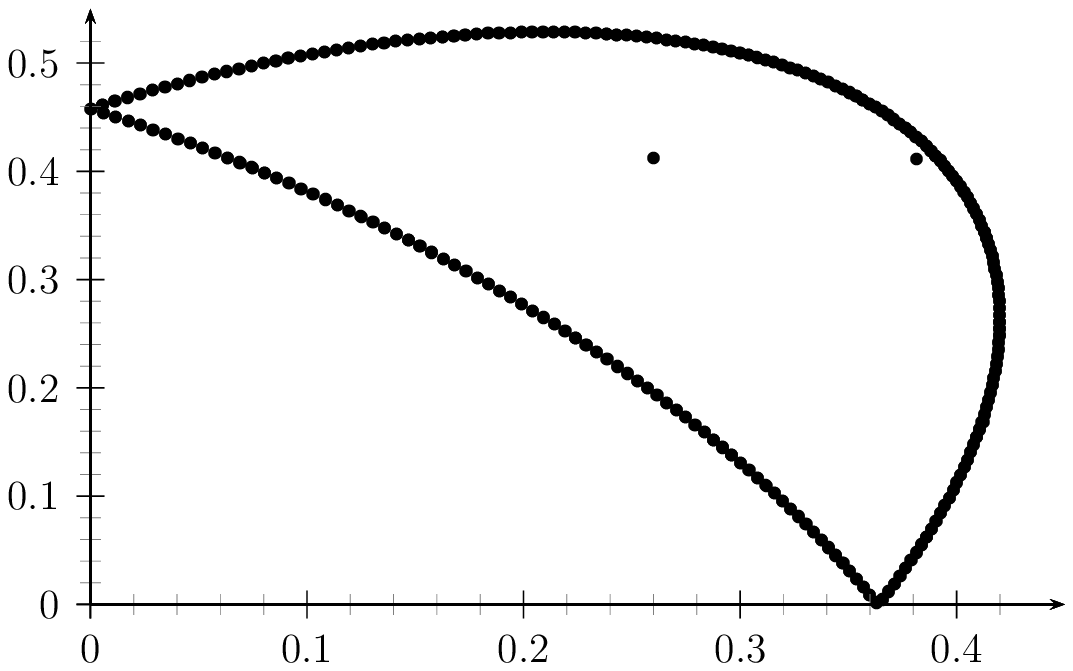}
    \captionof{figure}{
    The points $(c_n^{1/2}\|c_n\sqrt[3]{2}\|,c_n^{1/2}\|c_n\sqrt[3]{4}\|)$ for $n\leq 500$
    }
    \label{fig:LWfactorspointsAbsVal}
    \includegraphics[scale=.75]{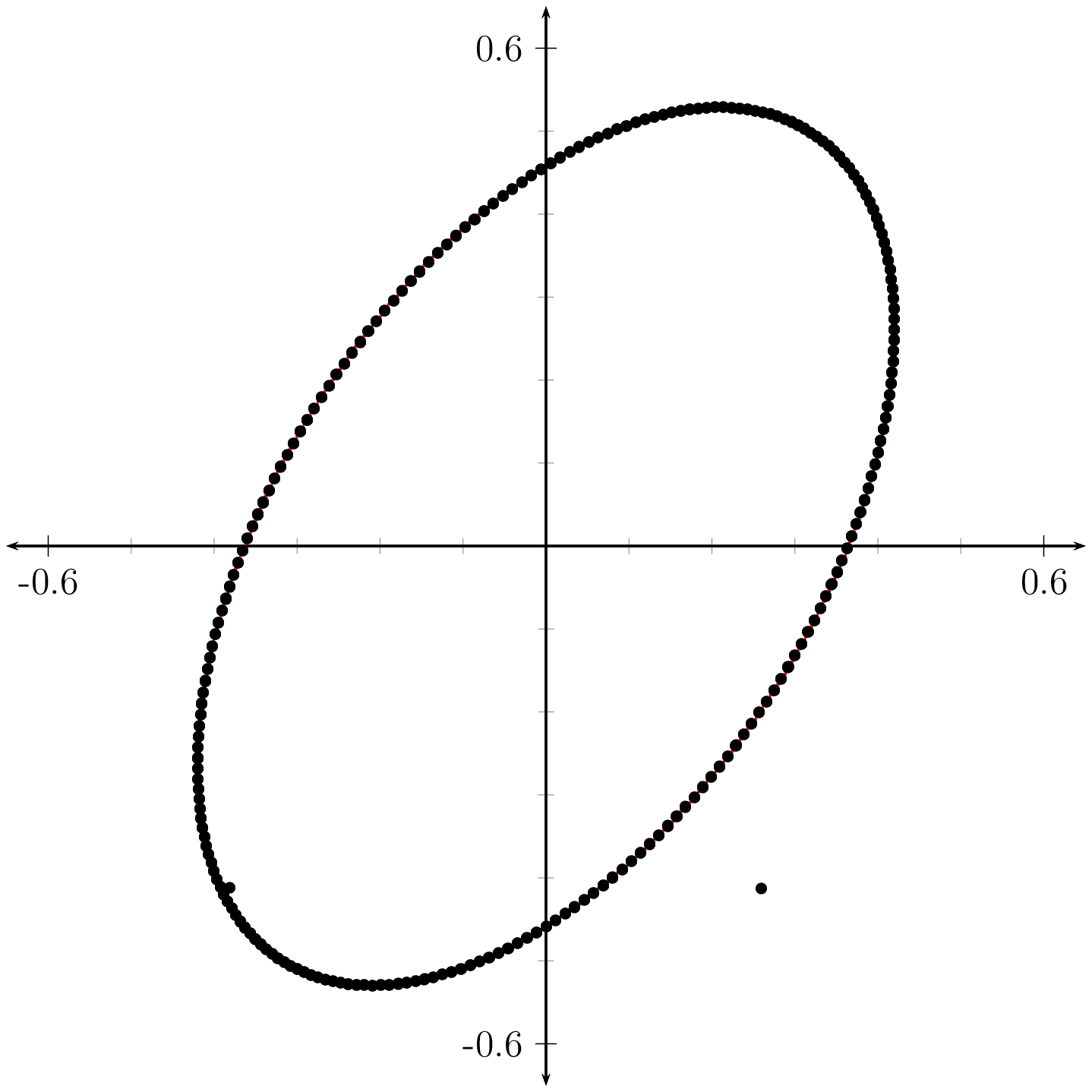}
    \captionof{figure}{
    $(c_n^{1/2}\langle c_n\sqrt[3]{2}\rangle,c_n^{1/2}\langle c_n\sqrt[3]{4}\rangle)$ for
    $n\leq 500$.
    }
    \label{fig:LWseqex1long}
\end{figure}
\begin{figure}[!hptb]
    \centering
    \includegraphics[scale=.8]{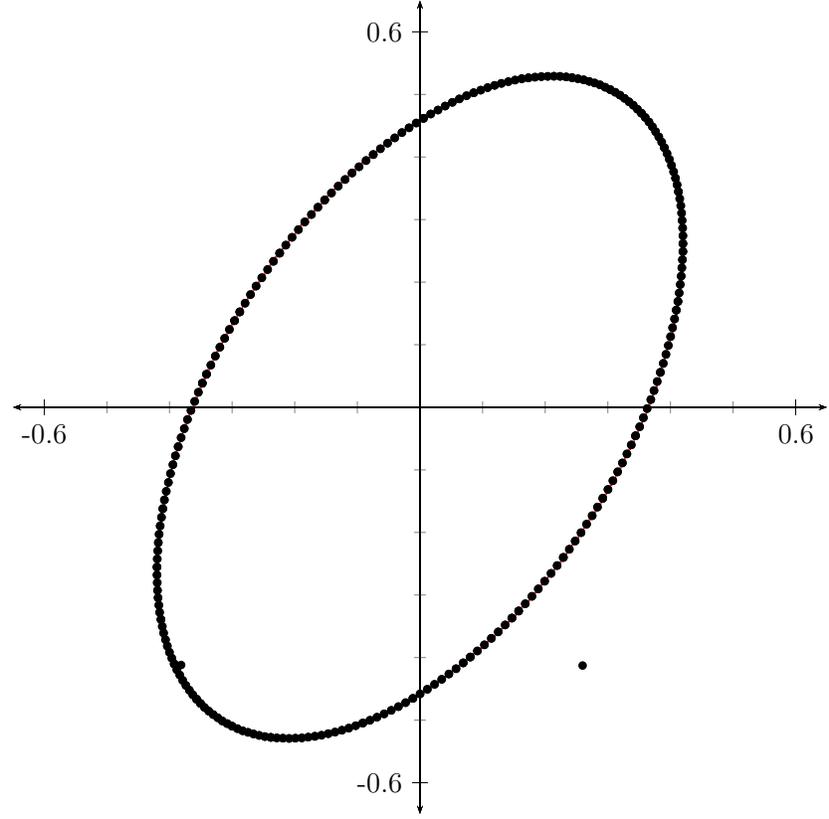}
    \captionof{figure}{
    $(c_n^{1/2}\langle c_n\sqrt[3]{2}\rangle,c_n^{1/2}\langle c_n\sqrt[3]{4}\rangle)$ for
    $n\leq 25$.
    }
    \label{fig:LWseqex1short}
\end{figure}

By considering $\lb c_n\theta\rb$ and $\lb c_n\theta^2\rb$, we can better see the relationship between $\|c_n\theta\|$ and $\|c_n\theta^2\|$. We plot the points $(c_n^{1/2}\lb c_n\theta\rb,c_n^{1/2}\lb c_n\theta^2\rb)$ for $n\leq 500$ in Figure~\ref{fig:LWseqex1long} and for $n\leq 25$ in Figure~\ref{fig:LWseqex1short}. From these graphs, we see that the sequence $\{(c_n^{1/2}\langle c_n\theta\rangle,c_n^{1/2}\langle c_n\theta^2\rangle)\}$ appears to fill an ellipse. In Lemma~\ref{lem:ellipse} we will see how to derive that the equation of the ellipse in this example is
\begin{equation*}
(\sqrt[3]{2}\cdot x-2y)^2+3\sqrt[3]{4}\cdot x^2=\frac{4}{3\sqrt[3]{4}},
\end{equation*}
and that the sequence is in fact asymptotic to the ellipse. In Table~\ref{tab:LWseqcbrt2} we list $c_n$, $c_n^{1/2}\langle c_n\theta\rangle$, $c_n^{1/2}\langle c_n\theta^2\rangle$, and $c_n\|c_n\theta\|\|c_n\theta^2\|$ for the first several $n$.

From Figure~\ref{fig:LWseqex1short}, it appears that the points go around the ellipse at a fairly constant rate of rotation. For $n>1$, going forward five terms in the sequence yields a point on the ellipse a little more than two full counterclockwise rotations away, so the angle between successive terms would be a little more than $4\pi/5\approx 2.51$. Indeed, by transforming the ellipse to a circle and applying the same transformation to the points $\{(c_n^{1/2}\langle c_n\theta\rangle,c_n^{1/2}\langle c_n\theta^2\rangle)\}$, we find that the angles between successive points seem to converge to roughly $2.551688241$. If we knew the exact angle of rotation, then maybe we could use it to predict when either $c_n^{1/2}\langle c_n\theta\rangle$ and $c_n^{1/2}\langle c_n\theta^2\rangle$ gets arbitrarily close to zero.

We will see that by considering the sequence $\{Q_n\}$ of denominators of convergents of $\phi/\pi$, where
\begin{equation*}
\phi
=
\arctan
    \left(
    \frac{\sqrt{3\theta^2}\cdot(b_1-c_1\theta)}{2a_1-b_1\theta-c_1\theta^2}
    \right)
=
\arctan\left(\frac{\sqrt{3}(\theta-\theta^2)}{2-\theta-\theta^2}\right)
\end{equation*}
(and where $\pi-\phi\approx 2.551688241$), we have
\begin{equation*}
c_{Q_n}^{1/2}\|c_{Q_n}\theta\|<\frac{C}{Q_{n+1}}
\end{equation*}
for some constant $C$. Therefore
\begin{equation*}
\lim_{n\to\infty}
c_{Q_n}\|c_{Q_n}\theta\|\|c_{Q_n}\theta^2\|
=\lim_{n\to\infty}
    \left(
    |c_{Q_n}^{1/2}\langle c_{Q_n}\theta\rangle|
    \cdot
    |c_{Q_n}^{1/2}\langle c_{Q_n}\theta^2\rangle|
    \right)
=0.
\end{equation*}

\begin{table}[p!htb]
\centering
\ra{1.3}
\begin{tabular}{@{} lclll @{}}
\toprule
$n$
& $a_n+b_n\theta+c_n\theta^2=(1+\theta+\theta^2)^n$
&$c_n^{1/2}\langle c_n\theta\rangle$
&$c_n^{1/2}\langle c_n\theta^2\rangle$
& $c_n\|c_n\theta\|\|c_n\theta^2\|$\\
\midrule
1  &  $1 + \theta + \theta^2 $  &  0.2599210  &  -0.4125989  &  0.1072432\\
2  &  $5 + 4\theta + 3\theta^2 $  &  -0.3814614  &  -0.4118762  &  0.1571149\\
3  &  $19 + 15\theta + 12\theta^2 $  &  0.4124103  &  0.1690919  &  0.0697352\\
4  &  $73 + 58\theta + 46\theta^2 $  &  -0.2959246  &  0.1386877  &  0.04104111\\
5  &  $281 + 223\theta + 177\theta^2 $  &  0.08016847  &  -0.3993077  &  0.03201189\\
6  &  $1081 + 858\theta + 681\theta^2 $  &  0.1627079  &  0.5249569  &  0.0854146\\
7  &  $4159 + 3301\theta + 2620\theta^2 $  &  -0.3505866  &  -0.4731548  &  0.1658817\\
8  &  $16001 + 12700\theta + 10080\theta^2 $  &  0.4199639  &  0.2614234  &  0.1097884\\
9  &  $61561 + 48861\theta + 38781\theta^2 $  &  -0.3473879  &  0.03867265  &  0.01343441\\
10  &  $236845 + 187984\theta + 149203\theta^2 $  &  0.1573906  &  -0.3256969  &  0.0512616\\
11  &  $911219 + 723235\theta + 574032\theta^2 $  &  0.08580665  &  0.5026315  &  0.04312913\\
12  &  $3505753 + 2782518\theta + 2208486\theta^2 $  &  -0.3000002  &  -0.5096706  &  0.1529013\\
13  &  $13487761 + 10705243\theta + 8496757\theta^2 $  &  0.4127900  &  0.3444347  &  0.1421792\\
\midrule
$n$
&
$c_n$
&$c_n^{1/2}\langle c_n\theta\rangle$
&$c_n^{1/2}\langle c_n\theta^2\rangle$
& $c_n\|c_n\theta\|\|c_n\theta^2\|$\\

\midrule

14  &  32689761  &  -0.386052  &  -0.0627756  &  0.0242346\\
15  &  125768040  &  0.228823  &  -0.240102  &  0.054941\\
16  &  483870160  &  0.00575037  &  0.461823  &  0.00265565\\
17  &  1861604361  &  -0.238380  &  -0.527441  &  0.125732\\
18  &  7162191603  &  0.390435  &  0.414778  &  0.161944\\
19  &  27555258052  &  -0.410517  &  -0.161915  &  0.066469\\
20  &  106013953326  &  0.291840  &  -0.145677  &  0.0425145\\
21  &  407869825737  &  -0.0745174  &  0.404029  &  0.0301072\\
22  &  1569206595241  &  -0.167993  &  -0.525814  &  0.088333\\
23  &  6037243216260  &  0.353720  &  0.469867  &  0.166202\\
24  &  23227219260240  &  -0.419885  &  -0.255100  &  0.107113\\
25  &  89362594024741  &  0.344124  &  -0.0458947  &  0.0157935\\
26  &  343806683071203  &  -0.152045  &  0.331376  &  0.050384\\
27  &  1322735050548072  &  -0.0914277  &  -0.504849  &  0.0461571\\
28  &  5088987794882566  &  0.303996  &  0.507676  &  0.154332\\
\bottomrule
\end{tabular}
\caption{Approximating $\theta=\sqrt[3]{2}$ and $\theta^2$ using units of $\Q(\theta)$}
\label{tab:LWseqcbrt2}
\end{table}

This paper is a generalization of these results to include all pairs $(\alpha,\beta)$, where $\alpha,\beta\in K\subset \R$ and $K$ is a cubic field with only one real embedding.

\newpage

\section{Heuristics}
\label{sec:heuristics}
\input{disheuristics.tex}

\section{Acknowledgements}
\label{sec:ack}
\input{cubiclwack.tex}

%% file: disheuristics.tex
\newcommand{\figscale}{.825}

The following is not anything resembling a proof, but it illustrates how often we might expect to encounter numbers $n$ satisfying
\begin{equation*}
n\|n\alpha\|\|n\beta\|<\ep
\end{equation*}
or
\begin{equation*}
n\|n\alpha\|\|n\beta\|<\frac{C}{\log n}
\end{equation*}
for a given $\ep$ or $C$.

Suppose $1,\alpha,\beta$ are $\Q$-linearly independent, and consider the sequences $\{\lb n\alpha \rb\}$ and $\{\lb n\beta\rb\}$. Define for $T\geq 1$ the sets $U_\ep(T)$ and $V_C(T)$ by
\begin{gather*}
U_\ep(T)=\{n\leq T: n\|n\alpha\|\|n\beta\|<\ep\},\\
V_C(T)=\{n\leq T: n\log n\|n\alpha\|\|n\beta\|<C\}.
\end{gather*}
Supposing the terms act like independent uniformly distributed random variables, we estimate expected values for the sizes of $U_\ep(T)$ and $V_C(T)$.

If $X, Y$ are independent and uniformly distributed over $(-\frac{1}{2},\frac{1}{2})$ and if $0<r<\frac{1}{4}$, then
\begin{align*}
P(|X||Y|<r)
&=
4 \iint\limits_{\substack {0<u,v<1/2\\uv<r}}\dd u\dd v\\
&=4\int_0^{2r}\frac{1}{2}\dd u+4\int_{2r}^{1/2}\frac{r}{u}\dd u\\
&=4r-4r\log 4r.
\end{align*}
So if we assume that $\lb n\alpha \rb$ and $\lb n\beta\rb$ act independent and uniform, and if
$\frac{\ep}{n},\frac{C}{n\log n}<\frac{1}{4}$, then
\begin{gather*}
P(n\|n\alpha\|\|n \beta\|
<
\ep)
=
P(\|n\alpha\|\|n\beta\|<\tfrac{\ep}{n})
=\frac{4\ep}{n}-\frac{4\ep}{n}\log\left(\frac{4\ep}{n}\right),\\
P(n\|n\alpha\|\|n \beta\|
<
\tfrac{C}{\log n})
=
P(\|n\alpha\|\|n\beta\|<\tfrac{C}{n\log n})
=\frac{4C}{n\log n}-\frac{4C}{n\log n}\log\left(\frac{4C}{n\log n}\right).
\end{gather*}

\subsection{Estimating $|U_\ep(T)|$}
Assume $\ep<\frac{1}{4}$. Then $\frac{\ep}{n}<\frac{1}{4}$ for all $n$. We estimate $|U_\ep(T)|$ as an expected value:
\begin{align*}
E\left(|U_\ep(T)|\right)
&=
\sum_{n\leq T} \frac{4\ep}{n}-\frac{4\ep}{n}\log\left(\frac{4\ep}{n}\right).
\end{align*}
Since $x(1-\log x)$ is positive and decreasing for $x\geq 1$, this sum equals
\begin{align*}
\int_1^T \frac{4\ep}{x}-\frac{4\ep}{x}\log\left(\frac{4\ep}{x}\right)\dd x+R
\end{align*}
for some $R$ with
\begin{equation*}
|R|
<
\frac{4\ep}{1}-\frac{4\ep}{1}\log\left(\frac{4\ep}{1}\right)
=
4\ep(1-\log 4\ep)<1
\end{equation*}
(since $4\ep<1$). So
\begin{align*}
E\left(|U_\ep(T)|\right)
&=
\int_1^T \left(4\ep-4\ep\log\left(\frac{4\ep}{x}\right)\right)\frac{\dd x}{x}+R\\
&=
\int_0^{\log T} \left(4\ep-4\ep\log4\ep+4\ep u\right)\dd u +R\\
&=
2\ep(\log T)^2+4\ep(1-\log 4\ep)\log T+R.
\end{align*}
Define
\begin{equation*}
F_\ep(T)
=
2\ep(\log T)^2+4\ep(1-\log 4\ep)\log T.
\end{equation*}
In Figures~\ref{fig:Uapprox1}, \ref{fig:Uapprox2}, \ref{fig:Uapprox3}, and~\ref{fig:Uapprox4} we compare the graphs of $F_\ep(T)$ and $|U_\ep(T)|$ for the following examples.
\begin{enumerate}
\item $(\sqrt{2},\sqrt{3})$: Both have bounded partial quotients, and it is unknown whether Littlewood's conjecture is true for this pair.
\item $(e,\pi)$: We know that $e$ has unbounded partial quotients, so we know that Littlewood's conjecture is true for this pair.
\item $(\sqrt[3]{2},\sqrt[3]{4})$: We do not know whether the partial quotients are bounded or unbounded. We {\em do} know that Littlewood's conjecture is true for this pair.
\item $(\frac{1+\sqrt{5}}{2},\frac{\sqrt{65}}{5})$: Like the pair $(\sqrt{2},\sqrt{3})$, these have small partial quotients. The continued fractions are $\frac{1+\sqrt{5}}{2}=[1;\overline{1}]$ and $\frac{\sqrt{65}}{5}=[1;\overline{1,1,1,1,2}]$. Since both numbers are ``badly approximable'' by rationals, it is worth considering how well they can be simultaneously approximated.
\end{enumerate}

\begin{figure}[!hptb]
    \centering
    \includegraphics[width=\figscale\textwidth]{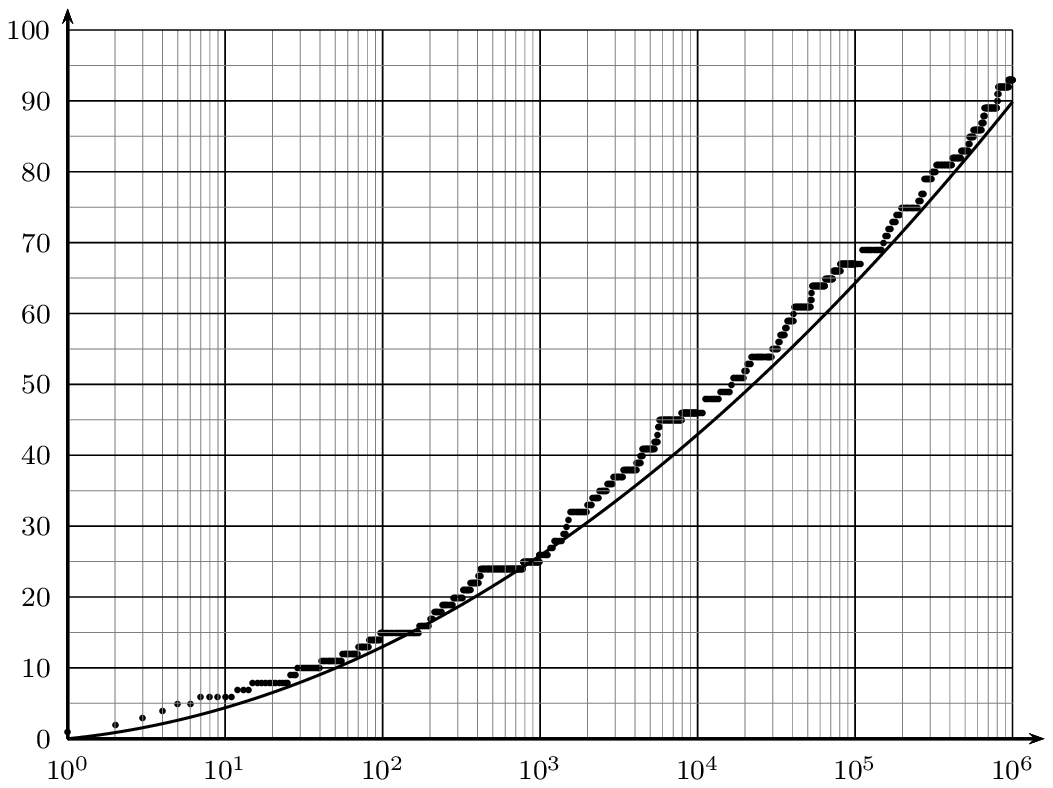}
    \caption{
    $|U_\ep(n)|$ and $F_\ep(n)$ for the pair $(\sqrt{2},\sqrt{3})$ (with $\ep=.2$).
    }
    \label{fig:Uapprox1}
    \bigskip
    \bigskip
    \includegraphics[width=\figscale\textwidth]{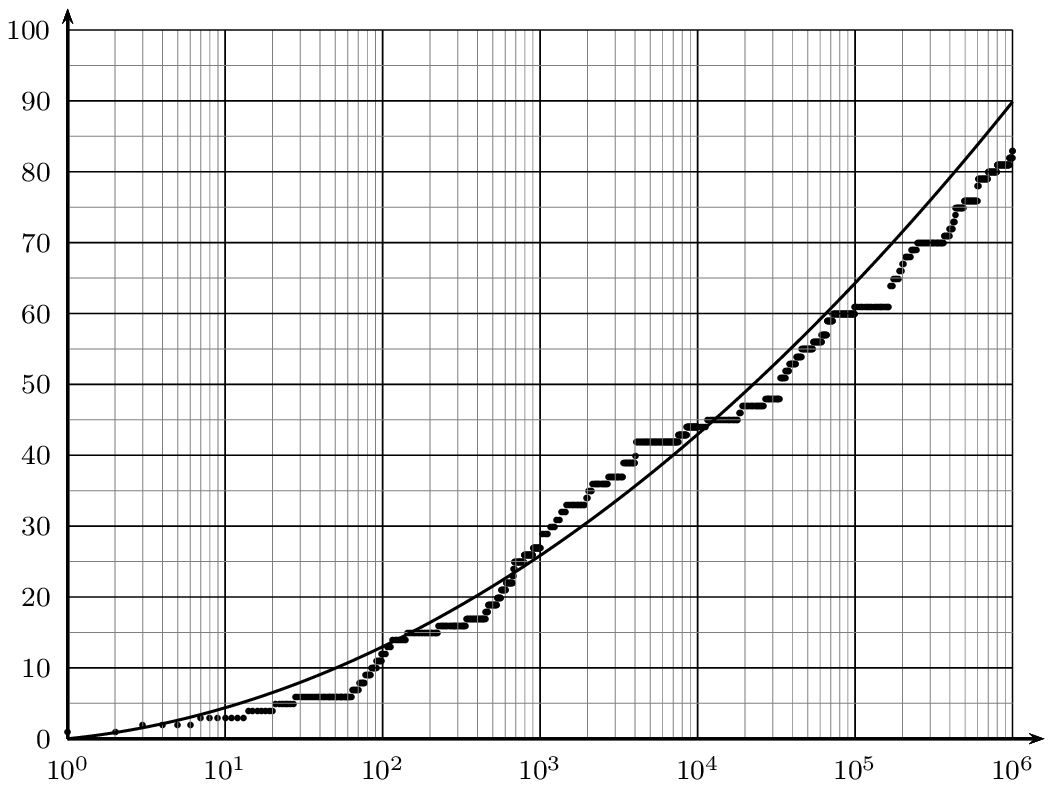}
    \caption{
    $|U_\ep(n)|$ and $F_\ep(n)$ for the pair $(e,\pi)$ (with $\ep=.2$).
    }
    \label{fig:Uapprox2}
\end{figure}

\begin{figure}[!hptb]
    \centering
    \includegraphics[width=\figscale\textwidth]{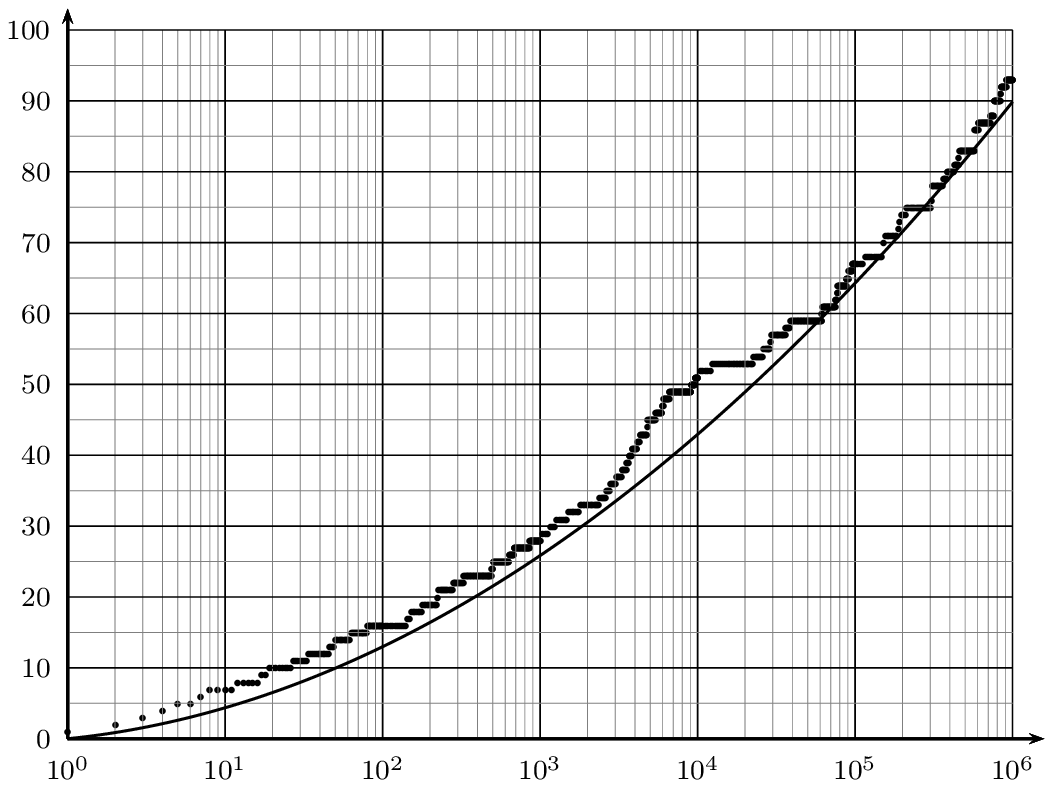}
    \caption{
    $|U_\ep(n)|$ and $F_\ep(n)$ for the pair $(\sqrt[3]{2},\sqrt[3]{4})$ (with $\ep=.2$).
    }
    \label{fig:Uapprox3}
    \bigskip
    \bigskip
    \includegraphics[width=\figscale\textwidth]{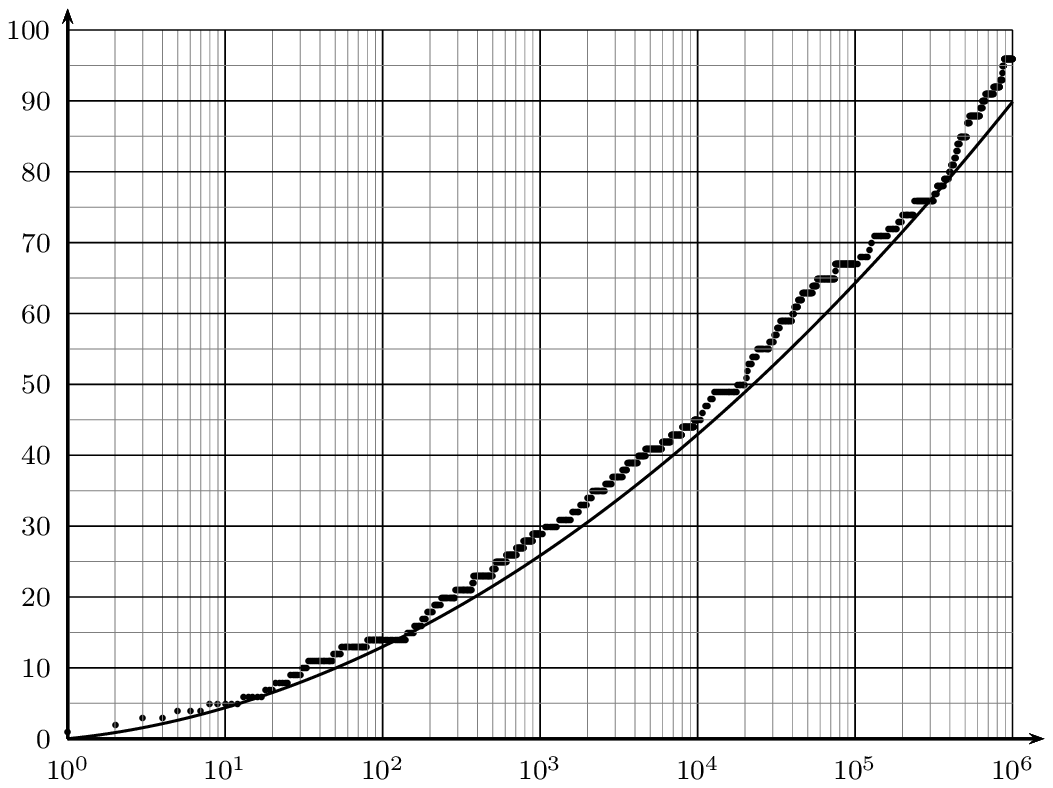}
    \caption{
    $|U_\ep(n)|$ and $F_\ep(n)$ for the pair $(\frac{1+\sqrt{5}}{2},\frac{\sqrt{65}}{5})$ (with $\ep=.2$).
    }
    \label{fig:Uapprox4}
\end{figure}

\subsection{Estimating $|V_C(T)|$}
Assume $C>0$, and let $m_C$ be the first integer such that $m_C\log m_C\geq 4C$. (Also, assume $T> m_C$.) For all $k<m_C$
\begin{gather*}
k \log k<4C,\\
\frac{1}{4}<\frac{C}{k \log k},\\
\|k\alpha\|\|k\beta\|\leq \frac{1}{4}<\frac{C}{k \log k}.
\end{gather*}
That is, $1,\dots,m_C-1$ are all in $V_C(T)$ (so $|V_C(m_C-1)|=m_C-1$). Since $\frac{C}{n\log n}<\frac{1}{4}$ for $m_C\leq n\leq T$, we estimate $|V_C(T)|$ as
\begin{align*}
E\left(|V_C(T)|\right)
&=
|V_C(m_C-1)|+\sum_{n=m_C}^T \frac{4C}{n\log n}-\frac{4C}{n\log n}\log\left(\frac{4C}{n\log n}\right)\\
&=
(m_C-1)+\sum_{n=m_C}^T \frac{4C}{n\log n}-\frac{4C}{n\log n}\log\left(\frac{4C}{n\log n}\right).
\end{align*}
Now
\begin{equation*}
\frac{4C}{x\log x}\left(1-\log\left(\frac{4C}{x\log x}\right)\right)
\end{equation*}
is positive and decreasing for $x\geq m_C$. (To see this, note that for $x\geq m_C$, $\frac{4C}{x\log x}$ is in $(0,1]$ and is decreasing.) So
\begin{align*}
\sum_{n=m_C}^T \frac{4C}{n\log n}-\frac{4C}{n\log n}\log\left(\frac{4C}{n\log n}\right)
=
\int_{m_C}^T \frac{4C}{x\log x}-\frac{4C}{x\log x}\log\left(\frac{4C}{x\log x}\right) \dd x + R,
\end{align*}
where $|R|<\frac{4C}{m_C\log m_C}(1-\log\frac{4C}{m_C\log m_C})<1$. If we let $I$ denote the integral, then
\begin{align*}
I
&=
\int_{m_C}^T \frac{4C}{\log x}-\frac{4C}{\log x}\log\left(\frac{4C}{x\log x}\right) \frac{\dd x}{x}\\
&=
4C
\int_{\log m_C}^{\log T}
\left(\frac{1-\log 4C}{x}+\frac{\log x}{x}+1\right) \dd x\\
&=
4C\left(
(1-\log 4C)\log x+\frac{(\log x)^2}{2}+x
\right)\biggr\rvert_{\log m_C}^{\log T}\\
&=
S_C(T)-S_C(m_C),
\end{align*}
where
\begin{align*}
S_C(x)=4C \log x+2C(\log\log x)^2+4C(1-\log 4C)\log\log x.
\end{align*}
%
Then
\begin{align*}
E\left(|V_C(T)|\right)
&=
S_C(T)-S_C(m_C)+(m_C-1)+R,
\end{align*}
and $|R|<1$.
Define
\begin{equation*}
G_C(T)=S_C(T)-S_C(m_C)+(m_C-1).
\end{equation*}
In Figures~\ref{fig:Vapprox1}, \ref{fig:Vapprox2}, \ref{fig:Vapprox3}, and~\ref{fig:Vapprox4} we compare the graphs of $G_C(T)$ and $|V_C(T)|$ for the same pairs as before.

\newpage

\newpage

\begin{figure}[!hptb]
    \centering
    \includegraphics[width=\figscale\textwidth]{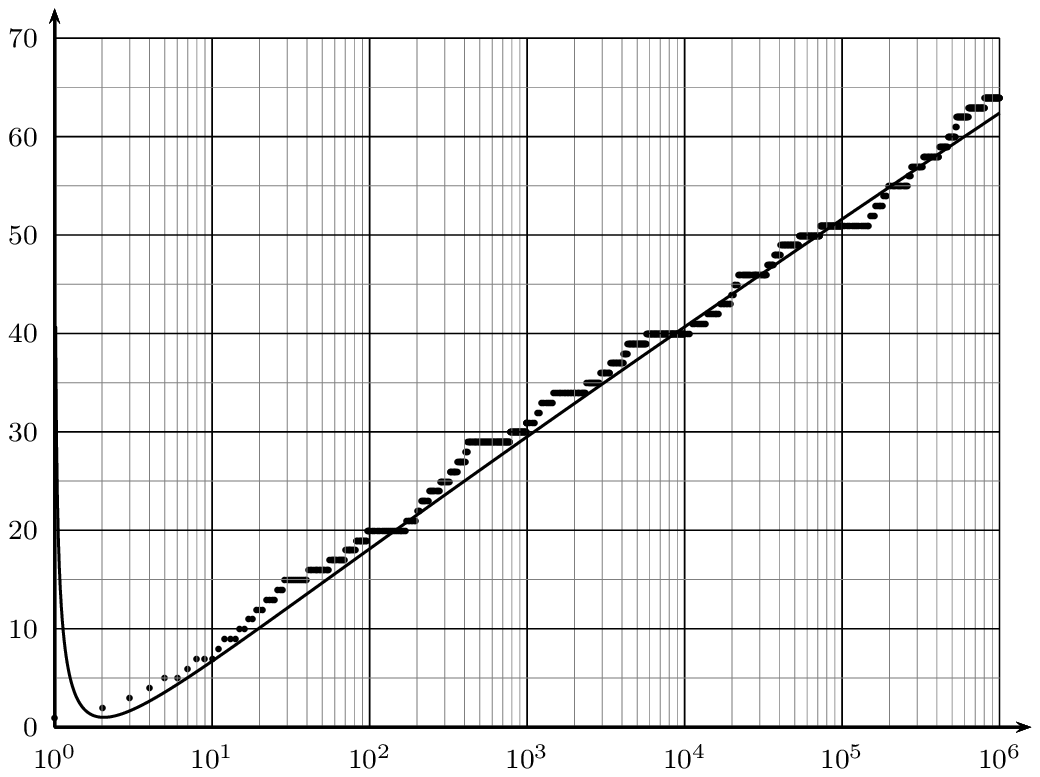}
    \caption{
    $|V_C(n)|$ and $G_C(n)$ for the pair $(\sqrt{2},\sqrt{3})$ (with $C=1$).
    }
    \label{fig:Vapprox1}
    \bigskip
    \includegraphics[width=\figscale\textwidth]{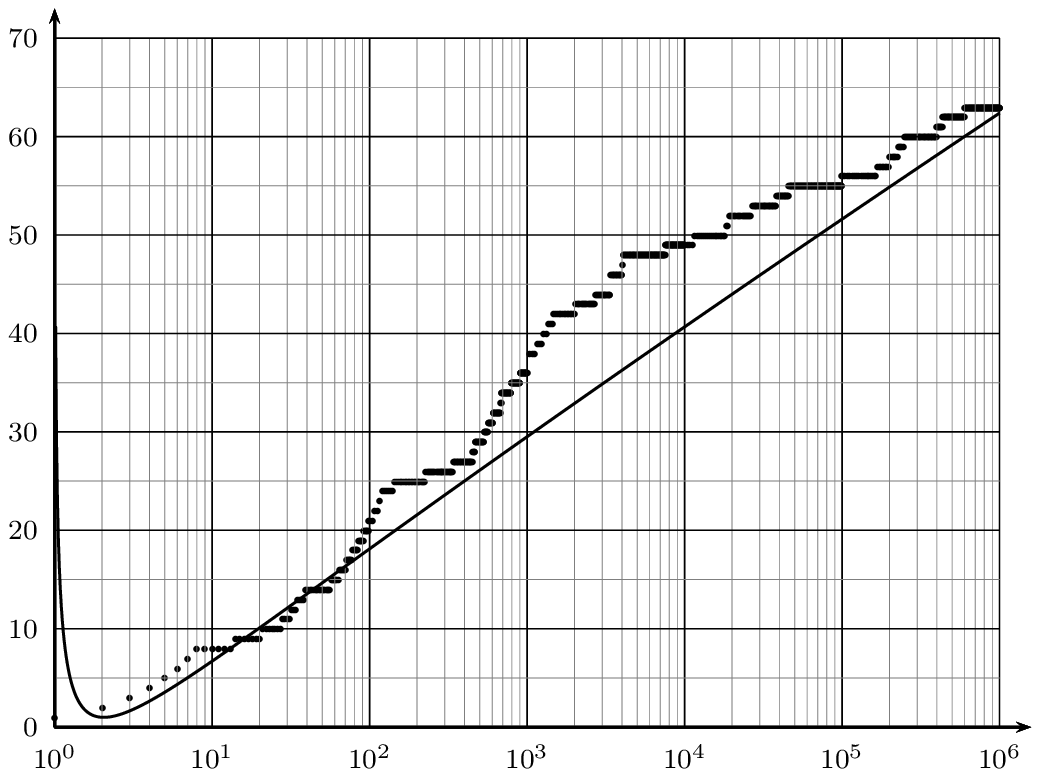}
    \caption{
    $|V_C(n)|$ and $G_C(n)$ for the pair $(e,\pi)$ (with $C=1$).
    }
    \label{fig:Vapprox2}
\end{figure}

\begin{figure}[!hptb]
    \centering
    \includegraphics[width=\figscale\textwidth]{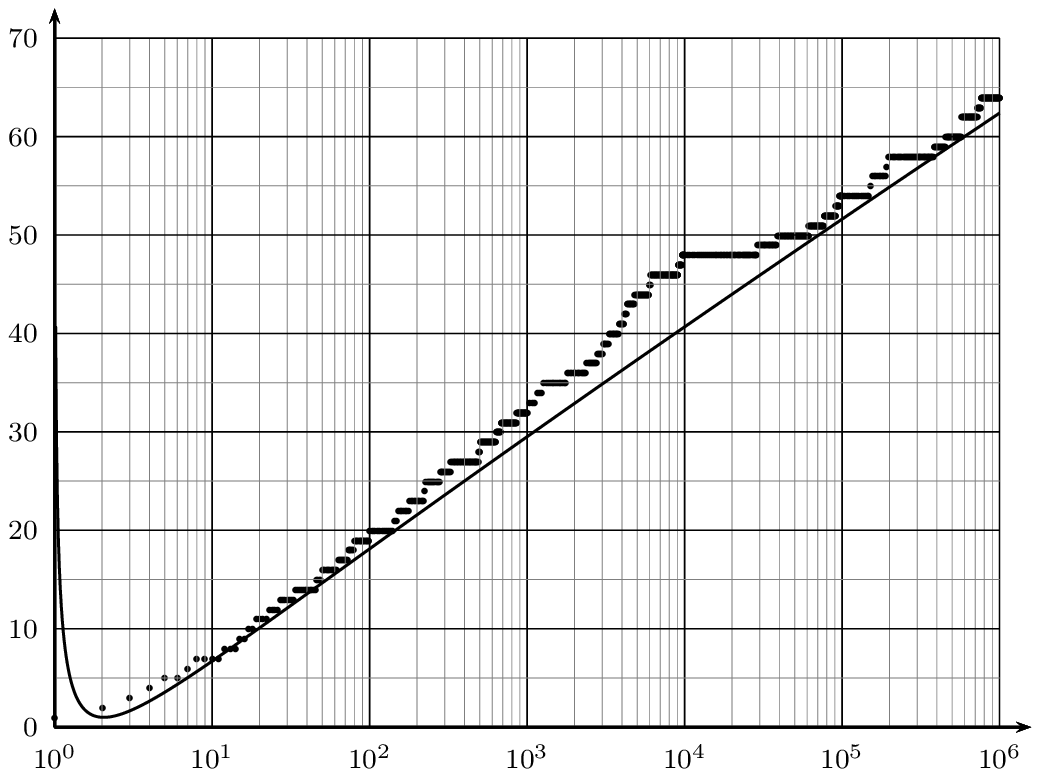}
    \caption{
    $|V_C(n)|$ and $G_C(n)$ for the pair $(\sqrt[3]{2},\sqrt[3]{4})$ (with $C=1$).
    }
    \label{fig:Vapprox3}
    \bigskip
    \includegraphics[width=\figscale\textwidth]{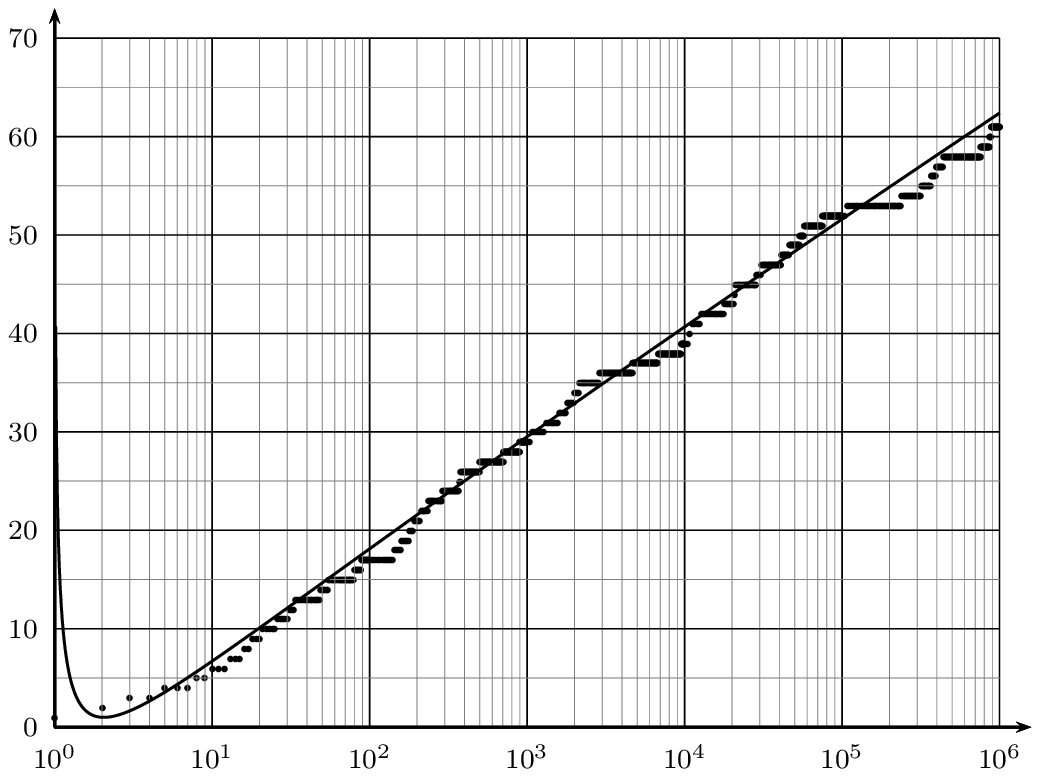}
    \caption{
    $|V_C(n)|$ and $G_C(n)$ for the pair $(\frac{1+\sqrt{5}}{2},\frac{\sqrt{65}}{5})$ (with $C=1$).
    }
    \label{fig:Vapprox4}
\end{figure}

%% file: cubiclwack.tex
This paper is adapted from my Ph.D. dissertation \cite{bib:Hi}, submitted at the University of Arizona in December 2014. I thank my advisor Marek Rychlik, as well as Kirti Joshi, Dinesh Thakur, and Dan Madden, for reading several drafts of my dissertation and providing countless valuable suggestions, and for encouraging me to adapt my dissertation into this paper. I would also like to thank John Brillhart and Yann Bugeaud for pointing me toward some results I would have otherwise missed.

I am grateful for the financial support I received (\mbox{VIGRE} fellowships, teaching assistantships, and a Thesis/Dissertation Tuition Scholarship) while working toward my Ph.D. at the University of Arizona.

%% file: cubiclwmainresults.tex

In this chapter, we prove the following theorems:\\

{\bf Theorem 1. } {\em
Suppose $\theta\in\R$ is the only real root of an irreducible polynomial of the form $x^3-px-q\in\Z[x]$. Then we can construct a Peck sequence for the pair $(\theta,\theta^2)$.
}\\

{\bf Theorem 2. } {\em Suppose $\alpha\in\R$ is the only real root of an irreducible cubic polynomial in $\Q[x]$. Then we can construct a Peck sequence for the pair $(\alpha,\alpha^2)$.
}\\

{\bf Theorem 3. }{\em
Suppose $\alpha,\beta\in K$, where $K\subseteq \R$ is a cubic field with only one real embedding.
Then we can construct a Peck sequence for the pair $(\alpha,\beta)$.
}\\

The following immediately follows from Theorem~3:\\

{\bf Corollary. (Special Case of Littlewood's Conjecture) }{\em If $\alpha, \beta$ are as in Theorem~3, then we can construct a sequence $\{\psi_n\}$ of positive integers such that
\begin{equation*}
\lim_{n\to\infty}\psi_n\|\psi_n\alpha\|\|\psi_n\beta\|=0.
\end{equation*}
}

%% file: cubiclwlemmas.tex
\subsection{Definitions and Notation used in Lemmas}
Suppose $\theta\in\R$ is the only real root of the irreducible cubic polynomial
\[
x^3-px-q\in\Z[x].
\]
Let $K$ denote $\Q(\theta)$, and let $\sigma_1,\sigma_2,\sigma_3$ denote the embeddings of $K$ in $\C$. (We assume $\sigma_1$ is the real embedding. We will make an assumption about the choice of $\sigma_2$ and $\sigma_3$ in Remark~\ref{rem:thetaprime}.) For $\zeta\in K$, let $\zeta_i$ denote $\sigma_i (\zeta)$ for $i=1,2,3$ (so $\zeta_1=\zeta$).
We use two standard results from number theory (see \cite{bib:Ma}):
\begin{itemize}
\item[(i)] There are positive integers $d$ for which $\O_K\subseteq \frac{1}{d}\Z[\theta]$. In particular, \begin{equation*}
    |\disc(1,\theta,\theta^2)|=|4p^3-27q^2|
    \end{equation*}
    would work (see Theorem 9 of Marcus). In practice, we can find this by computing (in PARI/GP, for instance) an integral basis.
\item[(ii)] If $K$ is a cubic field with only one real embedding, then the group of units of $\O_K$ has rank~1. (Dirichlet's units theorem)
\end{itemize}
Let $\lambda\in\O_K$ be a unit of infinite order (i.e., $\lambda^n\neq 1$ for $n>0$).
In particular $\lambda\neq\pm 1$, so we can assume WLOG that $\lambda>1$. (If not, then replace $\lambda$ with one of $\pm\lambda^{\pm 1}$.)\\

Define the sequences $\{a_n\}$, $\{b_n\}$, $\{c_n\}$ of rational numbers by
\begin{equation*}
a_n+b_n\theta+c_n\theta^2=\lambda^n.
\end{equation*}
Since $1,\theta,\theta^2$ is a $\Q$-basis of $\Q(\theta)$, each $a_n$, $b_n$, $c_n$ is well-defined. Since
$\lambda\in\O_K\subseteq\frac{1}{d}\Z[\theta],$
each $a_n,b_n,c_n$ is in $\frac{1}{d}\Z$.
Also note that $\lambda_i^n=\sigma_i(\lambda^n)=a_n+b_n\theta_i+c_n\theta_i^2$.\\

We define auxiliary sequences $\{X_n\}$, $\{Y_n\}$, and $\{Z_n\}$ by:
\begin{gather*}
X_n=a_n+p c_n-b_n\theta\\
Y_n=a_n+p c_n -c_n\theta^2\\
Z_n=b_n\theta-c_n\theta^2
\end{gather*}
(where $p\in\Q$ is as in $x^3-px-q$).\\

Finally, define the sequence $\{k_n\}$ by $k_n=d c_n$.\\

%

\begin{rem}\label{rem:kncomp}
We can easily compute $a_n,b_n,c_n$ for an arbitrary $n$ once we know $\lambda=a+b\theta+c\theta^2$. One way is to work in $\Q[x]/(x^3-px-q)$ (making the identification $\theta\leftrightarrow x$), and then (using repeated squaring)
\begin{equation*}
a_n+b_n\theta+c_n\theta^2
=(a+b\theta+c\theta^2)^n
\longleftrightarrow (a+bx+cx^2)^n \mod x^3-px-q.
\end{equation*}
Another way is to work in the matrix ring $\Q[T]$, where
\begin{equation*}
T=
\begin{pmatrix}
0 & 0 & q\\
1 & 0 & p\\
0 & 1 & 0\\
\end{pmatrix}
\end{equation*}
is the companion matrix of $x^3-px-q$. In this ring we identify $\theta$ and $T$, and so
\begin{equation*}
a_n+b_n\theta+c_n\theta^2
\longleftrightarrow
a_nI+b_nT+c_nT^2
=(aI+bT+cT^2)^n.
\end{equation*}
In particular,
\begin{equation*}
\begin{pmatrix}
a_n\\
b_n\\
c_n
\end{pmatrix}
=
\begin{pmatrix}
a & qc & qb\\
b & a+pc & pb+qc\\
c & b & a +pc
\end{pmatrix}^n
\begin{pmatrix}
1\\
0\\
0
\end{pmatrix}.
\end{equation*}
Again, this exponentiation can be done by repeated squaring.\\

We used this matrix method in Mathematica~8.0 to compute our examples.
In the Appendix, we list some specific commands and functions we used.
\end{rem}

\subsection{Lemmas}

\begin{lem}\label{eq:theta23}
\[
\theta_2,\theta_3=\frac{-\theta\pm\sqrt{4p-3\theta^2}}{2}.
\]
\end{lem}

\begin{proof}
We have two expressions for the minimal polynomial of $\theta$:
\begin{equation*}
x^3-px-q
\end{equation*}
and
\begin{equation*}
(x-\theta_1)(x-\theta_2)(x-\theta_3)=
x^3-(\theta_1+\theta_2+\theta_3)x^2+(\theta_1\theta_2
+\theta_1\theta_3+\theta_2\theta_3)x-\theta_1\theta_2\theta_3.
\end{equation*}
Comparing coefficients, we see that $\theta_1+\theta_2+\theta_3=0$ and $\theta_1\theta_2\theta_3=q$. Then
\begin{equation*}
(x-\theta_2)(x-\theta_3)=x^2-(\theta_2+\theta_3)x+\theta_2\theta_3=x^2-(-\theta) x+\frac{q}{\theta}.
\end{equation*}
Since $\theta^3=p\theta+q$, we have $\theta^2-p=\frac{q}{\theta}$. So
\begin{equation*}
(x-\theta_2)(x-\theta_3)=x^2+\theta x+(\theta^2-p).
\end{equation*}
Then we use the quadratic formula and the fact that $\theta_2$ and $\theta_3$ are roots.
\end{proof}

\begin{rem}\label{rem:thetaprime}

WLOG we assume that $\sigma_2$ and $\sigma_3$ are such that $\theta_2=\frac{-\theta+\sqrt{4p-3\theta^2}}{2}$ and $\theta_3=\frac{-\theta-\sqrt{4p-3\theta^2}}{2}$.

\end{rem}

\begin{lem}
We have the following inequalities:
\begin{align}
3\theta^2-4p&>0\label{ineq:disc}\\
\theta^2-p&>0\label{ineq:theta2}\\
3\theta^2-p&>0\label{ineq:theta3}
\end{align}
\end{lem}

\begin{proof}
Because $(x-\theta_2)(x-\theta_3)$ has complex roots, its discriminant $4p-3\theta^2$ is negative. That is, $p<\frac{3\theta^2}{4}$. So we have
\begin{equation*}
p<\frac{3\theta^2}{4}<\theta^2<3\theta^2,
\end{equation*}
and therefore
$\theta^2-p>0$ and $3\theta^2-p>0$ as well.
\end{proof}

The following identities generalize the factorization of $N_{K/\Q}$ that we had in \eqref{eq:normfac}.
\begin{lem}\label{lem:ellipsoididentity}
If $\zeta=x+y\theta+z\theta^2\in K^\times$, then
\begin{equation}\label{eq:ident1}
(2(x+p z)-y\theta-z\theta^2)^2+(3\theta^2 -4p)(y-z\theta)^2=\frac{4 N(\zeta)}{\zeta}
\end{equation}
and
\begin{equation}\label{eq:ident2}
(x+pz-y\theta)^2+(x+pz-z\theta^2)^2+(1-2p/\theta^2)(y\theta-z\theta^2)^2=\frac{2 N(\zeta)}{\zeta}.
\end{equation}
\end{lem}

\begin{proof}
We start with $\zeta_1\zeta_2\zeta_3=N(\zeta)$, or $\zeta_2\zeta_3=\frac{N(\zeta)}{\zeta}$. We wish to find $\zeta_2$ and $\zeta_3$ in terms of $\theta$. Now (using Lemma \ref{eq:theta23})
\begin{align*}
\left(\frac{-\theta\pm\sqrt{4p-3\theta^2}}{2}\right)^2
&=\frac{\theta^2\mp 2\theta\sqrt{4p-3\theta^2}+4p-3\theta^2}{4}
=p+\frac{-\theta^2\mp\theta\sqrt{4p-3\theta^2}}{2},
\end{align*}
so
\begin{align*}
\zeta_2
&=x+y\theta_2+z\theta_2^2\\
&=x+y\left(\frac{-\theta+\sqrt{4p-3\theta^2}}{2}\right)
+z\left(p+\frac{-\theta^2-\theta\sqrt{4p-3\theta^2}}{2}\right)\\
&=\left((x+p z)-\frac{y\theta}{2}-\frac{z\theta^2}{2}\right)
+\frac{\sqrt{4p-3\theta^2}}{2}(y-z\theta)\\
&=\frac{1}{2}\left(2(x+p z)-y\theta-z\theta^2\right)+\frac{\sqrt{4p-3\theta^2}}{2}(y-z\theta).
\end{align*}
Similarly we find that
\[
\zeta_3=\frac{1}{2}(2(x+p z)-y\theta-z\theta^2)-\frac{\sqrt{4p-3\theta^2}}{2}(y-z\theta).
\]
Then
\[
\frac{N(\zeta)}{\zeta}=\zeta_2\zeta_3
=\frac{1}{4}\left(2(x+p z)-y\theta-z\theta^2\right)^2
-\frac{4p-3\theta^2}{4}(y-z\theta)^2,
\]
or
\[
\frac{4 N(\zeta)}{\zeta}
=\left(2(x+p z)-y\theta-z\theta^2\right)^2+(3\theta^2-4p)(y-z\theta)^2.
\]
To get \eqref{eq:ident2} we use the identity
\[
(U+V)^2+3(U-V)^2=2U^2+2V^2+2(U-V)^2
\]
with $U=x+pz-z\theta^2$ and $V=x+pz-y\theta$:
\begin{align*}
\frac{4N(\zeta)}{\zeta}
&=\left(2(x+p z)-y\theta-z\theta^2\right)^2+3(y\theta-z\theta^2)^2-\frac{4p}{\theta^2}(y\theta-z\theta^2)^2\\
&=(U+V)^2+3(U-V)^2-\frac{4p}{\theta^2}(U-V)^2\\
&=2U^2+2V^2+2(U-V)^2-\frac{4p}{\theta^2}(U-V)^2,
\end{align*}
or
\begin{align*}
\frac{2N(\zeta)}{\zeta}
&=U^2+V^2+(1-2p/\theta^2)(U-V)^2\\
&=(x+pz-z\theta^2)^2+(x+pz-y\theta)^2+(1-2p/\theta^2)(y\theta-z\theta^2)^2.
\end{align*}
\end{proof}

\begin{defn}\label{defn:constants}
We define the constants $C_1$, $C_2$, and $C_3$ by:
\begin{gather*}
C_1=\max\{\sqrt{2},\frac{\sqrt{2}|\theta|}{\sqrt{3\theta^2-4p}}\}\\
C_2=\frac{\sqrt{d}}{\sqrt{3\theta^2-p}}\\
C_3=\max\{1,\frac{3C_2}{2}\}.
\end{gather*}
We will use these constants throughout the lemmas, as well as in the proofs of the theorems and the constructions in Chapter~3.
\end{defn}

\begin{lem}\label{lem:C1altdefn}
$C_1$ also satisfies
\begin{equation*}
C_1=
\begin{cases}
\frac{\sqrt{2}|\theta|}{\sqrt{3\theta^2-4p}},& \theta^2-2p<0\\
\sqrt{2}, & \theta^2-2p>0.
\end{cases}
\end{equation*}
\end{lem}

\begin{proof}
Note that $C_1=\sqrt{2}\cdot\max\{1,\frac{|\theta|}{\sqrt{3\theta^2-4p}}\}$ and that
\begin{align*}
1<\frac{|\theta|}{\sqrt{3\theta^2-4p}}
&\Leftrightarrow
    3\theta^2-4p<\theta^2\\
&\Leftrightarrow
    2\theta^2-4p<0\\
&\Leftrightarrow
    \theta^2-2p<0.
\end{align*}
\end{proof}

\begin{lem}\label{lem:bounds} For all positive integers $n$,
\begin{align}
|a_n+p c_n-b_n\theta|
    =|X_n|
    &<\dfrac{C_1}{\lambda^{n/2}}\label{ineq:Xn}\\
|a_n+p c_n-c_n\theta^2|
    =|Y_n|
    &<\frac{C_1}{\lambda^{n/2}}\label{ineq:Yn}\\
|b_n\theta-c_n\theta^2|
    =|Z_n|
    &<\frac{2|\theta|}{\sqrt{3\theta^2-4p}}
    \cdot
    \frac{1}{\lambda^{n/2}}
    \leq
    \frac{\sqrt{2}C_1}{\lambda^{n/2}}\label{ineq:Zn}\\
|\lambda^n-(3\theta^2-p)c_n|
    &<\frac{(1+\sqrt{2})C_1}{\lambda^{n/2}}\label{ineq:lam}
\end{align}
\end{lem}

\begin{proof}
(Third inequality) With $\zeta=\lambda^n$ in \eqref{eq:ident1},
\begin{equation*}
(X_n+Y_n)^2+(3\theta^2-4p)\left(\frac{Z_n}{\theta}\right)^2
    =\frac{4N(\lambda^n)}{\lambda^n}=\frac{4}{\lambda^n},
\end{equation*}
\begin{equation*}
(3\theta^2-4p)\left(\frac{Z_n}{\theta}\right)^2
    <\frac{4}{\lambda^n}.
\end{equation*}
Then since $3\theta^2-4p>0$ (from \eqref{ineq:disc}),
\begin{align*}
Z_n^2
    &<\frac{4\theta^2}{3\theta^2-4p}\cdot\frac{1}{\lambda^n},\\
|Z_n|
    &<\frac{2|\theta|}
        {\sqrt{3\theta^2-4p}}
        \cdot\frac{1}{\lambda^{n/2}}
    \leq \frac{\sqrt{2}C_1}{\lambda^{n/2}}.
\end{align*}
(First and second inequalities) By \eqref{eq:ident2},
\begin{equation}\label{eq:ellipsoid}
X_n^2+Y_n^2+\left(\frac{\theta^2-2p}{\theta^2}\right)Z_n^2=\frac{2}{\lambda^n}.
\end{equation}
\\
(Case 1: $\theta^2-2p<0$.) From \eqref{eq:ellipsoid} we have
\begin{equation*}
X_n^2+Y_n^2
=\frac{2p-\theta^2}{\theta^2} Z_n^2+\frac{2}{\lambda^n}.
\end{equation*}
Since $\frac{2p-\theta^2}{\theta^2}Z_n^2>0$, and since
$Z_n^2
    <\frac{4\theta^2}{3\theta^2-4p}\cdot\frac{1}{\lambda^n}$,
\begin{align*}
X_n^2+Y_n^2
&<\frac{2p-\theta^2}{\theta^2}\left( \frac{4\theta^2}{3\theta^2-4p}\cdot\frac{1}{\lambda^n}\right)+\frac{2}{\lambda^n}\\
&=\left(\frac{4(2p-\theta^2)}{(3\theta^2-4p)}+2\right)\frac{1}{\lambda^n}\\
&=\frac{4(2p-\theta^2)+2(3\theta^2-4p)}{3\theta^2-4p}\cdot\frac{1}{\lambda^n}\\
&=\frac{2\theta^2}{3\theta^2-4p}\cdot\frac{1}{\lambda^n}.
\end{align*}
So
\begin{equation*}
|X_n|,|Y_n|
<
\frac{\sqrt{2}|\theta|}{\sqrt{3\theta^2-4p}}\cdot\frac{1}{\lambda^{n/2}}.
\end{equation*}
By Lemma \ref{lem:C1altdefn}, the RHS of this inequality is $\frac{C_1}{\lambda^{n/2}}$.
\\
(Case 2: $\theta^2-2p>0$.) Suppose $\theta^2-2p>0$. Then
\begin{equation*}
X_n^2+Y_n^2
\leq
X_n^2+Y_n^2+\left(\frac{\theta^2-2p}{\theta^2}\right)Z_n^2
=\frac{2}{\lambda^n}.
\end{equation*}
We claim that equality cannot hold with our choice of $\lambda$. Note that we would have equality iff $b_n\theta-c_n\theta^2=Z_n=0$. Since $b_n,c_n\in\Q$ and $\theta\notin\Q$, this would mean $b_n=c_n=0$, so $\lambda^n=a_n$. Then
\begin{equation*}
1=N_{K/\Q}(\lambda)=N_{K/\Q}(a_n)=\sigma_1(a_n)\sigma_2(a_n)\sigma_3(a_n)=a_n^3=(\lambda^n)^3.
\end{equation*}
But $\lambda>1$ (by assumption), so $\lambda^{3n}>1$. Therefore we have the strict inequality
\begin{equation*}
X_n^2+Y_n^2<\frac{2}{\lambda^n},
\end{equation*}
and so
\begin{equation*}
|X_n|,|Y_n|
<
\frac{\sqrt{2}}{\lambda^{n/2}}.
\end{equation*}
By Lemma \ref{lem:C1altdefn}, the RHS of this inequality equals $\frac{C_1}{\lambda^{n/2}}$.\\

(Fourth inequality) Since
\begin{align*}
\lambda^n-(3\theta^2-p)c_n
&=a_n+b_n\theta+c_n\theta^2-3 c_n\theta^2+pc_n\\
&=(a_n+pc_n-c_n\theta^2)+(b_n\theta-c_n\theta^2)\\
&=Y_n+Z_n,
\end{align*}
we have
\begin{equation*}
|\lambda^n-(3\theta^2-p)c_n|
\leq|Y_n|+|Z_n|
=(1+\sqrt{2})\frac{C_1}{\lambda^{n/2}}.
\end{equation*}
%
\end{proof}

{\bf Recall: } We are using the notation $\langle x\rangle$ to denote $x-\lfloor x+\frac{1}{2}\rfloor$, the signed distance from $x$ to the nearest integer (so $\|x\|=|\langle x\rangle|$).

\begin{lem}\label{lem:normkntheta}
\begin{itemize}
\item[]
\item[(a)] If $\lambda^{n/2}>\dfrac{2\sqrt{2}dC_1}{|\theta|}$, then
$\langle k_n\theta\rangle
=-\dfrac{d Z_n}{\theta}$.
\item[(b)] If $\lambda^{n/2}>2 dC_1$, then
$\langle k_n\theta^2\rangle
=-d Y_n$.
\end{itemize}
\end{lem}

\begin{proof}
(a) Suppose $\lambda^{n/2}>\dfrac{2\sqrt{2}dC_1}{|\theta|}$.  Then $\dfrac{\sqrt{2}C_1}{\lambda^{n/2}}<\dfrac{|\theta|}{2d}$, so by \eqref{ineq:Zn}
\begin{equation*}
|db_n-dc_n\theta|
=\left|\dfrac{d Z_n}{\theta}\right|
<\frac{d}{|\theta|}
    \dfrac{\sqrt{2}C_1}{\lambda^{n/2}}
<\frac{1}{2}.
\end{equation*}
So since $db_n$ is an integer (by the choice of $d$), it is the closest integer to $dc_n\theta=k_n\theta$. That is,
\begin{equation*}
\langle k_n\theta\rangle=k_n\theta -db_n=-\dfrac{d Z_n}{\theta}.
\end{equation*}

(b) The proof is identical to (a). If $\lambda^{n/2}>2 dC_1$, then $\dfrac{C_1}{\lambda^{n/2}}<\dfrac{1}{2d}$, and so by \eqref{ineq:Yn}
\begin{equation*}
|da_n+p dc_n-dc_n\theta^2|=|dY_n|<\frac{dC_1}{\lambda^{n/2}}<\frac{1}{2}.
\end{equation*}
Now $da_n+pdc_n$ is an integer (since $p\in\Z$, and since $da_n,dc_n\in\Z$ by our choice of $d$), so it is the closest integer to $dc_n\theta^2=k_n\theta^2$. So
\begin{equation*}
\langle k_n\theta^2\rangle=k_n\theta^2-(da_n+pk_n)=-Y_n.
\end{equation*}

\end{proof}

\begin{lem}\label{lem:cnposinc}
\begin{itemize}
\item[]
\item[(a)] If $\lambda^{3n/2}>(1+\sqrt{2})C_1$, then $c_n>0$.
\item[(b)] If $\lambda^{3n/2}>\dfrac{(1+\sqrt{2})C_1}{\sqrt{\lambda}(\sqrt{\lambda}-1)}$, then $c_n$ is increasing.
\end{itemize}
\end{lem}

\begin{proof}
Put $\eta=(1+\sqrt{2})C_1$.\\

(a) Suppose $\lambda^{3n/2}>\eta$. Then $\lambda^n>\dfrac{\eta}{\lambda^{n/2}}$, or
\begin{equation*}
\lambda^n-\dfrac{\eta}{\lambda^{n/2}}>0.
\end{equation*}
By \eqref{ineq:lam}, we have
\begin{equation*}
\lambda^n-\frac{\eta}{\lambda^{n/2}}<(3\theta^2-p)c_n.
\end{equation*}
So $0<(3\theta^2-p)c_n$. And $3\theta^2-p>0$ (by \ref{ineq:theta3}), so $c_n>0$.\\

(b) We use \eqref{ineq:lam} again to get
\begin{align*}
(3\theta^2-p)c_{n+1}&>\lambda^{n+1}-\frac{\eta}{\lambda^{(n+1)/2}},\\
(3\theta^2-p)c_n&<\lambda^n+\frac{\eta}{\lambda^{n/2}}.
\end{align*}
Then
\begin{equation*}
(3\theta^2-p)(c_{n+1}-c_n)
>\lambda^n(\lambda-1)-\frac{\eta}{\lambda^{n/2}}\left(1+\frac{1}{\sqrt{\lambda}}\right),
\end{equation*}
which is positive if
\begin{equation*}
\lambda^n(\lambda -1)
>\frac{\eta}{\lambda^{n/2}}
    \left(1+\frac{1}{\sqrt{\lambda}}\right),
\end{equation*}
or
\begin{equation*}
\lambda^{3n/2}
>\frac{\eta(\sqrt{\lambda}+1)}{\sqrt{\lambda}
    (\lambda-1)}
=\frac{\eta}{\sqrt{\lambda}
    (\sqrt{\lambda}-1)}.
\end{equation*}
\end{proof}

\begin{lem}\label{ineq:sqrtlam} If $\lambda^{3n/2}>2(1+\sqrt{2})C_1$, then
\begin{equation}
|\lambda^{n/2}-\sqrt{3\theta^2-p}\cdot c_n^{1/2}|
<
\frac{\sqrt{2}C_1}{\lambda^n}.
\end{equation}
\end{lem}

\begin{proof}
Put $\eta=(\sqrt{2}+1)C_1$ as before, and suppose $\lambda^{3n/2}>2\eta$. By Lemma \ref{lem:cnposinc}, $c_n$ is positive. Now
\begin{equation*}
|\lambda^{n/2}-\sqrt{3\theta^2-p}\cdot c_n^{1/2}|\cdot
|\lambda^{n/2}+\sqrt{3\theta^2-p}\cdot c_n^{1/2}|
=|\lambda^n-(3\theta^2-p)c_n|<\frac{\eta}{\lambda^{n/2}}
\end{equation*}
(by \eqref{ineq:lam}), so
\begin{equation*}
|\lambda^{n/2}-\sqrt{3\theta^2-p}\cdot c_n^{1/2}|<\frac{\eta}{\lambda^{n/2}}
\cdot
\left(\lambda^{n/2}+\sqrt{3\theta^2-p}\cdot c_n^{1/2}\right)^{-1}.
\end{equation*}
We rewrite $\lambda^{3n/2}>2\eta$ as $\dfrac{\eta}{\lambda^{n/2}}<\dfrac{\lambda^n}{2}$, and then
(by \eqref{ineq:lam} again)
\begin{gather*}
(3\theta^2-p)c_n>\lambda^n-\frac{\eta}{\lambda^{n/2}}>\lambda^n-\frac{\lambda^n}{2}=\frac{\lambda^n}{2},\\
\sqrt{3\theta^2-p}\cdot c_n^{1/2}>\frac{\lambda^{n/2}}{\sqrt{2}},\\
\lambda^{n/2}+\sqrt{3\theta^2-p}\cdot c_n^{1/2}>\left(1+\frac{1}{\sqrt{2}}\right)\lambda^{n/2}.
\end{gather*}
So
\begin{align*}
|\lambda^{n/2}-\sqrt{3\theta^2-p}\cdot c_n^{1/2}|
    &<\frac{\eta}{\lambda^{n/2}}
    \cdot
    \left(\left(1+\frac{1}{\sqrt{2}}\right)\lambda^{n/2}\right)^{-1}\\
    &=\frac{\eta\sqrt{2}}{\sqrt{2}+1}\cdot\frac{1}{\lambda^n}\\
    &=\frac{(\sqrt{2}+1)C_1\sqrt{2}}{\sqrt{2}+1}
    \cdot
    \frac{1}{\lambda^n}\\
    &=\frac{\sqrt{2}C_1}{\lambda^n}.
\end{align*}

\end{proof}

\begin{cor}\label{ineq:knbound}
If $\lambda^{3n/2}>2(1+\sqrt{2})C_1$, then
\begin{equation*}
\frac{C_2}{2}\cdot\lambda^{n/2}
<
k_n^{1/2}
<
\frac{3 C_2}{2}\cdot\lambda^{n/2}.
\end{equation*}
\end{cor}

\begin{proof}
Suppose $\lambda^{3n/2}>2(1+\sqrt{2})C_1$. From Lemma \ref{ineq:sqrtlam}, we have
\begin{equation}\label{ineq:cnlowerbound}
\lambda^{n/2}-\frac{\sqrt{2}C_1}{\lambda^n}
<
\sqrt{3\theta^2-p}\cdot c_n^{1/2}
<
\lambda^{n/2}+\frac{\sqrt{2}C_1}{\lambda^n}.
\end{equation}
Now $\lambda^{3n/2}>2(1+\sqrt{2})C_1>2\sqrt{2}C_1$, so
\begin{equation*}
\frac{\sqrt{2}C_1}{\lambda^n}<\frac{\lambda^{n/2}}{2}.
\end{equation*}
Combining this with \eqref{ineq:cnlowerbound} gives us
\begin{equation*}
\frac{\lambda^{n/2}}{2}
<
\sqrt{3\theta^2-p}\cdot c_n^{1/2}
<
\frac{3\lambda^{n/2}}{2}.
\end{equation*}
Now multiply through by $C_2=\frac{\sqrt{d}}{\sqrt{3\theta^2-p}}$ (and recall that $k_n=dc_n$).
\end{proof}

\begin{lem}\label{lem:lambda2n}
For all positive integers $n$,
\begin{equation*}
\lambda_2^n=\dfrac{(X_n+Y_n)}{2}+\dfrac{\sqrt{4p-3\theta^2}}{2}\frac{Z_n}{\theta}.
\end{equation*}

\end{lem}

\begin{proof}
As in the proof of Lemma \ref{lem:ellipsoididentity}, we use $\theta_2=\frac{-\theta+\sqrt{4p-3\theta^2}}{2}$ (see Remark \ref{rem:thetaprime}) to find $\lambda_2^n$ in terms of $\theta$:
\begin{align*}
\lambda_2^n
&=a_n+b_n\theta_2+c_n\theta_2^2\\
&=a_n+b_n\left(\frac{-\theta+\sqrt{4p-3\theta^2}}{2}\right)+c_n\left(p+\frac{-\theta^2-\theta\sqrt{4p-3\theta^2}}{2}\right)\\
&=\left(a_n+pc_n-\frac{b_n\theta}{2}-\frac{c_n\theta^2}{2}\right)+\frac{\sqrt{4p-3\theta^2}}{2}(b_n-c_n\theta)\\
&=\frac{X_n+Y_n}{2}+\frac{\sqrt{4p-3\theta^2}}{2}\frac{Z_n}{\theta}.
\end{align*}
\end{proof}

\begin{cor}\label{cor:imth2n}
For all positive integers $n$,
\begin{equation*}
\Im(\lambda_2^n)
=\frac{\sqrt{3\theta^2-4p}}{2}
    \frac{Z_n}{\theta}
=\frac{\sqrt{3\theta^2-4p}}{2}
    (b_n-c_n\theta)
\end{equation*}
\end{cor}

\begin{proof}
This follows immediately from the previous lemma and the fact that (by Lemma \ref{ineq:disc}) $4p-3\theta^2<0$.
\end{proof}

\begin{lem}\label{lem:denseinS1}
The set $\{(\sqrt{\lambda}\lambda_2)^n:n\in\N\}$ is a dense subset of the unit circle.
\end{lem}

\begin{proof}
It is enough to show that $|\sqrt{\lambda}\lambda_2|=1$ and that $(\sqrt{\lambda}\lambda_2)^k\neq 1$ for $k\in\N$.
Since $\lambda_2$ and $\lambda_3$ are complex conjugates and $\lambda=\lambda_1$ is a unit,
\begin{equation*}
|\sqrt{\lambda}\lambda_2|^2=(\sqrt{\lambda}\lambda_2)(\sqrt{\lambda}\lambda_3)=\lambda_1 \lambda_2 \lambda_3=N(\lambda)=1.
\end{equation*}
Now suppose that $k$ is a positive integer and $(\sqrt{\lambda}\lambda_2)^k= 1$. Then $\Im((\sqrt{\lambda}\lambda_2)^{m})=0$ for every positive integer multiple $m$ of $k$.
Since
\begin{equation*}
\Im((\sqrt{\lambda}\lambda_2)^{m})=\lambda^{m/2}\Im (\lambda_2^{m})=\lambda^{m/2}\frac{\sqrt{3\theta^2-4p}}{2} \frac{Z_m}{\theta},
\end{equation*}
and since $\lambda^{m/2}$ and $\sqrt{3\theta^2-4p}$ are both nonzero, we would have
\begin{equation*}
b_{m}-c_{m}\theta=\frac{Z_{m}}{\theta}=0,
\end{equation*}
or $b_{m}=c_{m}\theta$. Since $b_{m},c_{m}\in\Q$ and $\theta\notin\Q$, this would imply $b_{m}=c_{m}=0$. In particular, we would have $c_m=0$ for every positive multiple $m$ of $k$. But by Lemma \ref{lem:cnposinc}, eventually $\{c_n\}$ is positive (and so there can be only finitely many $m$ such that $c_m=0$). Therefore $(\sqrt{\lambda}\lambda_2)^k\neq 1$.
\end{proof}

\begin{defn}\label{defn:phi}
Put
\begin{equation*}
\phi=\arctan
    \left(
    \frac{\sqrt{3\theta^2-4p}\cdot(b_1-c_1\theta)}{2(a_1+pc_1)-b_1\theta-c_1\theta^2}
    \right)
\end{equation*}
\end{defn}

\begin{lem}
$\sqrt{\lambda}\lambda_2=\pm e^{i\phi}$
\end{lem}

\begin{proof}
Since $\sqrt{\lambda}\lambda_2$ is on the unit circle (by \ref{lem:denseinS1}), it can be expressed as $e^{iw}$ for some $w\in[-\pi/2,3\pi/2)$. By Lemma \ref{lem:lambda2n}, we have
\begin{equation*}
e^{iw}
=\sqrt{\lambda}\lambda_2
=\sqrt{\lambda}\left(\dfrac{2(a_1+pc_1)-b_1\theta-c_1\theta^2}{2}
    +i \dfrac{\sqrt{3\theta^2-4p}}{2}(b_1-c_1\theta)\right),
\end{equation*}
so
\begin{equation*}
\tan w
=\frac{\Im \sqrt{\lambda}\lambda_2}{\Re \sqrt{\lambda}\lambda_2}
=\frac{\sqrt{3\theta^2-4p}\cdot(b_1-c_1\theta)}{2(a_1+pc_1)-b_1\theta-c_1\theta^2}
=\tan \phi.
\end{equation*}
Therefore either $w=\phi$ or $w=\phi+\pi$, so
$\sqrt{\lambda}\lambda_2=e^{iw}=\pm e^{i\phi}$.
\end{proof}

\begin{lem}\label{lem:convergents}
Let $\{P_n/Q_n\}$ be the sequence of convergents of $\phi/\pi$. For $n\geq 1$,
\begin{equation*}
\left|
\Im (\sqrt{\lambda}\lambda_2)^{Q_n}
\right|
<
\frac{\pi}{Q_{n+1}}.
\end{equation*}
\end{lem}

\begin{proof}
By the theory of continued fractions (see \cite{bib:Kh}, Theorem 9),
\begin{equation*}
\left|\frac{\phi}{\pi}-\frac{P_n}{Q_n}\right|
<\frac{1}{Q_n Q_{n+1}},
\end{equation*}
or
\begin{equation*}
\left|Q_n\phi -P_n\pi \right|
<\frac{\pi}{Q_{n+1}}.
\end{equation*}
Then
\begin{equation*}
|\Im (\sqrt{\lambda}\lambda_2)^{Q_n}|
=|\Im e^{i Q_n\phi}|
=|\sin Q_n\phi|
=|\sin(Q_n\phi- P_n\pi)|
\leq |Q_n\phi-P_n\pi |
<\frac{\pi}{Q_{n+1}}.
\end{equation*}
\end{proof}


\begin{lem}\label{lem:LWnorm}

\begin{itemize}
\item[]
\item[\normalfont (a)] If $\|x_1\|+\cdots+\|x_k\|\leq \frac{1}{2}$, then $\langle x_1+\cdots+x_k\rangle=\langle x_1\rangle+\cdots+\langle x_k\rangle$.
\item[\normalfont (b)] $\langle -x\rangle=-\langle x \rangle$
\item[\normalfont (c)] If $n\in\Z$ and $|n|\|x\|<\frac{1}{2}$, then $\langle nx\rangle=n\langle x\rangle$.
\end{itemize}
\end{lem}

\begin{proof}

(a) Say $x_i=n_i+\delta_i$ for each $i$, where $n_i\in\Z$ and $|\delta_i|\leq\frac{1}{2}$. Then each $\langle x_i\rangle=\delta_i$. Put
\begin{align*}
m'&=n_1+\cdots+n_k,\\
\delta'&=\delta_1+\cdots+\delta_k.
\end{align*}
Then
\begin{align*}
\langle x_1+\cdots+x_k\rangle
=\langle(n_1+\cdots+n_k)+(\delta_1+\cdots+\delta_k)\rangle
=\langle m'+\delta'\rangle.
\end{align*}
Since
\begin{align*}
|\delta'|\leq |\delta_1|+\cdots+|\delta_k|=\|x_1\|+\cdots\|x_k\|\leq\frac{1}{2},
\end{align*}
we have
\begin{align*}
\langle x_1+\cdots+x_k\rangle=
\langle m'+\delta'\rangle=\delta'
=\delta_1+\cdots+\delta_k
=\langle x_1\rangle+\cdots+\langle x_k\rangle.
\end{align*}

(b) Say $x=n+\delta$, where $n\in\Z$ and $|\delta|\leq\frac{1}{2}$. Then $\langle x\rangle=\delta$. Now $-x=-n+(-\delta)$, where $-n\in\Z$ and $|-\delta|\leq\frac{1}{2}$, so
\begin{equation*}
\langle -x\rangle=-\delta=-\langle x\rangle.
\end{equation*}

(c) If $n=0$, then $\langle nx\rangle=n\langle x\rangle$. If $n\neq 0$ and $|n|\|x\|<\frac{1}{2}$, put $x_1=x_2=\cdots=x_{|n|}=x$ in part (a) to get $|n|\langle x\rangle=\langle |n|x\rangle$. Therefore $\langle nx\rangle=n\langle x\rangle$ if $n=|n|$. Now if $n=-|n|$, we apply part (b) to get
\begin{equation*}
\langle nx\rangle=\langle -|n|x\rangle=-\langle |n|x\rangle=- |n|\langle x\rangle=n \langle x\rangle.
\end{equation*}
\end{proof}

\begin{rem}\label{rem:1abdependent}
In the introduction, we mentioned that it is straightforward to show Littlewood's conjecture for the pair $(\alpha,\beta)$ when $1,\alpha,\beta$ is $\Q$-linearly dependent. If~$\alpha$ or~$\beta$ is rational (WLOG say $\alpha=\frac{r}{s}$, with $r,s\in\Z$ and $s>0$), then for all $k$, $ks$ is an integer for which
\begin{equation*}
ks\|ks\alpha\|\|ks\beta\|
=
ks\cdot 0\cdot \|ks\beta\|
=
0.
\end{equation*}
If $\alpha$ and $\beta$ are irrational, the following lemma gives us an increasing sequence $\{k_n\}$ of integers for which $k_n\|k_n\alpha\|\|k_n\beta\|<\frac{C}{k_n}$ (for some~$C$) for all $n$, and therefore that Littlewood's conjecture holds for the pair $(\alpha,\beta)$.
\end{rem}

\begin{lem}\label{lem:1abdependent}
If $\alpha,\beta\in \R$ are such that $1,\alpha,\beta$ is $\Q$-linearly dependent, then we can construct an increasing sequence $\{k_n\}$ of positive integers such that $\{k_n\|k_n\alpha\|\}$ and $\{k_n\|k_n\beta\|\}$ are both bounded.
\end{lem}

\begin{proof}
Let $\{q_n\}$ be the sequence of denominators of convergents of $\alpha$. Suppose
\begin{equation*}
\beta
=
\frac{r_1}{s} \alpha+\frac{r_2}{s},
\end{equation*}
where $r_1,r_2,s\in\Z$ and $s>0$. So $s\beta=r_1\alpha+r_2$. Put $k_n=s q_n$. For all $n$, $\|q_n\alpha\|<\frac{1}{q_n}$. So eventually (say for $n>n'$) both $|r_1|\|q_n\alpha\|$ and $s\|q_n\alpha\|$ are less than $\frac{1}{2}$. Suppose $n>n'$. Then by Lemma~\ref{lem:LWnorm}(c),
$\lb r_1 q_n\alpha\rb=r_1\lb q_n\alpha\rb$ and $\lb s q_n\alpha\rb=s\lb q_n\alpha\rb$. Now
\begin{equation*}
\lb k_n\alpha\rb
=
\lb s q_n\alpha\rb
=
s \lb q_n\alpha\rb,
\end{equation*}
\begin{equation*}
\lb k_n\beta\rb
=
\lb s q_n\beta\rb
=
\lb q_n\cdot s\beta\rb
=
\lb q_n(r_1\alpha+r_2)\rb
=
\lb r_1q_n\alpha+q_n r_2\rb
=
\lb r_1q_n\alpha\rb
\end{equation*}
(this last equality due to $q_nr_2$ being an integer). So
\begin{gather*}
k_n\|k_n \alpha\|=k_n\cdot s\|q_n\alpha\|<k_n\frac{s}{q_n}=k_n\frac{s^2}{k_n}=s^2,\\
k_n\|k_n \beta\|=k_n\cdot |r_1|\|q_n\alpha\|<k_n\frac{|r_1|}{q_n}=s|r_1|.
\end{gather*}
\end{proof}

\subsection{Equation of the Ellipse}

Put $u_n=k_n^{1/2}\langle k_n\theta\rangle$ and $v_n=k_n^{1/2}\langle k_n\theta^2\rangle$ (and assume that $k_n>0$), and consider the sequence $\{(u_n,v_n)\}$. We will see in the proof of Theorem~1 that $u_n$ and $v_n$ are bounded, but we can say more about the relationship between them:

\begin{lem}\label{lem:ellipse}
The points $(u_n,v_n)$ get arbitrarily close to the ellipse
\begin{equation*}
(\theta x-2 y)^2+(3\theta^2-4p)x^2
=
\frac{4 d^3}{3\theta^2-p}.
\end{equation*}
\end{lem}

\begin{proof}
By Lemma~\ref{lem:normkntheta}, eventually
$\langle k_n\theta \rangle=-\frac{d Z_n}{\theta}$
and
$\langle k_n\theta^2 \rangle=-d Y_n$. By Lemma~\ref{lem:ellipsoididentity} (and by the definitions of $a_n$, $b_n$, $c_n$)
\begin{equation*}
(X_n+Y_n)^2+(3\theta^2-4p)\left(\frac{Z_n}{\theta}\right)^2
=
\frac{4}{\lambda^n},
\end{equation*}
or (since $X_n=Y_n-Z_n$)
\begin{gather*}
(2Y_n-Z_n)^2+(3\theta^2-4p)\left(\frac{Z_n}{\theta}\right)^2
=
\frac{4}{\lambda^n}
,\\
(2dY_n-dZ_n)^2+(3\theta^2-4p)\left(\frac{dZ_n}{\theta}\right)^2
=
\frac{4d^2}{\lambda^n}
,\\
(-2\langle k_n\theta \rangle+\theta\langle k_n\theta^2 \rangle)^2+(3\theta^2-4p)\langle k_n\theta \rangle^2
=
\frac{4d^2}{\lambda^n}
.
\end{gather*}
Multiplying both sides by $k_n$, we have
\begin{equation*}
(\theta u_n-2 v_n)^2+(3\theta^2-4p)u_n^2
=
4d^2\frac{k_n}{\lambda^n}.
\end{equation*}
By Lemma~\ref{lem:bounds},
\begin{equation*}
\frac{d\lambda^n}{3\theta^2-p}-\frac{C'}{\lambda^{n/2}}
<
k_n
<
\frac{d\lambda^n}{3\theta^2-p}+\frac{C'}{\lambda^{n/2}},
\end{equation*}
where $C'=\frac{(1+\sqrt{2})C_1 d}{3\theta^2-p}$. So
\begin{equation*}
\frac{4d^3}{3\theta^2-p}-\frac{4d^2C'}{\lambda^{3n/2}}
<
k_n\frac{ 4d^2}{\lambda^n}
<
\frac{4d^3}{3\theta^2-p}+\frac{4d^2C'}{\lambda^{3n/2}},
\end{equation*}
or
\begin{equation*}
\frac{4d^3}{3\theta^2-p}-\frac{4d^2C'}{\lambda^{3n/2}}
<
(\theta u_n-2 v_n)^2+(3\theta^2-4p)u_n^2
<
\frac{4d^3}{3\theta^2-p}+\frac{4d^2C'}{\lambda^{3n/2}}.
\end{equation*}
Since $\lambda>1$, this shows that the points $\{(u_n,v_n)\}$ get arbitrarily close to
\begin{equation*}
(\theta x-2 y)^2+(3\theta^2-4p)x^2
=
\frac{4 d^3}{3\theta^2-p}.
\end{equation*}
Since $3\theta^2-4p>0$~\eqref{ineq:disc}, this is an ellipse.
\end{proof}

%% file: disthm1.tex
{\bf Theorem 1. } {\em
Suppose $\theta\in\R$ is the only real root of an irreducible polynomial of the form $x^3-px-q\in\Z[x]$. Then we can construct a Peck sequence for the pair $(\theta,\theta^2)$.
}

\begin{proof}[Proof.]
Suppose $x^3-px-q\in\Z[x]$ is irreducible and has only one real root $\theta$. Let $K$ denote $\Q(\theta)$. Let $d$ be a positive integer such that $\O_K\subset\frac{1}{d}\Z[\theta]$. We can find this with PARI/GP by finding an integral basis using \verb|bnfinit(X^3-p*X-q).zk|, and then letting $d$ be the largest denominator in this integral basis. We want to find an infinite-order unit $\lambda$ satisfying
\begin{equation*}
\lambda>\left((1+\sqrt{2})C_1\right)^{2/3},
\end{equation*}
where (as in Definition~\ref{defn:constants})
$
C_1=\max\{\sqrt{2},\frac{\sqrt{2}|\theta|}{\sqrt{3\theta^2-4p}}\}.
$
We can produce such a $\lambda$ as follows. We can find a fundamental unit $\ep_0$ of $\O_K$ (for example with PARI/GP, using \verb|bnfinit(X^3-p*X-q).fu|). By putting
\begin{equation*}
\ep'=
\max\left\{\pm\ep_0,\pm\frac{1}{\ep_0}\right\},
\end{equation*}
we have a unit with $\ep'>1$. 
Put
\begin{equation*}
m'
=
\left\lceil
\frac{2}{3}
\log_{\ep'} (1+\sqrt{2})C_1
\right\rceil,
\end{equation*}
and finally
\begin{equation*}
\lambda=(\ep')^{m'}.
\end{equation*}

As before, we define $\{a_n\}$, $\{b_n\}$, $\{c_n\}$, and $\{k_n\}$ by
\begin{equation*}
a_n+b_n\theta+c_n\theta^2=\lambda^n
\end{equation*}
and $k_n=d c_n$. We will show that $\{k_n\}$ is a Peck sequence for the pair $(\theta,\theta^2)$.\\

We need to show that
\begin{equation}
\max\{\|k_n\theta\|,\|k_n\theta^2\|\}<\frac{M_{\theta_1}}{k_n^{1/2}}
\end{equation}
for some constant $M_{\theta_1}$, and that $\{k_n\}$ has a subsequence $\{\psi_n\}$ such that
\begin{equation}\label{thm1:ineq:logpsi}
\|\psi_n\theta\|
<
\frac{M_{\theta_2}}{\psi_n^{1/2}\log \psi_n}
\end{equation}
for some constant $M_{\theta_2}$.
\\

Since $\lambda>\left((1+\sqrt{2})C_2\right)^{2/3}$, we have
\begin{equation*}
\lambda^{3n/2}>(1+\sqrt{2})C_1
\end{equation*}
for all $n$, and so each $c_n$ (and therefore each $k_n$) is positive. Put
\begin{equation*}
N_\theta'
=
\max
\left\{
\dfrac{2\sqrt{2}dC_1}{|\theta|},
2 dC_1,
\left(2(1+\sqrt{2})C_1\right)^{1/3},
\lambda^{1/2}
\right\}
\end{equation*}
and $n_\theta=2\log_\lambda N_\theta'$. Suppose $n>n_\theta$. We note that $n\geq 2$, since $n>n_\theta$ is an integer and $n_\theta\geq 1$. (We will use this later on.) Now
\begin{gather*}
\lambda^{n/2}>\dfrac{2\sqrt{2}dC_1}{|\theta|},\\
\lambda^{n/2}>2 dC_1,\\
\lambda^{3n/2}>2(1+\sqrt{2})C_1,
\end{gather*}
so (by Lemma \ref{lem:normkntheta} and Corollary \ref{ineq:knbound})
\begin{gather}
\|k_n\theta\|=\left|\dfrac{d Z_n}{\theta}\right|,\label{thm1:eqn:knthetanorm}\\
\|k_n\theta^2\|=|d Y_n|,\label{thm1:eqn:kntheta2norm}
\end{gather}
and
\begin{equation}\label{thm1:ineq:sqrtkn}
k_n^{1/2}
<
\frac{3}{2} C_2\lambda^{n/2}
\end{equation}
(recall from Definition~\ref{defn:constants} that $C_2=\frac{\sqrt{d}}{\sqrt{3\theta^2-p}}$).
Combining \eqref{thm1:eqn:knthetanorm} and \eqref{thm1:eqn:kntheta2norm} with 
Lemma \ref{lem:bounds} gives us
\begin{gather*}
\|k_n\theta\|
=
\frac{d}{|\theta|}|Z_n|
<
\frac{d\sqrt{2}C_1}{|\theta|}\frac{1}{\lambda^{n/2}},\\
\|k_n\theta^2\|
=
d|Y_n|
<
\frac{dC_1}{\lambda^{n/2}}.
\end{gather*}
Together with \eqref{thm1:ineq:sqrtkn}, we have
\begin{gather*}
\|k_n\theta\|
<
\frac{d\sqrt{2}C_1}{|\theta|}
    \cdot
    \frac{3C_2}{2}
    \cdot
    \frac{1}{k_n^{1/2}},\\
\|k_n\theta^2\|
<
dC_1
    \cdot
    \frac{3C_2}{2}
    \cdot
    \frac{1}{k_n^{1/2}}.
\end{gather*}
Put
\begin{equation*}
M'
=
\max_{i \leq n_\theta}
\{
k_i^{1/2}\|k_i\theta\|,
k_i^{1/2}\|k_i\theta^2\|
\}
\end{equation*}
and
\begin{equation*}
M_{\theta_1}
=
\max
\{
M',\frac{3 d\sqrt{2} C_1 C_2}{2|\theta|},\frac{3 d C_1 C_2}{2}
\}.
\end{equation*}
Then for all $n$,
\begin{equation*}
\max
\{
\|k_n\theta\|,
\|k_n\theta^2\|
\}
<
\frac{M_{\theta_1}}{k_n^{1/2}}.
\end{equation*}



To find a subsequence of $\{k_n\}$ satisfying \eqref{thm1:ineq:logpsi}, we first combine \eqref{thm1:eqn:knthetanorm} with Corollary \ref{cor:imth2n} to get that (for $n>n_\theta$)
\begin{equation*}
\|k_n\theta\|
=
d\left|\frac{Z_n}{\theta}\right|
=
\frac{2d}{\sqrt{3\theta^2-4p}}
    \left|
    \Im \lambda_2^n
    \right|.
\end{equation*}
This, along with \eqref{thm1:ineq:sqrtkn}, gives us
\begin{equation*}
k_n^{1/2}\|k_n\theta\|
<
\frac{3 C_2}{2}
    \cdot
    \frac{2d}{\sqrt{3\theta^2-4p}}
    \cdot
    \lambda^{n/2}
    \left|
    \Im \lambda_2^n
    \right|
=
\frac{3dC_2}{2\sqrt{3\theta^2-4p}}
    \left|
    \Im (\sqrt{\lambda}\lambda_2)^n
    \right|.
\end{equation*}
Put
\begin{equation*}
\phi=\arctan
    \left(
    \frac{\sqrt{3\theta^2-4p}(b_1-c_1\theta)}{2(a_1+pc_1)-b_1\theta-c_1\theta^2}
    \right)
\end{equation*}
(so that $\sqrt{\lambda}\lambda_2=\pm e^{i\phi}$), and let $\{P_n/Q_n\}$ be the sequence of convergents of $\phi/\pi$. Then (by Lemma \ref{lem:convergents})
\begin{equation}\label{thm1:ineq:kQnLW}
k_{Q_n}^{1/2}\|k_{Q_n}\theta\|
<
\frac{3dC_2}{2\sqrt{3\theta^2-4p}}
    \left|
    \Im (\lambda^{1/2}\lambda_2)^{Q_n}
    \right|
<
\frac{3dC_2}{2\sqrt{3\theta^2-4p}}
\cdot
\frac{\pi}{Q_{n+1}}
\end{equation}
for $Q_n>n_\theta$.

By Corollary~\ref{ineq:knbound},
\begin{equation*}
k_{Q_{n+1}}^{1/2}
<
\frac{3C_2}{2}\lambda^{Q_{n+1}/2}
\leq
C_3\lambda^{Q_{n+1}/2}
\end{equation*}
(where $C_3=\max\{1,\frac{3C_2}{2}\}$ as in Definition~\ref{defn:constants}). Earlier we noted that $n\geq 2$, and so $Q_{n+1}\geq n\geq 2$ (because $\{Q_n\}$ is bounded below termwise by the Fibonacci sequence.) Then (since $C_3\geq 1$) $C_3\leq C_3^{Q_{n+1}/2}$, so
\begin{equation}\label{thm1:kQnplus1}
k_{Q_{n+1}}^{1/2}
<
(C_3\lambda)^{Q_{n+1}/2}.
\end{equation}
Then
\begin{equation*}
\frac{1}{Q_{n+1}}
<
\frac{\log C_3\lambda}{\log k_{Q_{n+1}}}.
\end{equation*}
So putting
\begin{equation*}
M_{\theta_2}
=
\frac{3dC_2\pi }{2\sqrt{3\theta^2-4p}}
\cdot
\log C_3\lambda,
\end{equation*}
we have that for $Q_n>n_\theta$, the subsequence $\{k_{Q_n}\}$ satisfies
\begin{equation}\label{thm1:LWprodbound1}
k_{Q_n}^{1/2}\|k_{Q_n}\theta\|
<
\frac{3dC_2}{2\sqrt{3\theta^2-4p}}
    \cdot
    \frac{\pi}{Q_{n+1}}
<
\frac{M_{\theta_2}}{\log k_{Q_{n+1}}}.
\end{equation}
Since $Q_{n+1}>Q_n$ for $n\geq 2$, this gives us
\begin{equation}
\|k_{Q_n}\theta\|
<
\frac{M_{\theta_2}}{k_{Q_n}^{1/2}\log k_{Q_{n}}}
\end{equation}
when $Q_n>n_\theta$. We can shift the sequence $\{k_{Q_n}\}$ to produce a sequence $\{\psi_n\}$ such that
\begin{equation*}
\|\psi_n\theta\|
<
\frac{M_{\theta_2}}{\psi_n^{1/2}\log \psi_n}
\end{equation*}
for all $n$.
\end{proof}

%% file: disthm2.tex
{\bf Theorem 2. } {\em Suppose $\alpha\in\R$ is the only real root of an irreducible cubic polynomial in $\Q[x]$. Then we can construct a Peck sequence for the pair $(\alpha,\alpha^2)$.
}


\begin{proof}
First, we make some reductions in order to use Theorem~1. By clearing denominators, we get an irreducible cubic
\begin{equation*}
f(x)=Ax^3+Bx^2+Cx+D\in\Z[x]
\end{equation*}
(so $A\neq 0$) for which $\alpha$ is the only real root. WLOG we assume $A$ is positive. By putting
\begin{equation*}
g(x)=27A^2
    f
    \left(
    \frac{x-B}{3A}
    \right),
\end{equation*}
we transform $f(x)$ into the form $x^3-px-q\in\Z[x]$, where
\begin{gather*}
p=3A(B^2-3AC),\\
q=-2B^3+9ABC-27A^2D.
\end{gather*}
Since $f$ is irreducible, so is $g$. Recall that $\alpha_i$ denotes $\sigma_i(\alpha)$ for $i=1,2,3$ (where $\sigma_1$ is the real embedding). Put
\begin{equation*}
\theta_i=3A \alpha_i+B.
\end{equation*}
Now $\theta:=\theta_1$ is real, and $\theta_2,\theta_3$ are complex, and all three satisfy
\begin{equation*}
g(\theta_i)=27A^2f(\alpha_i)=0.
\end{equation*}
So $\theta$ satisfies the hypothesis of Theorem 1.\\

Let $\{k_n\}$, $C_2$, $C_3$, $M_{\theta_1}$, $M_{\theta_2}$, $\{Q_j\}$, $N_\theta'$, $n_\theta$ be as in the proof of Theorem 1. Define the sequence $\{\ell_n\}$ by
\begin{equation*}
\ell_n=9A^2 k_n.
\end{equation*}
To show that $\{\ell_n\}$ is a Peck sequence for $(\alpha,\alpha^2)$, we first note that since $\{k_n\}$ is an increasing sequence of positive integers, so is $\{\ell_n\}$. To show the other properties, we put
\begin{equation*}
N_\alpha'
=
\max
\{
N_\theta',
12 AM_{\theta_1},
16|B|M_{\theta_1}
\}
\end{equation*}
and $n_\alpha=2 \log_\lambda N_\alpha'$, and use the following:


\begin{lem}
For $n>n_\alpha$,
\begin{equation*}
\langle \ell_n\alpha\rangle
=
3A \langle k_n\theta\rangle
\end{equation*}
and
\begin{equation*}
\langle \ell_n\alpha^2\rangle
=
\langle k_n\theta^2\rangle-2B\langle k_n\theta\rangle.
\end{equation*}
\end{lem}

\begin{proof}[Proof of Lemma]
Our strategy is to use Lemma \ref{lem:LWnorm} to get
$\langle \ell_n\alpha\rangle$ and $\langle \ell_n\alpha^2\rangle$
in terms of
$\langle k_n\theta\rangle$ and $\langle k_n\theta^2\rangle$. 
Let $n>n_\alpha$. Then
\begin{gather*}
\lambda^{n/2}>12 AM_{\theta_1},\\
\lambda^{n/2}>16|B|M_{\theta_1}.
\end{gather*}
Since $n>n_\alpha>n_\theta$, we also have $k_n^{1/2}>\frac{\lambda^{n/2}}{2}$ (by Corollary \ref{ineq:knbound}). So
\begin{gather*}
k_n^{1/2}>6AM_{\theta_1},\\
k_n^{1/2}>8|B|M_{\theta_1}.
\end{gather*}
Now
\begin{gather*}
3A\alpha=\theta-B,\\
9A^2\alpha^2=(3A\alpha)^2=\theta^2-2B\theta+B^2,
\end{gather*}
so
\begin{gather*}
\ell_n\alpha
=
9A^2 k_n\alpha
=
3Ak_n(3A\alpha)
=
3Ak_n(\theta-B)
=
3Ak_n\theta-3ABk_n,\\
\ell_n\alpha^2
=
9A^2k_n\alpha^2
=
k_n(3A\alpha)^2
=
k_n(\theta-B)^2
=
k_n\theta^2-2Bk_n\theta+B^2 k_n .
\end{gather*}
Since $3ABk_n$ and $B^2 k_n$ are integers, we have
\begin{equation}\label{eq:mnalpha1a}
\langle \ell_n\alpha\rangle
=\langle 3Ak_n\theta-3ABk_n\rangle
=\langle 3A k_n\theta\rangle
\end{equation}
and
\begin{equation}\label{eq:mnalpha2a}
\langle \ell_n\alpha^2\rangle
=\langle k_n\theta^2-2Bk_n\theta+B^2k_n\rangle
=\langle k_n\theta^2-2Bk_n\theta\rangle.
\end{equation}
Now $k_n^{1/2}>6AM_{\theta_1}>4M_{\theta_1}$ (the second inequality holding because $A\geq 1$) and $k_n^{1/2}>8|B|M_{\theta_1}$, so
\begin{gather*}
3A\|k_n\theta\|
<3A\frac{M_{\theta_1}}{k_n^{1/2}}
<3A\frac{M_{\theta_1}}{6AM_{\theta_1}}
=\frac{1}{2},\\
\|k_n\theta^2\|
<
\frac{M_{\theta_1}}{k_n^{1/2}}
<
\frac{M_{{\theta_1}}}{4M_{\theta_1}}
<
\frac{1}{4},\\
2|B|\|k_n\theta\|
<
2|B|\frac{M_{\theta_1}}{k_n^{1/2}}
<
2|B|\frac{M_{\theta_1}}{8|B|M_{\theta_1}}
=
\frac{1}{4}.
\end{gather*}
Since $3A\|k_n\theta\|<\frac{1}{2}$ and $2|B|\|k_n\theta\|<\frac{1}{2}$, we have (by Lemma \ref{lem:LWnorm}~(c))
\begin{gather*}
\langle 3A k_n\theta\rangle
=3A \langle k_n\theta\rangle,\\
\langle 2B k_n\theta\rangle
=2B \langle k_n\theta\rangle.
\end{gather*}
And since
\begin{equation*}
\|k_n\theta^2\|+\|2Bk_n\theta\|
=
\|k_n\theta^2\|+2|B|\|k_n\theta\|
<
\frac{1}{2},
\end{equation*}
we have (by Lemma\ref{lem:LWnorm}~(b) and (c))
\begin{equation*}
\langle k_n\theta^2-2B k_n\theta\rangle
=
\langle k_n\theta^2\rangle+\langle -2B k_n\theta\rangle
=
\langle k_n\theta^2\rangle-\langle 2B k_n\theta\rangle
=
\langle k_n\theta^2\rangle-2B\langle k_n\theta\rangle.
\end{equation*}
Then
\begin{gather*}
\langle \ell_n\alpha\rangle
=
\langle 3A k_n\theta\rangle
=
3A \langle k_n\theta\rangle,\\
\langle \ell_n\alpha^2\rangle
=
\langle k_n\theta^2-2Bk_n\theta\rangle
=
\langle k_n\theta^2\rangle-2B\langle k_n\theta\rangle.
\end{gather*}
\end{proof}


From here, it is fairly straightforward (and similar to Theorem 1) to show that $\{\ell_n\}$ is a Peck sequence for the pair $(\alpha,\alpha^2)$. For $n>n_\alpha$
\begin{equation*}
\ell_n^{1/2}\|\ell_n\alpha\|
=(3Ak_n^{1/2})(3A\|k_n\theta\|)
=9A^2k_n^{1/2}\|k_n\theta\|
<9A^2M_{\theta_1}
\end{equation*}
and
\begin{align*}
\ell_n^{1/2}\|\ell_n\alpha^2\|
&\leq
\left(3Ak_n^{1/2}\right)\left(\|k_n\theta^2\|+2|B|\|k_n\theta\|\right)
<
3AM_{\theta_1}\left(1+2|B|\right).
\end{align*}
\\
Put
\begin{equation*}
M_\alpha'
=
\max_{i\leq n_\alpha}
    \{
    \ell_i^{1/2}\|\ell_i\alpha\|,
    \ell_i^{1/2}\|\ell_i\alpha^2\|
    \}
\end{equation*}
and
\begin{equation*}
M_{\alpha_1}
=
\max
    \{
    M_\alpha',
    9A^2M_{\theta_1},
    3AM_{\theta_1}(1+2|B|)
    \}.
\end{equation*}
Then
\begin{equation}\label{thm2:ineq:Malpha}
\max
    \{
    \ell_n^{1/2}\|\ell_n\alpha\|,
    \ell_n^{1/2}\|\ell_n\alpha^2\|
    \}
<
M_{\alpha_1}
\end{equation}
for all $n$.


Now consider the subsequence $\{\ell_{Q_n}\}$. In equations \eqref{thm1:ineq:kQnLW} and \eqref{thm1:kQnplus1} in the proof of Theorem~1, we saw that
\begin{equation*}
k_{Q_n}^{1/2}\|k_{Q_n}\theta\|
<
\frac{3dC_2}{2\sqrt{3\theta^2-4p}}
\cdot
\frac{\pi}{Q_{n+1}}
\end{equation*}
and
\begin{gather*}
k_{Q_{n+1}}^{1/2}
<
C_3 \lambda^{Q_{n+1}/2}
\end{gather*}
when $Q_n>n_\theta$. If $Q_n>n_\alpha\geq n_\theta$, then
\begin{equation}\label{thm2:ineq:lQnLW}
\ell_{Q_n}^{1/2}\|\ell_{Q_n}\alpha\|
=
9A^2k_{Q_n}^{1/2}\|k_{Q_n}\theta\|
<
9A^2
\cdot
\frac{3dC_2}{2\sqrt{3\theta^2-4p}}
\cdot
\frac{\pi}{Q_{n+1}}.
\end{equation}
Since $Q_{n+1}/2\geq 1$ and $3A\geq 1$,
\begin{equation}\label{thm2:ineq:lQn}
\ell_{Q_{n+1}}^{1/2}
=
3Ak_{Q_{n+1}}^{1/2}
<
3 A (C_3 \lambda)^{Q_{n+1}/2}
\leq
(3AC_3\lambda)^{Q_{n+1}/2},
\end{equation}
or
\begin{equation*}
\frac{1}{Q_{n+1}}
<
\frac
    {\log (3 A C_3 \lambda)}
    {\log \ell_{Q_{n+1}}}.
\end{equation*}
Then
\begin{equation}
\ell_{Q_n}^{1/2}\|\ell_{Q_n}\alpha\|
<
9A^2
\cdot
\frac{3dC_2}{2\sqrt{3\theta^2-4p}}
\cdot
\frac
    {\pi\log (3 A C_3 \lambda)}
    {\log \ell_{Q_{n+1}}}.
\end{equation}
Putting
\begin{equation*}
M_{\alpha_2}
=
9A^2
\cdot
\frac{3dC_2}{2\sqrt{3\theta^2-4p}}
\cdot
\pi\log (3 A C_3 \lambda),
\end{equation*}
and since $Q_{n+1}>Q_n$ for $n\geq 2$, we have
\begin{equation*}
\|\ell_{Q_n}\alpha\|
<
\frac{M_{\alpha_2}}{\ell_{Q_n}^{1/2}\log \ell_{Q_n}}
\end{equation*}
when $Q_n>n_\alpha$. By an appropriate shift of the sequence $\{\ell_{Q_n}\}$, we get a sequence $\{\psi_n\}$ such that
\begin{equation*}
\|\psi_n\alpha\|
<
\frac{M_{\alpha_2}}{\psi_n^{1/2}\log \psi_n}
\end{equation*}
for all $n$.
\end{proof}

%% file: disthm3.tex
{\bf Theorem 3. }{\em
Suppose $\alpha,\beta\in K$, where $K\subseteq \R$ is a cubic field with only one real embedding.
Then we can construct a Peck sequence for the pair $(\alpha,\beta)$.
}

\begin{proof}
Let $\alpha, \beta\in K\subset \R$, where $K$ is a cubic field which has only one real embedding. As before, let $\sigma_1$ be the real embedding and $\sigma_2,\sigma_3$ the complex embeddings.\\
\\
{\bf Case 1: \boldmath $\alpha\in\Q$ or $\beta\in\Q$. } WLOG suppose $\beta=\frac{u}{v}$. If $\alpha\in\Q$ (say $\alpha=\frac{r}{s}$), then put $m_n=|nsv|$. If $\alpha\notin\Q$, then put $m_n=|v| q_n$, where $\{q_n\}$ is the sequence of denominators of convergents of $\alpha$. In the first case, $\|m_n\alpha\|=\|m_n\beta\|=0$ for all $n$. In the second case, $\|m_n\beta\|=0$ for all $n$, and eventually $\|m_n\alpha\|=|v|\|q_n\alpha\|<\frac{|v|}{q_n^2}$. In either case, $\{m_n^{1/2}\|m_n\alpha\|\}$, $\{m_n^{1/2}\|m_n\beta\|\}$, and $\{m_n^{1/2}\log m_n \|m_n\alpha\|\}$ are all bounded.\\
\\
{\bf Case 2: \boldmath $\alpha,\beta\notin\Q$.}
Let $F_\alpha(x)$ be the minimal polynomial (over $\Q$) of $\alpha$. Since $\alpha\notin\Q$, $F_\alpha$ is a cubic. Clearing denominators, we get (for some $A>0$)
\begin{equation*}
f(x)
:=
Ax^3+Bx^2+Cx+D
:=
A \cdot F_\alpha(x)
\in\Z[x],
\end{equation*}
where $\gcd(A,B,C,D)=1$. Since $f\in\Z[x]$ is an irreducible cubic with one real root~$\alpha$, we use Theorem~2.\\

Let $\{k_n\}$, $\{\ell_n\}$, $C_2$, $C_3$, $M_{\alpha_1}$, $M_{\alpha_2}$, $\{Q_j\}$, $N_{\alpha}'$, $n_\alpha$, be as in the proofs of Theorems~1 and 2. Since $1,\alpha,\alpha^2$ is a $\Q$-basis for $K$, we can write
\begin{equation*}
\beta=\frac{r_0}{s}+\frac{r_1}{s}\alpha+\frac{r_2}{s}\alpha^2,
\end{equation*}
where $r_0,r_1,r_2,s\in\Z$ and $s>0$. (We note that, since $\beta\notin\Q$, $r_1$ and $r_2$ are not both zero.) Put $m_n=s\ell_n$. We will show that $\{m_n\}$ is a Peck sequence for the pair $(\alpha,\beta)$. Put
\begin{equation*}
N_\beta'
=
\max
    \left\{
    \frac{8 M_{\alpha_1}|r_1|}{3A},
    \frac{8 M_{\alpha_1}|r_2|}{3A},
    \frac{4s M_{\alpha_1}}{3A},
    N_\alpha'
    \right\}
\end{equation*}
and
\begin{equation*}
n_\beta=2\log_\lambda N_\beta'.
\end{equation*}
As in the proof of Theorem~2, we define $n_\beta$ so that we can use Lemma~\ref{lem:LWnorm} to get bounds on $\{m_n^{1/2}\|m_n\alpha\|\}$ and $\{m_n^{1/2}\|m_n\beta\|\}$. We do this with the following:


\begin{lem}\label{lem:thm3}
For $n>n_\beta$
\begin{gather*}
\langle m_n \alpha \rangle
=
s \langle \ell_n \alpha \rangle
,\\
\langle
    m_n \beta
\rangle
=
r_1
\langle
    \ell_n \alpha
\rangle
+
r_2
\langle
    \ell_n \alpha^2
\rangle.
\end{gather*}
\end{lem}

\begin{proof}[Proof of Lemma.]
Let $n>n_\beta$ and note that
\begin{gather*}
\lambda^{n/2}
>
N_\beta'
\geq
\frac{2}{3A}
\max
    \{
    4 M_{\alpha_1} |r_1|,
    4 M_{\alpha_2} |r_2|,
    2 s M_{\alpha_1}
    \}
,
\end{gather*}
and so (using Corollary~\ref{ineq:knbound} and that $\ell_n=9A^2k_n$)
\begin{equation*}
\ell_n^{n/2}
=
3 A k_n^{n/2}
>
3 A \frac{\lambda^{n/2}}{2}
>
\max
    \{
    4 M_{\alpha_1} |r_1|,
    4 M_{\alpha_2} |r_2|,
    2 s M_{\alpha_1}
    \}
.
\end{equation*}
Then
\begin{gather}
|r_1|\frac{M_{\alpha_1}}{\ell_n^{1/2}}
<
\frac{1}{4}
,\label{thm2:ineq:r1Malpha}\\
|r_2|\frac{M_{\alpha_2}}{\ell_n^{1/2}}
<
\frac{1}{4}
,\label{thm2:ineq:r2Malpha}\\
s\frac{M_{\alpha_1}}{\ell_n^{1/2}}
<
\frac{1}{2}.\label{thm2:ineq:s1Malpha}
\end{gather}
Now
\begin{gather*}
s\beta
=
r_0 + r_1 \alpha + r_2 \alpha^2
,\\
m_n \beta
=
(s \ell_n) \beta
=
\ell_n (s\beta)
=
\ell_n (r_0+ r_1 \alpha + r_2 \alpha^2).
\end{gather*}
Since $r_0 \ell_n\in\Z$,
\begin{equation*}
\langle m_n\beta \rangle
=
\langle
    r_1 \ell_n \alpha + r_2 \ell_n \alpha^2
\rangle.
\end{equation*}
Combining \eqref{thm2:ineq:Malpha} (from Theorem~2) with \eqref{thm2:ineq:r1Malpha} and \eqref{thm2:ineq:r2Malpha}, we have
\begin{gather*}
|r_1|\|\ell_n\alpha\|
<
|r_1|\frac{M_{\alpha_1}}{\ell_n^{1/2}}
<
\frac{1}{4}
<\frac{1}{2}
\end{gather*}
and
\begin{equation*}
|r_2|\|\ell_n\alpha^2\|
<
|r_2|\frac{M_{\alpha_2}}{\ell_n^{1/2}}
<
\frac{1}{4}
<
\frac{1}{2},
\end{equation*}
so Lemma~\ref{lem:LWnorm}(c) yields
\begin{gather*}
\langle
    r_1 \ell_n \alpha
\rangle
=
r_1
\langle
    \ell_n \alpha
\rangle
,\\
\langle
    r_2 \ell_n \alpha^2
\rangle
=
r_2
\langle
    \ell_n \alpha^2
\rangle.
\end{gather*}
Then since
\begin{equation*}
\|r_1 \ell_n \alpha\|
+
\|r_2 \ell_n \alpha^2\|
=
|r_1| \|\ell_n \alpha\|
+
|r_2| \|\ell_n \alpha^2\|
<
\frac{1}{4}+\frac{1}{4}
=
\frac{1}{2},
\end{equation*}
Lemma~\ref{lem:LWnorm}(a) gives us
\begin{equation*}
\langle
    r_1 \ell_n \alpha + r_2 \ell_n \alpha^2
\rangle
=
\langle
    r_1 \ell_n \alpha
\rangle
+
\langle
    r_2 \ell_n \alpha^2
\rangle.
\end{equation*}
So
\begin{equation*}
\langle
    m_n \beta
\rangle
=
\langle
    r_1 \ell_n \alpha
\rangle
+
\langle
    r_2 \ell_n \alpha^2
\rangle
=
r_1
\langle
    \ell_n \alpha
\rangle
+
r_2
\langle
    \ell_n \alpha^2
\rangle.
\end{equation*}
To show
$\langle m_n \alpha \rangle
=
s \langle \ell_n \alpha \rangle$,
we use \eqref{thm2:ineq:Malpha} with \eqref{thm2:ineq:s1Malpha} to get
\begin{equation*}
s \|\ell_n \alpha\|
<
s\frac{M_{\alpha_1}}{\ell_n^{1/2}}
<
\frac{1}{2}.
\end{equation*}
Then (by Lemma~\ref{lem:LWnorm}(c))
$\langle m_n \alpha \rangle
=
\langle s \ell_n \alpha \rangle
=
s \langle \ell_n \alpha \rangle.$
\end{proof}


With this lemma in hand, the rest of the proof of Theorem~3 is nearly identical to the arguments in the proofs of Theorems~1 and~2. The lemma, along with \eqref{thm2:ineq:Malpha}, gives us that for $n>n_\beta$,
\begin{gather*}
m_n^{1/2}\|m_n\alpha\|
=
(s^{1/2}\ell_n^{1/2})(s\|\ell_n \alpha\|)
=
s^{3/2}\ell_n^{1/2}\|\ell_n \alpha\|
<
s^{3/2}M_{\alpha_1}
,\\
m_n^{1/2}\|m_n\beta\|
=
(s^{1/2}\ell_n^{1/2})
    (|r_1|\|\ell_n \alpha\|+|r_2|\|\ell_n \alpha^2\|)
<
s^{3/2}(|r_1|+|r_2|)M_{\alpha_1}.
\end{gather*}
We noted earlier that $r_1$ and $r_2$ are not both zero, so $|r_1|+|r_2|\geq 1$, and therefore
\begin{equation*}
m_n^{1/2}\|m_n\alpha\|
<
s^{3/2}(|r_1|+|r_2|)M_{\alpha_1}.
\end{equation*}
So with
\begin{equation*}
M_{\beta}'
=
\max_{i\leq n_\beta}
    \{
    m_n^{1/2}\|m_n\alpha\|,
    m_n^{1/2}\|m_n\beta\|
    \}
\end{equation*}
and
\begin{equation*}
M_{\beta_1}
=
\max
    \{
    M_\beta',
    s^{3/2}(|r_1|+|r_2|)M_{\alpha_1}
    \}
,
\end{equation*}
we have
\begin{equation*}
\max
    \{
    m_n^{1/2}\|m_n\alpha\|,
    m_n^{1/2}\|m_n\beta\|
    \}
\leq
M_{\beta_1}
\end{equation*}
for all $n$.\\


Consider the subsequence $\{m_{Q_n}\}$. In \eqref{thm2:ineq:lQnLW} and \eqref{thm2:ineq:lQn} from the proof of Theorem~2, we had
\begin{equation*}
\ell_{Q_n}^{1/2}\|\ell_{Q_n}\alpha\|
<
9A^2
\cdot
\frac{3dC_2}{2\sqrt{3\theta^2-4p}}
\cdot
\frac{\pi}{Q_{n+1}}
\end{equation*}
and
\begin{equation*}
\ell_{Q_{n+1}}^{1/2}
<
(3 A C_3 \lambda)^{Q_{n+1}/2}
\end{equation*}
for $Q_n>n_\alpha$. With Lemma~\ref{lem:thm3}, we have that
\begin{equation}\label{thm3:ineq:mQnLW}
m_{Q_n}^{1/2}\|m_{Q_n} \alpha\|
=
s^{3/2} \ell_{Q_n}^{1/2}\|\ell_n \alpha\|
<
9A^2s^{3/2}
\cdot
\frac{3dC_2}{2\sqrt{3\theta^2-4p}}
\cdot
\frac{\pi}{Q_{n+1}}
\end{equation}
for $Q_n>n_\beta\geq n_\alpha$. Since $Q_{n+1}/2\geq 1$ and $s\geq 1$, we have
\begin{equation*}
m_{Q_{n+1}}^{1/2}
=
s^{1/2} \ell_{Q_{n+1}}^{1/2}
<
s^{1/2}(3A C_3 \lambda)^{Q_{n+1}/2}
\geq
(3As^{1/2}C_3\lambda)^{Q_{n+1}/2},
\end{equation*}
from which we get
\begin{equation*}
\frac{1}{Q_{n+1}}
\leq
\frac{\log(3As^{1/2}C_3\lambda)}{\log m_{Q_{n+1}}}.
\end{equation*}
Then
\begin{equation}\label{thm3:ineq:mQnLW}
m_{Q_n}^{1/2}\|m_{Q_n} \alpha\|
<
9A^2s^{3/2}
\cdot
\frac{3dC_2}{2\sqrt{3\theta^2-4p}}
\cdot
\frac{\pi\log(3As^{1/2}C_3\lambda)}{\log m_{Q_{n+1}}}.
\end{equation}
Put
\begin{equation*}
M_{\beta_2}
=
9A^2s^{3/2}
\cdot
\frac{3dC_2\pi}{2\sqrt{3\theta^2-4p}}
\cdot
\log(3As^{1/2}C_3\lambda).
\end{equation*}
Then since $Q_{n+1}>Q_n$ for $n\geq 2$, we get
\begin{equation*}
\|m_{Q_n}\alpha\|
<
\frac{M_{\beta_2}}{m_{Q_n}^{1/2}\log m_{Q_n}}
\end{equation*}
for $n>n_\beta$. By appropriately shifting the sequence $\{m_{Q_n}\}$, we have a sequence $\{\psi_n\}$ such that
\begin{equation*}
\|\psi_n\alpha\|
<
\frac{M_{\beta_2}}{\psi_n^{1/2}\log \psi_n}
\end{equation*}
for all $n$.
\end{proof}

%% file: disfurtherquestions.tex

\subsection{Constructing Simultaneous Approximations}
The main idea of this paper is to take an appropriate $\theta$ with $K=\Q(\theta)$, find a unit $\lambda\in\O_K$ with $\lambda>1$, and then use the sequence $\{\lambda^n\}$ to produce a sequence that satisfies Peck's inequalities for the power basis $1,\theta,\theta^2$.

\begin{enumerate}
\item Our most immediate question is whether this method could work when $K$ is a totally real cubic field? (So far, our method has failed to produce Peck sequences in this case.)

    Peck's proof in \cite{bib:Pe} makes no distinction between the cases that $K$ has one real embedding or three real embeddings. In the second case, the unit group has rank~2. So if this method were to work in that case, then perhaps finding an appropriate $\lambda$ is a more delicate process than in the rank~1 case (where our main requirement was only that $\lambda>1$).
\item Another question is whether this method could generalize to find simultaneous approximations of pairs coming from higher-degree number fields?

    For example, consider $K=\Q(\sqrt{2},\sqrt{3})$. Since $K=\Q(\theta)$ for $\theta=\sqrt{2}+\sqrt{3}$, could we take powers of an appropriate unit to construct ``good'' simultaneous approximations to the basis $1,\theta,\theta^2,\theta^3$? If so, that would also give us simultaneous approximations to the pair $(\sqrt{2},\sqrt{3})$. How good would such approximations be? (I.e., would they satisfy Peck's inequalities? Would they be sharp enough to show Littlewood's conjecture for the pair $(\sqrt{2},\sqrt{3})$?)

\end{enumerate}

\subsection{Questions Based on Heuristics}
Let $f:\N\to(0,\infty)$, and consider sequences of the form $\{f(n)\|n\alpha\|\|n\beta\|\}$ for $\alpha,\beta\in\R$. Based on the probabilistic argument in Section~\ref{sec:heuristics}, it seems natural to ask:
\begin{enumerate}
\item[3.]
If $\sum_{n=1}^\infty \frac{1}{f(n)}$ diverges, then is
\begin{equation*}
\liminf_{n\to\infty} f(n)\|n\alpha\|\|n\beta\|=0?
\end{equation*}
\item[4.]
If $\sum_{n=1}^\infty\frac{1}{f(n)}$ converges, then for almost all $\alpha,\beta\in\R$ is
\begin{equation*}
\liminf_{n\to\infty} f(n)\|n\alpha\|\|n\beta\|>0?
\end{equation*}
\end{enumerate}

We could also consider a subset $S\subset \N$ and the set $\{s\|s\alpha\|\|s\beta\|:s\in S\}$. Based on the probabilistic argument again, it seems reasonable to ask:

\begin{enumerate}
\item[5.] If $\sum_{s\in S}\frac{1}{s}$ diverges, then is
\begin{equation*}
\inf_{s\in S}s\|s\alpha\|\|s\beta\|=0?
\end{equation*}
\item[6.] If $\sum_{s\in S}\frac{1}{s}$ converges, then for almost all $\alpha,\beta$ is
\begin{equation*}
\inf_{s\in S}s\|s\alpha\|\|s\beta\|>0?
\end{equation*}
\end{enumerate}

%
%
%


%% file: disexamples.tex
\newcommand{\figwidth}{.9}

\input{disExAlg.tex}

\section{Examples}

\input{disEx0.0.2.tex}

\newpage

\input{disEx0.1.1.tex}

\newpage

\input{disEx7.0.2.tex}

\newpage

\input{disEx0.147.740.tex}

\newpage

\input{disEx0.8.10.tex}

\newpage

\input{disEx0.-1.1.tex}

\newpage

\input{disEx1.1.1.tex}

\newpage

\input{disEx0.-2.1.tex}

\newpage

\input{disEx0.2.2.tex}

\newpage

\input{disEx1.-1.2.tex}

\newpage

\input{disEx0.1.2.tex}

%
%
%
%

%% file: disExAlg.tex
If $\alpha,\beta\in K$, where $K$ is a cubic field with only one real embedding, then as a result of the theorems, we are able to construct an eventually-increasing sequence $\{\psi_n\}$ 
of positive integers satisfying
\begin{equation*}
\psi_n\|\psi_n\alpha\|\|\psi_n\beta\|
<
\frac{C}{\log \psi_n}
\end{equation*}
(for some $C>0$), thus providing a constructive proof that Littlewood's conjecture holds for the pair $(\alpha,\beta)$. However, in the proofs of the theorems, we actually have a slightly sharper bound (and get the $\frac{C}{\log\psi_n}$ bound as a consequence).\\

In this chapter, we will consider this sharper bound
\begin{equation*}
\psi_n\|\psi_n\alpha\|\|\psi_n\beta\|
<
\frac{C_0}{Q_{n+1}},
\end{equation*}
and we will compare $C_0/Q_{n+1}$ to the actual value of $\psi_n\|\psi_n\alpha\|\|\psi_n\beta\|$ for several examples of pairs $(\alpha,\beta)$.

\section{Constructing $\{\psi_n\}$ and Calculating the Bound}
To bound the Littlewood product $\psi_n\|\psi_n\alpha\|\|\psi_n\beta\|$, we will consider separate bounds on $\psi_n^{1/2}\|\psi_n\alpha\|$ and $\psi_n^{1/2}\|\psi_n\beta\|$, as we did in the proofs of the theorems. We are interested in {\em eventual} bounds, so we will slightly modify our setup in the following ways:
\begin{itemize}
\item[(i)] In the proof of Theorem~1, we wanted $\lambda>((1+\sqrt{2})C_1)^{1/3}$. For our constructions, we will just require $\lambda>1$.
\item[(ii)] The constants $M_{\theta_1}$, $M_{\alpha_1}$, and $M_{\beta_1}$ were defined so that certain inequalities held for all $n$. We will relax that requirement so that the corresponding inequalities hold {\em eventually}, and then we will slightly modify the definitions of the constants. We will denote these re-defined constants by $\widetilde{M_\theta}$, $\widetilde{M_\alpha}$, and $\widetilde{M_{\beta_1}}$. We will define $\widetilde{N}$ so that the inequalities hold for $\lambda^{n/2}>\widetilde{N}$.
\item[(iii)] In the proofs of the theorems, we took $\{\psi_n\}$ to be an appropriately shifted subsequence (of $\{k_{Q_n}\}$ or $\{\ell_{Q_n}\}$ or $\{m_{Q_n}\}$). Instead, we will take $\{\psi_n\}$ to be $m_{Q_n}$, and construct a constant $n_0$ for which $\psi_n$ satisfies the inequalities when $Q_n>n_0$.
\end{itemize}

We summarize our results in the following algorithm to construct $\{\psi_n\}=\{m_{Q_n}\}$ and to compute $n_0$, $C_0$ such that
\begin{equation*}
\psi_n\|\psi_n\alpha\|\|\psi_n\beta\|
\leq
\frac{C_0}{Q_{n+1}}
\end{equation*}
for $Q_n>n_0$. We will assume that $\alpha,\beta\notin\Q$, and that we have the polynomial $f(x)=Ax^3+Bx^2+Cx+D\in\Z[x]$ for which: $f(\alpha)=0$, $\gcd(A,B,C,D)=1$, and $A>0$. In the case that $B=0$, we can modify the algorithm to reduce $C_0$ by a factor of 81. In each step of the algorithm, we will state any modifications for the case $B=0$.

\subsection{The Algorithm}

\newcommand{\algwidtha}{1.01\textwidth}
\newcommand{\algshifta}{-2.1cm}
\newcommand{\algwidthb}{.99\textwidth}
\newcommand{\algshiftb}{-1.9cm}

\begin{itemize}
\item[]
\begin{itemize}
\item[] \hskip\algshifta\vbox{\hrule width\algwidtha height 2pt}
\item[\bf Given:    ] $\alpha,\beta$: $\alpha$ is real root of $f(x)=Ax^3+Bx^2+Cx+D$, $\beta=\frac{r_0}{s}+\frac{r_1}{s}\alpha+\frac{r_2}{s}\alpha^2$
\item[] \hskip\algshiftb\vbox{\hrule width\algwidthb height 1pt}
\item[\bf Step 1.   ] Define $g\in\Z[x]$, $p,q\in\Z$, $\theta\in K$ by:
\begin{gather*}
g(x)=27A^2f\left(\frac{x-B}{3A}\right)=x^3-px-q\\
p=3(B^2-3AC)\\
q=-2B^3+9ABC-27A^2D\\
\theta=3A\alpha+B
\end{gather*}
(These are chosen so that $\theta$ is the real root of $g(x)$.)
Choose a positive integer $d$ such that $\O_K\subset \frac{1}{d}\Z[\theta]$
(this can be computed in
PARI/GP\footnote{By \verb|bnfinit(X^3-p*X-q).zk|}).

\item[] \hskip\algshiftb\makebox[\algwidthb]{\dotfill}
\item[If $B=0$:] Put $\theta=A\alpha$, $p=-AC$, $q=-A^2D$, $g(x)=A^2f(x/A)=x^3-px-q$


\item[] \hskip\algshiftb\vbox{\hrule width\algwidthb height 1pt}
\item[\bf Step 2.   ] Find fundamental unit $\ep_0$ of $\O_K$
(can be computed in
PARI/GP\footnote{By \verb|bnfinit(X^3-p*X-q).fu|}),
 and put
\begin{equation*}
\lambda=\max\{\pm\ep_0,\pm{\ep_0}^{-1}\}
\end{equation*}


\item[] \hskip\algshiftb\vbox{\hrule width\algwidthb height 1pt}
\item[\bf Step 3.   ] Define the sequences $\{a_n\}$, $\{b_n\}$, $\{c_n\}$, $\{m_n\}$ by:
\begin{gather*}
a_n+b_n\theta+c_n\theta^2 = \lambda^n\\
m_n=(9A^2sd)c_n
\end{gather*}
\item[] \hskip\algshiftb\makebox[\algwidthb]{\dotfill}
\item[If $B=0$:] Put $m_n=(A^2sd)c_n$.


\item[] \hskip\algshiftb\vbox{\hrule width\algwidthb height 1pt}
\item[\bf Step 4.   ] Define the constant $\phi$ and the sequence $\{Q_n\}$ by:
\begin{equation*}
\phi=\arctan
    \left(
    \frac{\sqrt{3\theta^2-4p}(b_1-c_1\theta)}{2(a_1+pc_1)-b_1\theta-c_1\theta^2}
    \right)
\end{equation*}
And $\{Q_n\}$ is the sequence of denominators of the {\em non-integer} convergents of $\phi/\pi$


\item[] \hskip\algshiftb\vbox{\hrule width\algwidthb height 1pt}
\item[\bf Step 5.   ] Define constants $C_1$, $C_2$, 
$\widetilde{M_\theta}$, $\widetilde{M_\alpha}$, $\widetilde{M_{\beta_1}}$, $\widetilde{M_{\beta_2}}$, $\widetilde{N}$ by:
\begin{align*}
C_1
&=
\max\left(
    \sqrt{2},
    \frac{\sqrt{2}|\theta|}{\sqrt{3\theta^2-4p}}\right
    ),\\
C_2
&=
\frac{\sqrt{d}}{\sqrt{3\theta^2-p}},\\
\widetilde{M_\theta}
&=
\max
\left\{
\frac{3 d\sqrt{2} C_1 C_2}{2|\theta|},\frac{3 d C_1 C_2}{2}
\right\}\\
\widetilde{M_\alpha}
&=
\max\left\{
    9A^2\widetilde{M_\theta},
    3A\widetilde{M_\theta}(1+2|B|)
    \right\},\\
\widetilde{M_{\beta_1}}
&=
s^{3/2}(|r_1|+|r_2|)\widetilde{M_\alpha}\\
\widetilde{M_{\beta_2}}
&=
9A^2s^{3/2}
\cdot
\frac{3dC_2\pi}{2\sqrt{3\theta^2-4p}}
\\
\widetilde{N}
&=
\max
    \Bigg\{
    \lambda^{1/2},
    \frac{2\sqrt{2}d C_1}{|\theta|},
    2dC_1,
    \left(2(1+\sqrt{2})C_2\right)^{1/3},
    12A\widetilde{M_\theta},\\
    &\qquad\qquad\qquad\qquad\qquad\qquad\qquad
    16|B|\widetilde{M_\theta},
    \frac{8\widetilde{M_\alpha}|r_1|}{3A},
    \frac{8\widetilde{M_\alpha}|r_2|}{3A},
    \frac{4s\widetilde{M_\alpha}}{3A}
    \Bigg\}
\end{align*}

\item[] \hskip\algshiftb\makebox[\algwidthb]{\dotfill}
\item[If $B=0$:] Put
\begin{align*}
\widetilde{M_\alpha}
&=
A^2\widetilde{M_\theta},\\
\widetilde{M_{\beta_2}}
&=
A^2s^{3/2}
\cdot
\frac{3dC_2\pi}{2\sqrt{3\theta^2-4p}}
,\\
\widetilde{N}
&=
\max
    \Bigg\{
    \lambda^{1/2},
    \frac{2\sqrt{2}d C_1}{|\theta|},
    2dC_1,
    \left(2(1+\sqrt{2})C_2\right)^{1/3},
    4A\widetilde{M_\theta},\\
    &\qquad\qquad\qquad\qquad\qquad\qquad\qquad
    \frac{8\widetilde{M_\alpha}|r_1|}{A},
    \frac{8\widetilde{M_\alpha}|r_2|}{A},
    \frac{4s\widetilde{M_\alpha}}{A}
    \Bigg\}.
\end{align*}


\item[] \hskip\algshiftb\vbox{\hrule width\algwidthb height 1pt}
\item[\bf Step 6.   ] Define the sequence $\{\psi_n\}$ and constants $C_0$, $n_0$ by:
\begin{gather*}
\psi_n
=
m_{Q_n}\\
C_0
=
\widetilde{M_{\beta_1}} \cdot\widetilde{M_{\beta_2}}\\
n_0
=
2\log_\lambda \widetilde{N}
\end{gather*}


\item[] \hskip\algshiftb\vbox{\hrule width\algwidthb height 1pt}
\item[\bf Result.   ] For $Q_n>n_0$
\begin{equation*}
\psi_n\|\psi_n\alpha\|\|\psi_n\beta\|
\leq
\frac{C_0}{Q_{n+1}}.
\end{equation*}

\item[] \hskip\algshifta\vbox{\hrule width\algwidtha height 2pt}
\end{itemize}
\end{itemize}

\begin{rem} As we stated earlier, the purpose of modifying the algorithm for the case $B=0$ is to sharpen our bounds. It will reduce $C_0$ by a factor of 81, while at worst it will increase $n_0$ by up to $\log_\lambda 9$. (For instance, in Example~1 it reduces $C_0$ from $87.46$ to $1.08$ and decreases $n_0$ from $4.85$ to $3.22$.)
\end{rem}

%% file: disEx0.0.2.tex
\subsection{Example 1: $\theta^3=2$}
In Example~1 from from Chapter~\ref{chap:intro} (Section~\ref{intro:ex1}), we considered the pair $(\theta,\theta^2)$ for $\theta=2^{1/3}$. So we take $\alpha=\theta$, $f(x)=x^3-2$, and $\beta=\frac{0}{1}+\frac{0}{1}\theta+\frac{1}{1}\theta^2$. Therefore we have $A=1$, $B=0$, $r_0=r_1=0$, $r_2=s=1$. (We remark that the discriminant of $K$ is $-108$.)

\begin{itemize}
  \item[]
  \begin{itemize}
\item[\bf Step 1. ] We already have $\theta=\alpha=2^{1/3}$, and $f(x)$ is already in the form $x^3-px-q$ (with $p=0$ and $q=2$). Using PARI/GP, we see that $\O_K=\Z[\theta]$, so we put $d=1$.
\item[\bf Step 2. ] Using PARI/GP, we find that a fundamental unit is
$
\ep_0=\theta-1\approx 0.2599.
$
Since $0<\ep_0<1$, we put
\begin{equation*}
\lambda=\ep_0^{-1}=1+\theta+\theta^2\approx 3.8473.
\end{equation*}
\item[\bf Step 3. ] Define the sequences $\{a_n\}$, $\{b_n\}$, $\{c_n\}$, $\{m_n\}$ by
\begin{equation*}
a_n+b_n\theta+c_n\theta^2=\lambda^n
\end{equation*}
and $m_n=(A^2sd)c_n=c_n$. (We saw the points $({m_n^{1/2}}\langle m_n\theta\rangle,{m_n^{1/2}}\langle m_n\theta^2\rangle)$ in Figures~\ref{fig:LWseqex1long} and~\ref{fig:LWseqex1short}.)
\item[\bf Step 4. ] Put
\begin{equation*}
\phi
=
\arctan
\left(
\frac{\sqrt{3}\cdot\theta(1-\theta)}{2-\theta-\theta^2}
\right)
\approx
0.5899.
\end{equation*}
The first few (non-integer) convergents of $\phi/\pi\approx 0.18777$ are:
\begin{equation*}
\frac{1}{5},\ \frac{3}{16},\ \frac{43}{229},\ \frac{1551}{8260},\ \frac{1594}{8489},\ \frac{3145}{16749},\ \frac{4739}{25238},\ \frac{64752}{344843},\ \frac{198995}{1059767},\ \frac{263747}{1404610}
\end{equation*}
\item[\bf Step 5. ] We have the following values for the constants:
\begin{center}
\begin{tabular}{c c c c c}
$C_1$ &$=$ & $\sqrt{2}$ &$\approx$  &1.41421\\
$C_2$ &$=$ & $\frac{1}{\theta\sqrt{3}}$ & $\approx$ &0.458243\\
$\widetilde{M_\theta}$ &$=$ &$\frac{\sqrt{3}}{\theta^2}$ &$\approx$ &1.09112\\
$\widetilde{M_\alpha}$ & $=$ & $M_\theta$\\
$\widetilde{M_{\beta_1}}$ & $=$ & $M_\alpha$\\
$\widetilde{M_{\beta_2}}$ &= &$\frac{\pi}{2\theta^2}$ &$\approx$ &0.989540
\end{tabular}
\end{center}
We also have
\begin{align*}
\widetilde{N}
&=
\max
\left\{
\lambda^{1/2},
\frac{4}{\theta},
2\sqrt{2},
\left(\frac{2(1+\sqrt{2})}{\theta\sqrt{2}}\right)^{1/3},
\frac{4\sqrt{3}}{\theta^2},
0,
\frac{8\sqrt{3}}{\theta^2},
\frac{4\sqrt{3}}{\theta^2}
\right\}\\
&\approx
\max\{
1.961,3.175,2.828,1.303,4.364,0,8.729,4.364
\}\\
&= 8.729.
\end{align*}
\item[\bf Step 6. ] We put
\begin{center}
\begin{tabular}{c c c c c}
$C_0$ &$=$ & $\frac{\sqrt{3}\cdot \pi}{4\theta}$ &$\approx$  & 1.07971\\
$n_0$ &$=$ & $2\log_\lambda\frac{8\sqrt{3}}{\theta^2}$ & $\approx$ &3.216
\end{tabular}
\end{center}
and we define $\{\psi_n\}$ by $\psi_n=m_{Q_n}$. The first several terms are:
\begin{gather*}
m_5,\ m_{16},\ m_{229},\ m_{8260},\ m_{8489},\ m_{16749},\ m_{25238},\ m_{344843},\ m_{1059767},\ c_{1404610}.
\end{gather*}
\item[\bf Result. ] For $Q_n\geq 4>n_0$,
\begin{equation*}
\psi_n\|\psi_n\theta\|\|\psi_n\theta\|
<
\frac{C_0}{Q_{n+1}}
\approx
\frac{1.07971}{Q_{n+1}}.
\end{equation*}
In Table~\ref{tab:LWex0.0.2} we list the first several values of $\frac{C_0}{Q_{n+1}}$ and compare it to the actual values of $\psi_n\|\psi_n\theta\|\|\psi_n\theta^2\|$ (Even with \verb|$MaxExtraPrecision| set to $10^{250,000,000}$, Mathematica~8.0 was unable to calculate this for $n=12$.)
\end{itemize}
\end{itemize}

\begin{table}[htb]
\centering
\begin{tabular}{p{1.9 cm} p{5.3 cm} p{2.9 cm} p{2.8 cm}}\toprule
$Q_n$ & $\psi_n$ $(=m_{Q_n})$ & $\psi_n\|\psi_n\theta\|\|\psi_n\theta^2\|$ & $C_0/Q_{n+1}$\\
\midrule
5  &  177  &  $0.0320119$  &  $0.0674819$\\
16  &  483870160  &  $0.00265565$  &  $0.00471489$\\
229  &  $210 \dots 617$     (134 digits)  &  $0.000071868$  &  $0.000130716$\\
8260  &  $539 \dots 526$     (4833 digits)  &  $0.0000473553$  &  $0.000127189$\\
8489  &  $540 \dots 581$     (4967 digits)  &  $0.0000245330$  &  $0.0000644642$\\
16749  &  $138 \dots 317$     (9801 digits)  &  $0.0000228153$  &  $0.0000427811$\\
25238  &  $357 \dots 401$     (14768 digits)  &  $1.72115*10^{-6}$  &  $3.13102*10^{-6}$\\
344843  &  $141 \dots 112$     (201788 digits)  &  $4.38682*10^{-7}$  &  $1.01882*10^{-6}$\\
1059767  &  $109 \dots 020$     (620132 digits)  &  $4.05112*10^{-7}$  &  $7.6869*10^{-7}$\\
1404610  &  $740 \dots 203$     (821919 digits)  &  $3.35689*10^{-8}$  &  $6.02682*10^{-8}$\\
17915087  &  $405 \dots 320$ (10483166) digits  & $2.28507*10^{-9}$ & $4.2809*10^{-9}$ \\
252215828 & $743 \dots 566$ (147586247) digits  & & $3.99699*10^{-9}$\\
\bottomrule
\end{tabular}
\caption{$\psi_n\|\psi_n\theta\|\|\psi_n\theta^2\|$ and $C_0/Q_{n+1}$ for $\theta=\sqrt[3]{2}$}
\label{tab:LWex0.0.2}
\end{table}


%% file: disEx0.1.1.tex

\subsection{Example 2: $\theta^3=\theta+1$} 
Let $\theta\approx 1.32472$ be the real root of $f(x)=x^3-x-1$, and consider the pair $(\theta,\theta^2)$. (We chose this example to have a unit $\lambda$ close to 1, so that we can calculate the Littlewood product $\psi_n\|\psi_n\theta\|\|\psi_n\theta^2\|$ for larger $n$ than in other examples.) We take $\alpha=\theta$ and $\beta=\theta^2$, so (as in Example~1) put $A=1$, $B=0$, $r_0=r_1=0$, $r_2=s=1$. The discriminant of $K$ is $-23$.

\begin{itemize}
  \item[]
\begin{itemize}
\item[\bf Step 1. ] We already have $\theta=\alpha$, and $f(x)$ is already in the form $x^3-px-q$ (with $p=1$ and $q=1$). Using PARI/GP, we see that $\O_K=\Z[\theta]$, so we put $d=1$.
\item[\bf Step 2. ] Now $\theta$ is already a unit, since the constant term of $f(x)$ is $-1$. So since $\theta>1$, we will put
\begin{equation*}
\lambda=\theta\approx 1.32472.
\end{equation*}
\item[\bf Step 3. ] Define the sequences $\{a_n\}$, $\{b_n\}$, $\{c_n\}$, $\{m_n\}$ by
\begin{equation*}
a_n+b_n\theta+c_n\theta^2=\lambda^n
\end{equation*}
and $m_n=(A^2sd)c_n=c_n$. (We plot the points $({m_n^{1/2}}\langle m_n\theta\rangle,{m_n^{1/2}}\langle m_n\theta^2\rangle)$ for $n=1$ to~200 in Figure \ref{fig:LWseqex0.1.1}.)
\item[\bf Step 4. ] Put
\begin{equation*}
\phi
=
\arctan
\left(
\frac{\sqrt{3\theta^2-4}}{-\theta}
\right)
\approx
-0.703858.
\end{equation*}
The first few convergents of $\phi/\pi\approx -0.224045$ are:
\begin{equation*}
-\frac{1}{4}, -\frac{2}{9}, -\frac{13}{58}, -\frac{41}{183}, -\frac{1038}{4633}, -\frac{1079}{4816}, -\frac{2117}{9449}, -\frac{15898}{70959}, -\frac{18015}{80408}, -\frac{63969148}{285519359}
\end{equation*}
\item[\bf Step 5. ] We have

\begin{tabular}[t]{rcl rcl rcl}
$C_1$  & $\approx$ &1.66593
\qquad\qquad
&$C_2$  & $\approx$ &0.484238
\qquad\qquad
&$\widetilde{M_\theta}$  & $\approx$ &1.29181\\
$\widetilde{M_\alpha}$  & $\approx$ &1.29181
&$\widetilde{M_{\beta_1}}$  & $\approx$ &1.29181
&$\widetilde{M_{\beta_2}}$  & $\approx$ & 2.02917
\end{tabular}

and
\begin{align*}
\widetilde{N}
&\approx
\max\{ 1.151, 3.557, 3.332, 1.327, 5.167, 0, 10.334, 5.167 \}
=
10.334
\end{align*}
\item[\bf Step 6. ] We put
\begin{align*}
C_0  & \approx 2.6213\\
n_0   & \approx 16.6109
\end{align*}
and we define $\{\psi_n\}$ by $\psi_n=m_{Q_n}$. The first several terms are:
\begin{gather*}
m_{4},\  m_{9},\  m_{58},\  m_{183},\  m_{4633},\  m_{4816},\  m_{9449},\  m_{70959},\  m_{80408},\  m_{285519359}
\end{gather*}
\item[\bf Result. ] For $Q_n\geq 17>n_0$,
\begin{equation*}
\psi_n\|\psi_n\theta\|\|\psi_n\theta\|
<
\frac{C_0}{Q_{n+1}}
\approx
\frac{2.6213}{Q_{n+1}}.
\end{equation*}
In Table~\ref{tab:LWseqex0.1.1} we compute the bound $\frac{C_0}{Q_{n+1}}$ and compare it to the actual values of $\psi_n\|\psi_n\theta\|\|\psi_n\theta^2\|$  for the first few $n$.
\end{itemize}
\end{itemize}

\begin{figure}[h!tb]
    \centering
        \includegraphics[width=\figwidth\textwidth
        ]
        {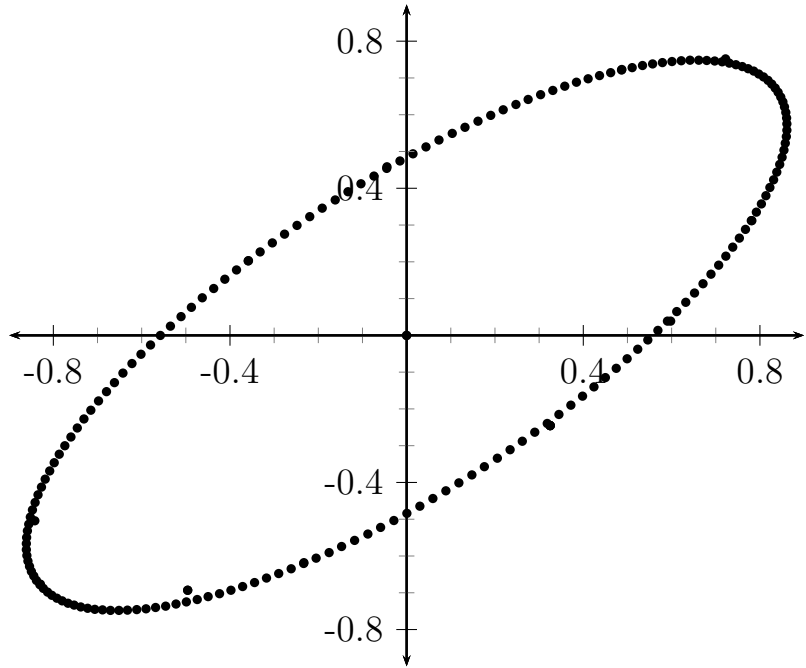}
    \caption{
    $({m_n^{1/2}}\langle m_n\theta\rangle,{m_n^{1/2}}\langle m_n\theta^2\rangle)$ for $\theta^3=\theta+1$, $n\leq 200$.
    }
    \label{fig:LWseqex0.1.1}
\end{figure}

\begin{table}[h!tb]
\centering
\begin{tabular}{p{1.9 cm} p{5.1 cm} p{2.9 cm} p{2.8 cm}}\toprule
$Q_n$ & $\psi_n$ $(=m_{Q_n})$ & $\psi_n\|\psi_n\theta\|\|\psi_n\theta^2\|$ & $C_0/Q_{n+1}$\\
\midrule
4  &  1  &  0.079596  &  0.291255\\
9  &  3  &  0.0205192  &  0.0451948\\
58  &  2839729  &  0.0072113  &  0.014324\\
183  &  5232446865180756766896  &  0.000276774  &  0.000565788\\
4633  &  $147 \dots 145$     (566 digits)  &  0.000146703  &  0.000544289\\
4816  &  $328 \dots 896$     (588 digits)  &  0.000130300  &  0.000277415\\
9449  &  $206 \dots 593$     (1154 digits)  &  0.0000162949  &  0.000036941\\
70959  &  $133 \dots 906$     (8666 digits)  &  0.0000162888  &  0.000032599\\
80408  &  $117 \dots 377$     (9820 digits)  &  $4.58789*10^{-9}$  &  $9.18080*10^{-9}$\\
285519359  &  $891 \dots 001$  (34868601 digits)  & $2.53019*10^{-9}$ &  $9.17822*10^{-9}$  \\ 
285599767  &  $447 \dots 577$     (34878421 digits)  &  $2.05770*10^{-9}$ &  $4.58976*10^{-9}$  \\ 
571119126  &  $169 \dots 901$     (69747023 digits)  & $4.72492*10^{-10}$ &  $1.01993*10^{-9}$  \\ 
\bottomrule
\end{tabular}
\caption{$\psi_n\|\psi_n\theta\|\|\psi_n\theta^2\|$ and $C_0/Q_{n+1}$, where $\theta^3=\theta+1$}
\label{tab:LWseqex0.1.1}
\end{table}

%% file: disEx7.0.2.tex

\subsection{Example 3: $\alpha^3=7\alpha^2+2$}
Consider the pair $(\alpha,\alpha^{-1}$), where $\alpha$ is the real root of $f(x)=x^3-7x^2-2$. This field has discriminant $-2852$. Since
\begin{gather*}
\alpha^3-7\alpha^2=2,\\
\alpha
\left(
\frac{\alpha^2-7\alpha}{2}
\right)
=1,
\end{gather*}
we can write $\alpha^{-1}$ in the basis $1,\alpha,\alpha^2$ as
\begin{equation*}
\alpha^{-1}=-\frac{7}{2}\alpha+\frac{1}{2}\alpha^2
\end{equation*}
(i.e., we have $r_0=0$, $r_1=-7$, $r_2=1$, $s=2$).

\begin{itemize}
  \item[]
\begin{itemize}
\item[\bf Step 1. ] We have $g(x)=x^3-147x-740$ (in particular, $p=147$), and we put $\theta=3\alpha-7\approx 14.121$. Using PARI/GP, we find that $\O_k\subset\frac{1}{9}\Z[\theta]$ (so we take $d=9$).
\item[\bf Step 2. ] Using PARI/GP, we find that a fundamental unit is
\begin{equation*}
\ep_0=-\frac{4309}{9}- \frac{62}{9}\theta + \frac{26}{9}\theta^2\approx-0.0000108758.
\end{equation*}
Then
\begin{equation*}
\lambda
=
\max\{\pm\ep_0,\pm\ep_0^{-1}\}
=
-\ep_0^{-1}
=
\frac{96109}{9}
+
\frac{25898}{9}\theta
+
\frac{1834}{9}\theta^2
\approx
91946.994
\end{equation*}
\item[\bf Step 3. ] Define the sequences $\{a_n\}$, $\{b_n\}$, $\{c_n\}$ as before, and define $\{m_n\}$ by
 $m_n=(9sd)c_n=162 c_n$. (We plot the points $({m_n^{1/2}}\langle m_n\theta\rangle,{m_n^{1/2}}\langle m_n\theta^2\rangle)$ for $n=1$ to~100 in Figure~\ref{fig:LWseqex7.0.2}.)
\item[\bf Step 4. ] Put
\begin{equation*}
\phi
=
\arctan\left(
\frac{\sqrt{3\theta^2-588}\cdot(25898-1834\cdot\theta)}
{731414-25898\cdot\theta-1834\cdot\theta^2}
\right)
 \approx -0.253982
\end{equation*}
The first few convergents of $\phi/\pi\approx -0.0808448$ are:
\begin{equation*}
-\frac{1}{12}, -\frac{2}{25}, -\frac{3}{37}, -\frac{8}{99}, -\frac{19}{235}, -\frac{46}{569}, -\frac{111}{1373}, -\frac{1378}{17045}, -\frac{1489}{18418}, -\frac{7334}{90717}.
\end{equation*}
\item[\bf Step 5. ] We have

\begin{tabular}[t]{rcl rcl rcl}
$C_1$  & $\approx$ & 6.24920
\qquad\qquad
&$C_2$  & $\approx$ & 0.141231
\qquad\qquad
&$\widetilde{M_\theta}$  & $\approx$ & 11.9149\\
$\widetilde{M_\alpha}$  & $\approx$ & 536.169
&$\widetilde{M_{\beta_1}}$  & $\approx$ & 12132.1
&$\widetilde{M_{\beta_2}}$  & $\approx$ & 47.7138
\end{tabular}

and
\begin{align*}
\widetilde{N}
&\approx
10008.5
\end{align*}
\item[\bf Step 6. ] We have
\begin{align*}
C_0  & \approx 578869\\
n_0   & \approx 1.6119
\end{align*}
and we define $\{\psi_n\}$ by $\psi_n=m_{Q_n}$. The first several terms are:
\begin{gather*}
m_{12},\  m_{25},\  m_{37},\  m_{99},\  m_{235},\  m_{569},\  m_{1373},\  m_{17045},\  m_{18418},\  m_{90717},\  m_{381286}.
\end{gather*}
\item[\bf Result. ] For $Q_n\geq 2>n_0$,
\begin{equation*}
\psi_n\|\psi_n\alpha\|\|\psi_n\alpha^{-1}\|
<
\frac{C_0}{Q_{n+1}}
\approx
\frac{578869}{Q_{n+1}}.
\end{equation*}
In Table~\ref{tab:LWseqex7.0.2}
we compute the bound $\frac{C_0}{Q_{n+1}}$ and compare it to the actual values of $\psi_n\|\psi_n\theta\|\|\psi_n\theta^2\|$  for the first few $n$.
\end{itemize}
\end{itemize}

\begin{figure}[h!tb]
    \centering
        \includegraphics[width=\figwidth\textwidth
        ]
        {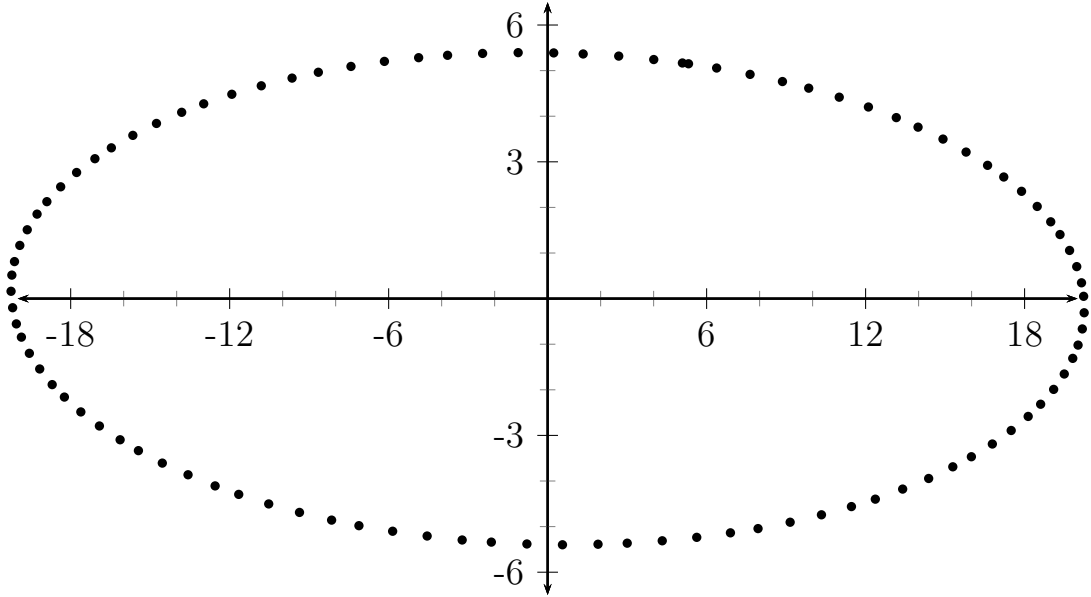}
    \caption{
    $(m_n^{1/2}\langle m_n\alpha\rangle,m_n^{1/2}\langle m_n\alpha^{-1}\rangle)$ for $\alpha^3=7\alpha^2+2$, $n\leq 100$.
    }
    \label{fig:LWseqex7.0.2}
\end{figure}

\begin{table}[h!tb]
\centering
\begin{tabular}{p{1.9 cm} p{5.1 cm} p{3.2 cm} p{2.8 cm}}\toprule
$Q_n$ & $\psi_n$ $(=m_{Q_n})$ & $\psi_n\|\psi_n\alpha\|\|\psi_n\alpha^{-1}\|$ & $C_0/Q_{n+1}$\\
\midrule
12  &  $131 \dots 120$     (60 digits)  &  10.2213  &  23154.8\\
25  &  $440 \dots 692$     (124 digits)  &  7.2066  &  15645.1\\
37  &  $160 \dots 332$     (184 digits)  &  3.00062  &  5847.17\\
99  &  $881 \dots 988$     (491 digits)  &  1.24736  &  2463.27\\
235  &  $969 \dots 892$     (1166 digits)  &  0.50306  &  1017.35\\
569  &  $642 \dots 132$     (2824 digits)  &  0.242022  &  421.609\\
1373  &  $310 \dots 052$     (6815 digits)  &  0.0188969  &  33.9612\\
17045  &  $112 \dots 812$     (84604 digits)  &  0.0152817  &  31.4295\\
18418  &  $975 \dots 496$     (91418 digits)  &  0.00361497  &  6.38105\\
90717  &  $615 \dots 348$     (450277 digits)  &  0.00082197  &  1.5182\\
381286  &  $939 \dots 264$     (1892526 digits)  & $0.000327067$  & 0.678398\\
853289  &  $379 \dots 748$     (4235331 digits)  & $0.000167840$ & 0.468882\\
1234575  &  $894 \dots 812$     (6127858 digits)  & $0.000159226$ & 0.277254\\
2087864  &  $851 \dots 344$     (10363190 digits)  & $8.6140*10^{-6}$ & 0.0149131\\
\bottomrule
\end{tabular}
\caption{$\psi_n\|\psi_n\alpha\|\|\psi_n\alpha^{-1}\|$ and $C_0/Q_{n+1}$, where $\alpha^3=7\alpha^2+2$}
\label{tab:LWseqex7.0.2}
\end{table}

%% file: disEx0.147.740.tex

\subsection{Example 4: $\theta^2=147\cdot\theta+740$}
Consider the pair $(\theta,\theta^2)$, where $\theta^3=147\cdot\theta+740$ as in Example~3. (The field is still the same as in Example~3, so the discriminant is still $-2852$.) We put $A=1$ and $B=0$, and (since we're taking $\beta=\theta^2$) $r_0=r_1=0$, $r_2=s=1$.

\begin{itemize}
  \item[]
\begin{itemize}
\item[\bf Step 1. ] We already have $\theta\approx 14.121$, and $f(x)$ is already in the form $x^3-px-q$ (with $p=147$ and $q=740$). We already found in Example~3 that we can take $d=9$.
\item[\bf Step 2. ] In the previous example, we found
\begin{equation*}
\lambda
=
\frac{96109}{9}
+
\frac{25898}{9}\theta
+
\frac{1834}{9}\theta^2
\approx
91946.994
\end{equation*}
\item[\bf Step 3. ] Define the sequences $\{a_n\}$, $\{b_n\}$, $\{c_n\}$ as before and define $\{m_n\}$ by
 $m_n=d c_n=9c_n$. (We plot the points $({m_n^{1/2}}\langle m_n\theta\rangle,{m_n^{1/2}}\langle m_n\theta^2\rangle)$ for $n=1$ to~100 in Figure~\ref{fig:LWseqex0.147.740}.)
\item[\bf Step 4. ] Since $\lambda$ is the same as in Example~3, so is $\phi$ (and therefore the sequence $\{Q_n\}$).

\item[\bf Step 5. ] We have

\begin{tabular}[t]{rcl rcl rcl}
$C_1$  &  $\approx$  &  $6.24920$
\qquad\qquad
& $C_2$  &  $\approx$  &  $0.141231$
\qquad\qquad
& $\widetilde{M_\theta}$  &  $\approx$  &  $11.9149$
\\
$\widetilde{M_\alpha}$  &  $\approx$  &  $11.9149$
& $\widetilde{M_{\beta_1}}$  &  $\approx$  &  $11.9149$
& $\widetilde{M_{\beta_2}}$  &  $\approx$  &  $1.87438$
\end{tabular}

and
\begin{align*}
\widetilde{N}
\approx
\max\{ 303.228, 11.27, 112.5, 0.8802, 47.66, 0, 95.32, 47.66 \}
=
\lambda^{1/2}
\end{align*}
\item[\bf Step 6. ] We have
\begin{gather*}
C_0   \approx 22.3329,\\
n_0    = 2\log_\lambda \widetilde{N}= 1.
\end{gather*}
and we define $\{\psi_n\}$ (as before) by $\psi_n=m_{Q_n}$.
\item[\bf Result. ] For $Q_n>n_0=1$,
\begin{equation*}
\psi_n\|\psi_n\theta\|\|\psi_n\theta\|
<
\frac{C_0}{Q_{n+1}}
\approx
\frac{22.3329}{Q_{n+1}}.
\end{equation*}
In Table~\ref{tab:LWseqex0.147.740}
we compute the bound $\frac{C_0}{Q_{n+1}}$ and compare it to the actual values of $\psi_n\|\psi_n\theta\|\|\psi_n\theta^2\|$  for the first few $n$.
\end{itemize}
\end{itemize}

\begin{figure}[h!tb]
    \centering
        \includegraphics[width=\figwidth\textwidth
        ]
        {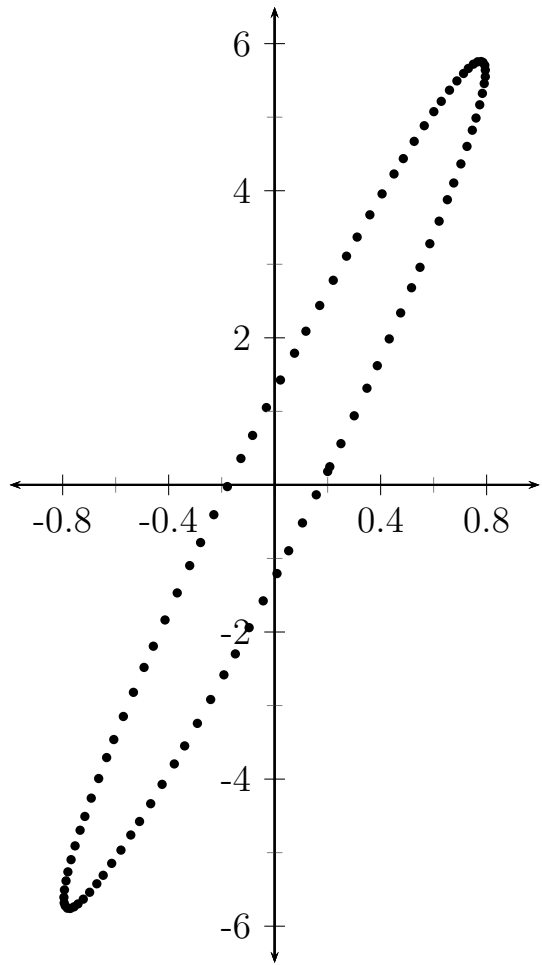}
    \caption{
    $(m_n^{1/2}\langle m_n\theta\rangle,m_n^{1/2}\langle m_n\theta^2\rangle)$ for $\theta^3=147\cdot\theta+740$,
    $n\leq 100$.
    }
    \label{fig:LWseqex0.147.740}
\end{figure}

\begin{table}[h!tb]
\centering
\begin{tabular}{p{1.9 cm} p{5.1 cm} p{2.9 cm} p{2.8 cm}}\toprule
$Q_n$ & $\psi_n$ $(=m_{Q_n})$ & $\psi_n\|\psi_n\theta\|\|\psi_n\theta^2\|$ & $C_0/Q_{n+1}$\\
\midrule
12  &  $728 \dots 840$     (58 digits)  &  $0.133515$  &  $0.893317$\\
25  &  $244 \dots 594$     (123 digits)  &  $0.0472532$  &  $0.603593$\\
37  &  $892 \dots 074$     (182 digits)  &  $0.0311235$  &  $0.225585$\\
99  &  $489 \dots 166$     (490 digits)  &  $0.0109711$  &  $0.0950337$\\
235  &  $538 \dots 494$     (1165 digits)  &  $0.00475198$  &  $0.0392494$\\
569  &  $357 \dots 674$     (2823 digits)  &  $0.00221919$  &  $0.0162658$\\
1373  &  $172 \dots 114$     (6814 digits)  &  $0.000175104$  &  $0.00131023$\\
17045  &  $626 \dots 434$     (84602 digits)  &  $0.000141411$  &  $0.00121256$\\
18418  &  $542 \dots 972$     (91417 digits)  &  $0.0000334768$  &  $0.000246182$\\
90717  &  $342 \dots 186$     (450276 digits)  & $7.6106*10^{-6}$ & 0.0000585726\\
381286  &  $469 \dots 132$     (1892526 digits)  & $3.02844*10^{-6}$ & 0.0000261728\\
853289  &  $189 \dots 874$     (4235331 digits)  & $1.55407*10^{-6}$ & 0.0000180896\\
1234575  &  $447 \dots 406$     (6127858 digits)  & $1.47433*10^{-6}$ & 0.0000106965\\
2087864  &  $425 \dots 672$     (10363190 digits)  & $7.9760*10^{-8}$ & $5.75352*10^{-7}$\\
\bottomrule
\end{tabular}
\caption{$\psi_n\|\psi_n\theta\|\|\psi_n\theta^2\|$ and $C_0/Q_{n+1}$, where $\theta^3=147\cdot\theta+740$}
\label{tab:LWseqex0.147.740}
\end{table}

%% file: disEx0.8.10.tex

\subsection{Example  5: $\theta^2=8\cdot\theta+10$}

Let $\theta$ be the real root of $x^3-8x-10$. (This cubic (a ``miracle cubic'', as D.~H.~Lehmer called it) was brought to my attention by John Brillhart, who discovered in 1964 that $\theta$ has several unusually large partial quotients very early in its continued fraction expansion. For example, $a_{17}=22986$, $a_{33}=1501790$, and $a_{121}=16467250$. This is related to the fact that the discriminant of the polynomial is $-652=-4\cdot163$ and that $\Q(\sqrt{-163})$ has class number 1. See \cite{bib:St}.)\\

We consider the pair $(\theta,\theta^2)$. As in Examples~1 and~2, we take $\alpha=\theta$ and $\beta=\theta^2$ (so $A=1$, $B=0$, $r_0=r_1=0$, $r_2=s=1$).

\begin{itemize}
  \item[]
\begin{itemize}
\item[\bf Step 1. ] We already have $\theta=\alpha$, and $f(x)$ is already in the form $x^3-px-q$ (with $p=8$ and $q=10$). Using PARI/GP, we see that $\O_K=\Z[\theta]$, so we put $d=1$.
\item[\bf Step 2. ] Using PARI/GP, we find that a fundamental unit is
\begin{equation*}
\ep_0
=\theta^2-11
\approx
0.0132932
\end{equation*}
Then (since $0<\ep_0<1$) we put
\begin{equation*}
\lambda
=
\ep_0^{-1}
=
9+10\cdot\theta+3\cdot\theta^2
\approx
75.2262.
\end{equation*}
\item[\bf Step 3. ] Define $\{a_n\}$, $\{b_n\}$, $\{c_n\}$ as before and define $\{m_n\}$ by
 $m_n=c_n$. (We plot the points $(m_n^{1/2}\langle m_n\theta\rangle,m_n^{1/2}\langle m_n\theta^2\rangle)$ for $n=1$ to~1000 in Figure~\ref{fig:LWseqex0.8.10}.)
\item[\bf Step 4. ] Put
\begin{equation*}
\phi
=
\arctan\left(
\frac{\sqrt{3\theta^2-32}\cdot(10-3\cdot\theta)}
{66-10\cdot\theta-3\cdot\theta^2}
\right)
 \approx
 -0.196350
\end{equation*}
Curiously, the continued fraction of $\phi/\pi\approx -0.062499998136$ also has an early large partial quotient. The continued fraction of $\phi/\pi$ is
\begin{equation*}
[-1;1, 15, 2095966, 30, 1, 2, 1, 1, 3, 1, 3, 1, 1, 1,\dots],
\end{equation*}
and the first few (non-integer) convergents of $\phi/\pi$ are:
\begin{equation*}
-\frac{1}{16}, -\frac{2095966}{33535457}, -\frac{62878981}{1006063726}, -\frac{64974947}{1039599183}, -\frac{192828875}{3085262092}, -\frac{257803822}{4124861275}.
\end{equation*}

\item[\bf Step 5. ] We have

\begin{tabular}[t]{rcl rcl rcl}
$C_1$  & $\approx$ & 4.60238
\qquad\qquad
&$C_2$  & $\approx$ & 0.199841
\qquad\qquad
&$\widetilde{M_\theta}$  & $\approx$ & 1.37961\\
$\widetilde{M_\alpha}$  & $\approx$ & 1.37961
&$\widetilde{M_{\beta_1}}$  & $\approx$ & 1.37961
&$\widetilde{M_{\beta_2}}$  & $\approx$ & 0.923493
\end{tabular}

and
\begin{align*}
\widetilde{N}
\approx
\max\{ 8.6733, 3.923, 9.205, 0.9882, 5.518, 0, 11.04, 5.518 \}
=
11.04
\end{align*}
\item[\bf Step 6. ] We have
\begin{align*}
C_0  & \approx 1.27406\\
n_0   & \approx 1.11156
\end{align*}
and we define $\{\psi_n\}$ by $\psi_n=m_{Q_n}$. The first several terms are:
\begin{gather*}
m_{16},\  m_{33535457},\  m_{1006063726},\  m_{1039599183},\  m_{3085262092},\  m_{4124861275},\  m_{7210123367}.
\end{gather*}
\item[\bf Result. ] For $Q_n\geq 2>n_0$,
\begin{equation*}
\psi_n\|\psi_n\theta\|\|\psi_n\theta\|
<
\frac{C_0}{Q_{n+1}}
\approx
\frac{1.27406}{Q_{n+1}}.
\end{equation*}
In Table~\ref{tab:LWseqex0.8.10}
we compute the bound $\frac{C_0}{Q_{n+1}}$ for several values of $n$ and compare it to the actual value of $\psi_n\|\psi_n\theta\|\|\psi_n\theta^2\|$  for $n=1$ and~2. (Since $Q_n$ is so large for $n>2$, calculating more values of $\psi_n\|\psi_n\theta\|\|\psi_n\theta^2\|$ requires more computing power than my circa-2009 computer has. On my desktop computer, Mathematica~8.0 takes about 45 minutes to calculate $\psi_n\|\psi_n\theta\|\|\psi_n\theta^2\|$ when $\psi_n$ has about 60~million digits, and reaches an overflow error when $\psi_n$ has somewhere between 60~million and 147~million digits.)

In Figure~\ref{fig:LWseqex0.8.10} we have the points $(m_n^{1/2}\langle m_n\theta\rangle,m_n^{1/2}\langle m_n\theta^2\rangle)$ for $n=1$ to~1000,  although there appear to be only 32 points. The reason is that since $Q_2$ is so large, $\phi$ (the angle of rotation of the sequence
$\{(\Im(\sqrt{\lambda}\lambda_2)^n,\Re(\sqrt{\lambda}\lambda_2)^n)\}$
)
is nearly $-\frac{\pi}{16}$. A transformation of the point $(\Im(\sqrt{\lambda}\lambda_2)^n,\Re(\sqrt{\lambda}\lambda_2)^n)$
yields the point
\begin{equation*}
(\lambda^{1/2}\langle m_n\theta\rangle,\lambda^{1/2}\langle m_n\theta^2\rangle),
\end{equation*}
and
\begin{equation*}
(m_n^{1/2}\langle m_n\theta\rangle,m_n^{1/2}\langle m_n\theta^2\rangle)
\approx
C'(\lambda^{1/2}\langle m_n\theta\rangle,\lambda^{1/2}\langle m_n\theta^2\rangle)
\end{equation*}
for some $C'$.
\end{itemize}
\end{itemize}

\begin{figure}[h!tb]
    \centering
        \includegraphics[width=\figwidth\textwidth
        ]
        {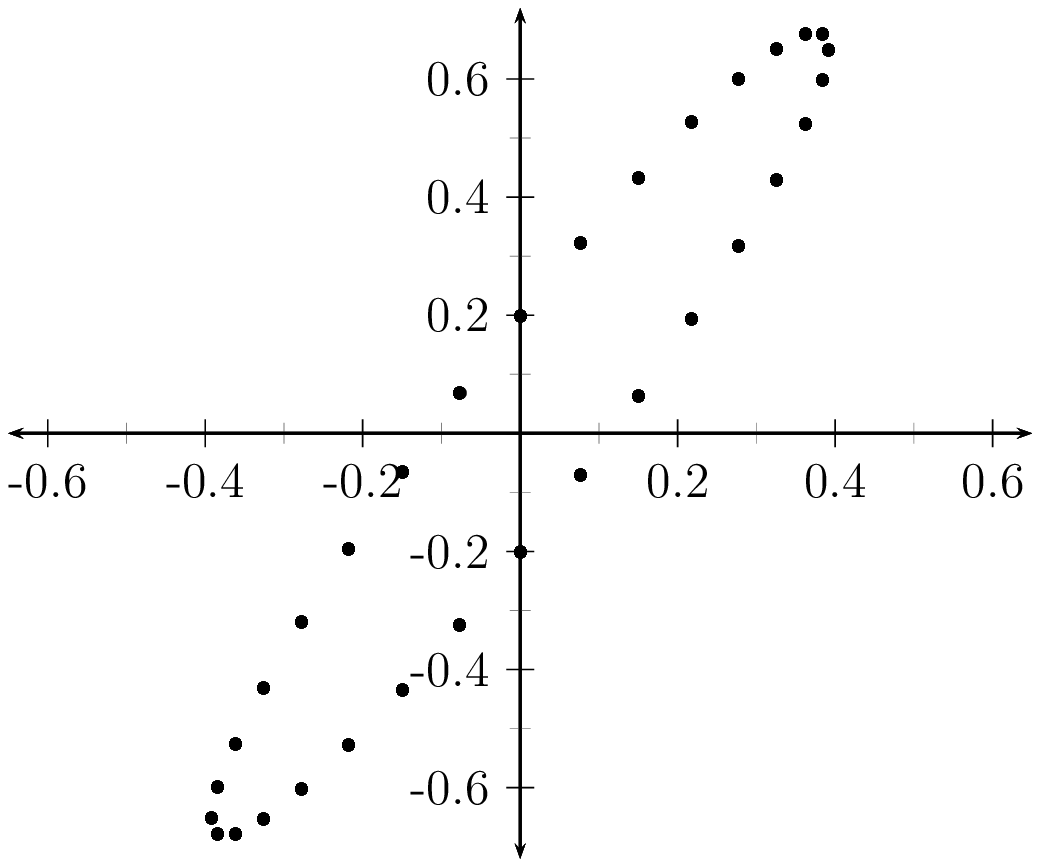}
    \caption{
    $(m_n^{1/2}\langle m_n\theta\rangle,m_n^{1/2}\langle m_n\theta^2\rangle)$ for
    $\theta^3=8\theta+10$, $n\leq 1000$.
    }
    \label{fig:LWseqex0.8.10}
\end{figure}

\begin{table}[h!tb]
\centering
\begin{tabular}{p{2.8 cm} p{4 cm} p{2.9 cm} p{2.8 cm}}\toprule
$Q_n$ & $\psi_n$ $(=m_{Q_n})$ & $\psi_n\|\psi_n\theta\|\|\psi_n\theta^2\|$ & $C_0/Q_{n+1}$\\
\midrule
16  &  $420 \dots 880$     (29 digits)  &  $7.3376 * 10^{-9}$  &  $3.79915*10^{-8}$\\
33 535 457  &  (63 million digits)  &  $2.38861*10^{-10}$ & $1.26638*10^{-9}$\\ 
1 006 063 726  &  (1.89 billion digits)  &  & $1.22553*10^{-9}$\\
1 039 599 183  &  (1.95 billion digits)  &  & $4.12951*10^{-10}$\\
3 085 262 092  &  (5.8 billion digits)  &  & $3.08874*10^{-10}$\\
4 124 861 275  &  (7.7 billion digits)  &  & $1.76705*10^{-10}$\\
7 210 123 367  &  (13.5 billion digits)  &  & $4.94681*10^{-11}$\\
25 755 231 376  &  (48 billion digits)  &  & $3.86485*10^{-11}$\\
32 965 354 743  &  (62 billion digits)  &  & $1.0221*10^{-11}$\\
124 651 295 605  &  (234 billion digits)  &  & $8.0833*10^{-12}$\\
\bottomrule
\end{tabular}
\caption{$\psi_n\|\psi_n\theta\|\|\psi_n\theta^2\|$ and $C_0/Q_{n+1}$, where $\theta^3=8\theta+10$}
\label{tab:LWseqex0.8.10}
\end{table}

%% file: disEx0.-1.1.tex

\subsection{Example 6: $\theta^3=-\theta+1$}
We consider the pair $(\theta,\theta^2)$, so $A=1$, $B=0$, $r_0=r_1=s=0$, $r_2=s=1$. The discriminant of $K$ is $-31$.

\begin{itemize}
  \item[]
\begin{itemize}
\item[\bf Step 1. ] We already have $f(x)=x^3+x-1$ in the form $x^3-px-q$ (with $p=-1$ and $q=1$). Using PARI/GP, we see that $\O_K=\Z[\theta]$, so we put $d=1$.
\item[\bf Step 2. ] Since $\theta\approx 0.6823$ is already a unit, and since $0<\theta<1$, we put
\begin{equation*}
\lambda=\theta^{-1}=\theta^2+1\approx 1.46557.
\end{equation*}
\item[\bf Step 3. ] Define the sequences $\{a_n\}$, $\{b_n\}$, $\{c_n\}$ as before and define $\{m_n\}$ by
 $m_n=c_n$. (We plot the points $\{{m_n^{1/2}}\langle m_n\theta\rangle,{m_n^{1/2}}\langle m_n\theta^2\rangle\}$ for $n=1,\dots,500$ in Figure \ref{fig:LWseqex0.-1.1}.)
\item[\bf Step 4. ] Put
\begin{equation*}
\phi
=
\arctan\left(
\frac{\sqrt{3\theta^2+4}}
{\theta}
\right)
 \approx 1.28511,
\end{equation*}
The first few convergents of $\phi/\pi \approx 0.409065$ are:
\begin{equation*}
\frac{1}{2}, \frac{2}{5}, \frac{9}{22}, \frac{704}{1721}, \frac{7753}{18953}, \frac{16210}{39627}, \frac{202273}{494477}, \frac{825302}{2017535}, \frac{81907171}{200230442}, \frac{246546815}{602708861}
\end{equation*}
\item[\bf Step 5. ] We have

\begin{tabular}[t]{rcl rcl rcl}
$C_1$  &  $\approx$  &  $1.41421$
\qquad\qquad
& $C_2$  &  $\approx$  &  $0.645940$
\qquad\qquad
& $\widetilde{M_\theta}$  &  $\approx$  &  $2.84001$
\\
$\widetilde{M_\alpha}$  &  $\approx$  &  $2.84001$
& $\widetilde{M_{\beta_1}}$  &  $\approx$  &  $2.84001$
& $\widetilde{M_{\beta_2}}$  &  $\approx$  &  $1.31029$
\end{tabular}

and
\begin{align*}
\widetilde{N}
\approx
\max\{ 1.21061, 5.862, 2.828, 1.461, 11.36, 0, 22.72, 11.36 \}
=
22.72
\end{align*}
\item[\bf Step 6. ] We have
\begin{align*}
C_0  & \approx 3.72125\\
n_0   & \approx 16.3416
\end{align*}
and we define $\{\psi_n\}$ by $\psi_n=m_{Q_n}$. The first several terms are:
\begin{gather*}
m_{2},\  m_{5},\  m_{22},\  m_{1721},\  m_{18953},\  m_{39627},\  m_{494477},\  m_{2017535},\  m_{200230442},\  m_{602708861}
\end{gather*}
\item[\bf Result. ] For $Q_n\geq 17>n_0$,
\begin{equation*}
\psi_n\|\psi_n\theta\|\|\psi_n\theta\|
<
\frac{C_0}{Q_{n+1}}
\approx
\frac{3.72125}{Q_{n+1}}.
\end{equation*}
In Table~\ref{tab:LWseqex0.-1.1}
we compute the bound $\frac{C_0}{Q_{n+1}}$ and compare it to the actual values of $\psi_n\|\psi_n\theta\|\|\psi_n\theta^2\|$  for the first few $n$.
\end{itemize}
\end{itemize}

\begin{figure}[h!tb]
    \centering
        \includegraphics[width=\figwidth\textwidth
        ]
        {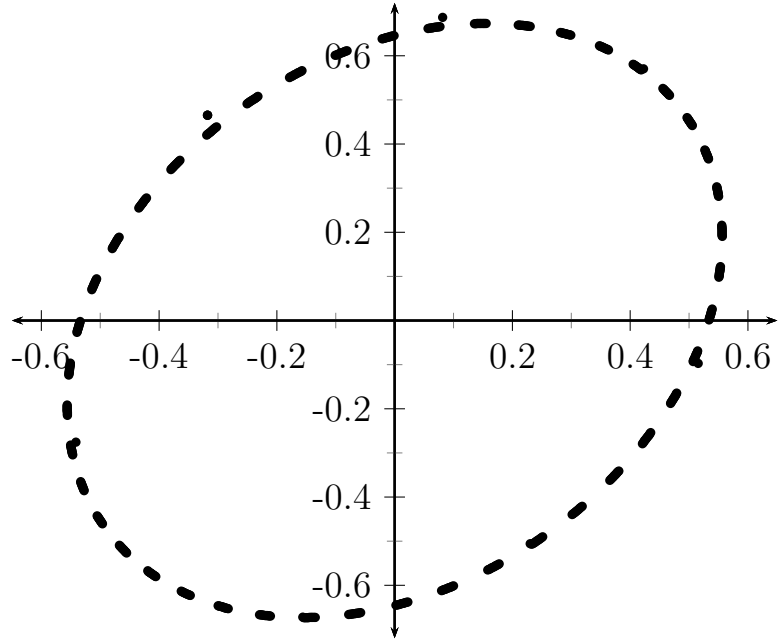}
    \caption{
    $(\sqrt{m_n}\langle m_n\theta\rangle,\sqrt{m_n}\langle m_n\theta^2\rangle)$ for
    $\theta^3=-\theta+1$, $n\leq 500$
    }
    \label{fig:LWseqex0.-1.1}
\end{figure}
\begin{table}[h!tb]
\centering
\begin{tabular}{p{1.9 cm} p{5.1 cm} p{2.9 cm} p{2.8 cm}}\toprule
$Q_n$ & $\psi_n$ $(=m_{Q_n})$ & $\psi_n\|\psi_n\theta\|\|\psi_n\theta^2\|$ & $C_0/Q_{n+1}$\\
\midrule
2  &  1  &  $0.147899$  &  $0.744249$\\
5  &  3  &  $0.055917$  &  $0.169148$\\
22  &  1873  &  $0.00065464$  &  $0.00216226$\\
1721  &  $208 \dots 364$     (286 digits)  &  $0.000057056$  &  $0.000196341$\\
18953  &  $890 \dots 536$     (3146 digits)  &  $0.0000274071$  &  $0.0000939069$\\
39627  &  $948 \dots 318$     (6578 digits)  &  $2.23749*10^{-6}$  &  $7.52562*10^{-6}$\\
494477  &  $169 \dots 664$     (82087 digits)  &  $5.5797*10^{-7}$  &  $1.84445*10^{-6}$\\
2017535  &  $259 \dots 473$     (334925 digits)  &  $5.6172*10^{-9}$  &  $1.85848*10^{-8}$\\
200230442  &  $781 \dots 124$     (33239641 digits)  &  $1.86084*10^{-9}$  &  $6.17420*10^{-9}$\\
\bottomrule
\end{tabular}
\caption{$\psi_n\|\psi_n\theta\|\|\psi_n\theta^2\|$ and $C_0/Q_{n+1}$, where $\theta^3=-\theta+1$}
\label{tab:LWseqex0.-1.1}
\end{table}

%% file: disEx1.1.1.tex

\subsection{Example 7: $\alpha^3=\alpha^2+\alpha+1$}
Consider the pair $(\alpha,\alpha^{-1})$, where $\alpha\approx 1.83929$ is the real root of $f(x)=x^3-x^2-x-1$. (The discriminant of $f$ is $-44$.) Note that $\alpha(\alpha^2-\alpha-1)=1$, or $\alpha^{-1}=\alpha^2-\alpha-1$. So we have $A=1$, $B=-1$, $r_0=s=1$, $r_1=r_2=-1$.

\begin{itemize}
  \item[]
\begin{itemize}
\item[\bf Step 1. ] We transform $f$ to
\begin{equation*}
g(x)
=
27 f\left(\frac{x+1}{3}\right)
=
x^3-12x-38,
\end{equation*}
which has $\theta:=3\alpha-1\approx 4.51786$ as its real root. Using PARI/GP, we find that $\O_K\subset \frac{1}{9}\Z[\theta]$ (so we put $d=9$).
\item[\bf Step 2. ] Note that $\alpha=\frac{\theta+1}{3}\approx 1.83929$ is a unit. Since $\alpha>1$, we take $\lambda=\alpha$.
\item[\bf Step 3. ] Define the sequences $\{a_n\}$, $\{b_n\}$, $\{c_n\}$ as before and define $\{m_n\}$ by
 $m_n=9A^2sd c_n=81 c_n$. (We plot the points $\{{m_n^{1/2}}\langle m_n\alpha\rangle,{m_n^{1/2}}\langle m_n\alpha^{-1}\rangle\}$ for $n=1,\dots,200$ in Figure~\ref{fig:LWseqex1.1.1}.)
\item[\bf Step 4. ] Put
\begin{equation*}
\phi
=
\arctan\left(
\frac{\sqrt{3\theta^2-48}}
{2-\theta}
\right)
 \approx -0.965359,
\end{equation*}
The first few convergents of $\phi/\pi \approx -0.307283$ are:
\begin{equation*}
-\frac{1}{3}, -\frac{3}{10}, -\frac{4}{13}, -\frac{55}{179}, -\frac{59}{192}, -\frac{173}{563}, -\frac{578}{1881}, -\frac{13467}{43826}, -\frac{162182}{527793}, -\frac{175649}{571619}
\end{equation*}
\item[\bf Step 5. ] We have

\begin{tabular}[t]{rcl rcl rcl}
$C_1$  &  $\approx$  &  $1.75637$
\qquad\qquad
& $C_2$  &  $\approx$  &  $0.427555$
\qquad\qquad
& $\widetilde{M_\theta}$  &  $\approx$  &  $10.1378$
\\
$\widetilde{M_\alpha}$  &  $\approx$  &  $91.2398$
& $\widetilde{M_{\beta_1}}$  &  $\approx$  &  $182.480$
& $\widetilde{M_{\beta_2}}$  &  $\approx$  &  $44.8628$
\end{tabular}

and
\begin{align*}
\widetilde{N}
\approx
\max\{ 1.356, 9.896, 31.61, 1.273, 121.7, 162.2, 243.3, 243.3, 121.7 \}
=
243.3
\end{align*}
\item[\bf Step 6. ] We have
\begin{align*}
C_0  & \approx 8186.54\\
n_0   & \approx 18.0326
\end{align*}
and we define $\{\psi_n\}$ by $\psi_n=m_{Q_n}$. The first several terms are:
\begin{gather*}
m_{3},\  m_{10},\  m_{13},\  m_{179},\  m_{192},\  m_{563},\  m_{1881},\  m_{43826},\  m_{527793},\  m_{571619},\  m_{1671031}
\end{gather*}
\item[\bf Result. ] For $Q_n\geq 19>n_0$,
\begin{equation*}
\psi_n\|\psi_n\alpha\|\|\psi_n\alpha^{-1}\|
<
\frac{C_0}{Q_{n+1}}
\approx
\frac{8186.54}{Q_{n+1}}.
\end{equation*}
In Table~\ref{tab:LWseqex1.1.1}
we compute the bound $\frac{C_0}{Q_{n+1}}$ and compare it to the actual values of $\psi_n\|\psi_n\alpha\|\|\psi_n\alpha^{-1}\|$  for the first few $n$.
\end{itemize}
\end{itemize}

\begin{figure}[htb]
    \centering
        \includegraphics[width=\figwidth\textwidth
        ]
        {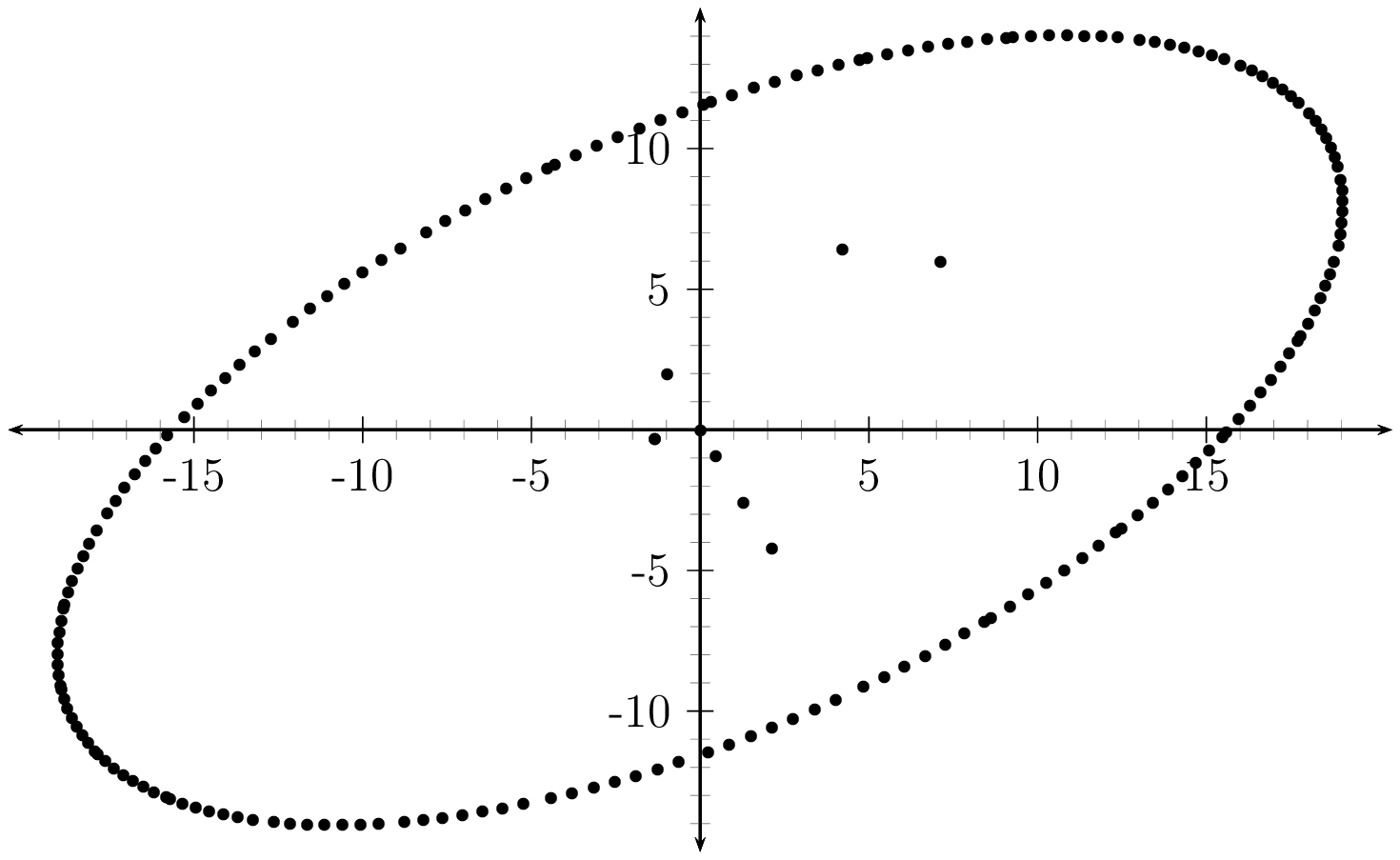}
    \caption{
    $(\sqrt{m_n}\langle m_n\alpha\rangle,\sqrt{m_n}\langle m_n\alpha^{-1}\rangle)$ for
    $\alpha^3=\alpha^2+\alpha+1$, $n\leq 200$.
    }
    \label{fig:LWseqex1.1.1}
\end{figure}

\begin{table}
\centering
\begin{tabular}{p{1.9 cm} p{5.1 cm} p{3.3 cm} p{2.8 cm}}\toprule
$Q_n$ & $\psi_n$ $(=m_{Q_n})$ & $\psi_n\|\psi_n\alpha\|\|\psi_n\alpha^{-1}\|$ & $C_0/Q_{n+1}$\\
\midrule
3  &  9  &  $0.429094$  &  $818.654$\\
10  &  729  &  $40.7299$  &  $629.734$\\
13  &  4536  &  $3.71296$  &  $45.7349$\\
179  &  $387 \dots 121$     (48 digits)  &  $2.54725$  &  $42.6382$\\
192  &  $106 \dots 568$     (52 digits)  &  $1.10673$  &  $14.5409$\\
563  &  $163 \dots 601$     (150 digits)  &  $0.361992$  &  $4.35223$\\
1881  &  $105 \dots 308$     (499 digits)  &  $0.0157040$  &  $0.186796$\\
43826  &  $553 \dots 553$     (11599 digits)  &  $0.00123216$  &  $0.0155109$\\
527793  &  $218 \dots 856$     (139681 digits)  &  $0.00091720$  &  $0.0143217$\\
571619  &  $735 \dots 889$     (151279 digits)  &  $0.000314968$  &  $0.00489910$\\
1671031  &  $436 \dots 357$     (442238 digits)  &  $0.000287266$  &  $0.00365039$\\
2242650  &  $195 \dots 217$     (593517 digits)  &  $0.0000277022$  &  $0.000339725$\\
24097531  &  $242 \dots 645$     (6377399 digits)  &  $0.0000102440$  &  $0.000162310$\\
50437712  &  $424 \dots 464$     (13348313 digits)  &  $7.2142*10^{-6}$  &  $0.000109834$\\
74535243  &  $625 \dots 773$     (19725711 digits)  &  $3.02985*10^{-6}$  &  $0.0000410336$\\
199508198  &  $613 \dots 799$     (52799734 digits)  &  $1.15446*10^{-6}$  &  $0.0000172875$\\
\bottomrule
\end{tabular}
\caption{$\psi_n\|\psi_n\alpha\|\|\psi_n\alpha^{-1}\|$ and $C_0/Q_{n+1}$, where $\alpha^3=-\alpha^2-\alpha-1$}
\label{tab:LWseqex1.1.1}
\end{table}

%% file: disEx0.-2.1.tex

\subsection{Example 8: $\theta^3=-2\cdot\theta+1$}

Let $\theta$ be the real root of $f(x)=x^3+2x-1$, and consider the pair $(\theta,\theta^2)$. (The discriminant of $f$ is $-59$.) We have $A=1$, $B=0$, $r_0=r_1=0$, $r_2=s=1$.

\begin{itemize}
  \item[]
\begin{itemize}
\item[\bf Step 1. ] We already have $f$ in the form $x^3-px-q$, with $p=-2$ and $q=1$. Using PARI/GP, we find that $\O_K=\Z[\theta]$, so we put $d=1$.
\item[\bf Step 2. ] Since $\theta\approx 0.4534$ is a unit and $0<\theta<1$, we put
\begin{equation*}
\lambda=\theta^{-1}=\theta^2+2\approx 2.20557.
\end{equation*}
\item[\bf Step 3. ] Define the sequences $\{a_n\}$, $\{b_n\}$, $\{c_n\}$ as before and (since $B=0$ and $d=s=1$) define $\{m_n\}$ by $m_n=c_n$. (We plot the points $\{{m_n^{1/2}}\langle m_n\theta\rangle,{m_n^{1/2}}\langle
 m_n\theta^2\rangle\}$ for $n=1,\dots,200$ in Figure~\ref{fig:LWseqex0.-2.1}.)
\item[\bf Step 4. ] Put
\begin{equation*}
\phi
=
\arctan\left(
\frac{\sqrt{3\theta^2+8}}
{\theta}
\right)
 \approx 1.41755,
\end{equation*}
The first few convergents of $\phi/\pi \approx 0.45122$ are:
\begin{equation*}
\frac{1}{2}, \frac{4}{9}, \frac{5}{11}, \frac{9}{20}, \frac{14}{31}, \frac{37}{82}, \frac{8265}{18317}, \frac{16567}{36716}, \frac{190502}{422193}, \frac{20971787}{46477946}, \frac{21162289}{46900139}, \frac{2243012132}{4970992541}
\end{equation*}
\item[\bf Step 5. ] We have

\begin{tabular}[t]{rcl rcl rcl}
$C_1$  &  $\approx$  &  $1.41421$
\qquad\qquad
& $C_2$  &  $\approx$  &  $0.618191$
\qquad\qquad
& $\widetilde{M_\theta}$  &  $\approx$  &  $4.09039$
\\
$\widetilde{M_\alpha}$  &  $\approx$  &  $4.09039$
& $\widetilde{M_{\beta_1}}$  &  $\approx$  &  $4.09039$
& $\widetilde{M_{\beta_2}}$  &  $\approx$  &  $0.992414$
\end{tabular}

and
\begin{align*}
\widetilde{N}
\approx
\max\{ 1.48512, 8.822, 2.828, 1.440, 16.36, 0, 32.72, 16.36 \}
=
32.72
\end{align*}
\item[\bf Step 6. ] We have
\begin{align*}
C_0  & \approx 4.05936\\
n_0   & \approx 8.81958
\end{align*}
and we define $\{\psi_n\}$ by $\psi_n=m_{Q_n}$. The first several terms are:
\begin{gather*}
m_{2},\  m_{9},\  m_{11},\  m_{20},\  m_{31},\  m_{82},\  m_{18317},\  m_{36716},\  m_{422193},\  m_{46477946},\  m_{46900139}
\end{gather*}
\item[\bf Result. ] For $Q_n\geq 9>n_0$,
\begin{equation*}
\psi_n\|\psi_n\theta\|\|\psi_n\theta^2\|
<
\frac{C_0}{Q_{n+1}}
\approx
\frac{4.05936}{Q_{n+1}}.
\end{equation*}
In Table~\ref{tab:LWseqex0.-2.1}
we compute the bound $\frac{C_0}{Q_{n+1}}$ and compare it to the actual values of
$\psi_n\|\psi_n\theta\|\|\psi_n\theta^2\|$  for the first few $n$.
\end{itemize}
\end{itemize}

\begin{figure}[htb]
    \centering
        \includegraphics[width=\figwidth\textwidth
        ]
        {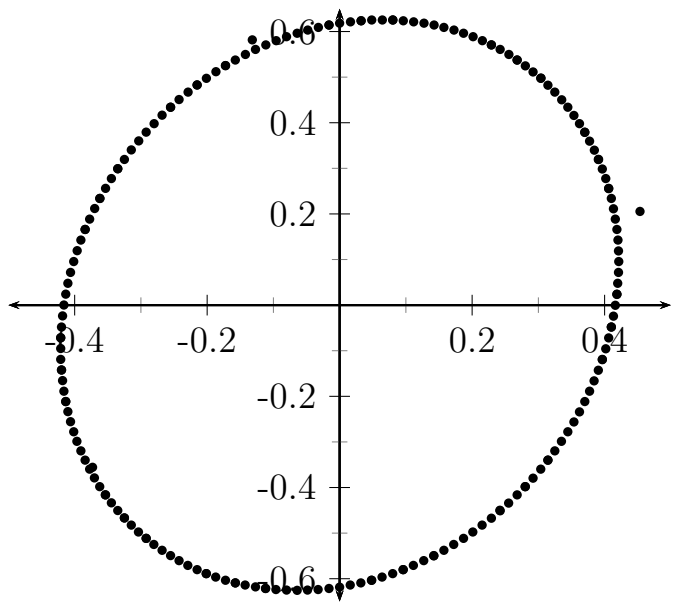}
    \caption{
    $(\sqrt{m_n}\langle m_n\theta\rangle,\sqrt{m_n}\langle m_n\theta^2\rangle)$ for
    $\theta^3=-2 \theta+1$, $n\leq 200$.
    }
    \label{fig:LWseqex0.-2.1}
\end{figure}

\begin{table}
\centering
\begin{tabular}{p{1.9 cm} p{5.1 cm} p{3.3 cm} p{2.8 cm}}\toprule
$Q_n$ & $\psi_n$ $(=m_{Q_n})$ & $\psi_n\|\psi_n\alpha\|\|\psi_n\alpha^{-1}\|$ & $C_0/Q_{n+1}$\\
\midrule
2  &  2  &  $0.076640$  &  $0.451040$\\
9  &  472  &  $0.050131$  &  $0.369032$\\
11  &  2296  &  $0.0291294$  &  $0.202968$\\
20  &  2835694  &  $0.0201199$  &  $0.130947$\\
31  &  17036776865  &  $0.0098903$  &  $0.0495043$\\
82  &  $563 \dots 769$     (28 digits)  &  $0.0000445638$  &  $0.000221617$\\
18317  &  $710 \dots 188$     (6292 digits)  &  $0.0000213130$  &  $0.000110561$\\
36716  &  $194 \dots 174$     (12613 digits)  &  $1.93599*10^{-6}$  &  $9.61493*10^{-6}$\\
422193  &  $427 \dots 544$     (145032 digits)  &  $1.74428*10^{-8}$  &  $8.73394*10^{-8}$\\
46477946  &  $495 \dots 762$     (15966138 digits)  &  $1.72784*10^{-8}$  &  $8.65532*10^{-8}$\\
46900139  &  $554 \dots 712$     (16111170 digits)  &  $1.64421*10^{-10}$  &  $8.16609*10^{-10}$\\
\bottomrule
\end{tabular}
\caption{$\psi_n\|\psi_n\theta\|\|\psi_n\theta^2\|$ and $C_0/Q_{n+1}$, where $\theta^3=-2\theta+1$}
\label{tab:LWseqex0.-2.1}
\end{table}

%% file: disEx0.2.2.tex

\subsection{Example 9: $\theta^3=2\cdot\theta+2$}

Let $\theta$ be the real root of $f(x)=x^3-2x-2$, and consider the pair $(\theta,\theta^2)$. (The discriminant of $f$ is $-76$.) We have $A=1$, $B=0$, $r_0=r_1=0$, $r_2=s=1$.

\begin{itemize}
  \item[]
\begin{itemize}
\item[\bf Step 1. ] We already have $f$ in the form $x^3-px-q$, with $p=2$ and $q=2$. Using PARI/GP, we find that $\O_K=\Z[\theta]$, so we put $d=1$.
\item[\bf Step 2. ] Using PARI/GP, we find that a fundamental unit is
\begin{equation*}
\ep_0=\theta+1\approx 2.76929.
\end{equation*}
Since $\ep_0>1$, we put $\lambda=\ep_0$.

\item[\bf Step 3. ] Define the sequences $\{a_n\}$, $\{b_n\}$, $\{c_n\}$ as before and define $\{m_n\}$ by $m_n=c_n$. (We plot the points $\{{m_n^{1/2}}\langle m_n\theta\rangle,{m_n^{1/2}}\langle
 m_n\theta^2\rangle\}$ for $n=1,\dots,200$ in Figure~\ref{fig:LWseqex0.2.2}.)
\item[\bf Step 4. ] Put
\begin{equation*}
\phi
=
\arctan\left(
\frac{\sqrt{3\theta^2-8}}
{2-\theta}
\right)
 \approx 1.37763,
\end{equation*}
The first few convergents of $\phi/\pi \approx 0.438515$ are:
\begin{equation*}
\frac{1}{2}, \frac{3}{7}, \frac{4}{9}, \frac{7}{16}, \frac{25}{57}, \frac{82}{187}, \frac{189}{431}, \frac{3484}{7945}, \frac{3673}{8376}, \frac{25522}{58201}, \frac{105761}{241180}, \frac{131283}{299381}, \frac{893459}{2037466}
\end{equation*}
\item[\bf Step 5. ] We have

\begin{tabular}[t]{rcl rcl rcl}
$C_1$  &  $\approx$  &  $2.12140$
\qquad\qquad
& $C_2$  &  $\approx$  &  $0.367826$
\qquad\qquad
& $\widetilde{M_\theta}$  &  $\approx$  &  $1.17046$
\\
$\widetilde{M_\alpha}$  &  $\approx$  &  $1.17046$
& $\widetilde{M_{\beta_1}}$  &  $\approx$  &  $1.17046$
& $\widetilde{M_{\beta_2}}$  &  $\approx$  &  $1.46957$
\end{tabular}

and
\begin{align*}
\widetilde{N}
\approx
\max\{ 1.66412, 3.391, 4.243, 1.211, 4.682, 0, 9.364, 4.682 \}
=
9.364
\end{align*}
\item[\bf Step 6. ] We have
\begin{align*}
C_0  & \approx 1.72008\\
n_0   & \approx 4.39202
\end{align*}
and we define $\{\psi_n\}$ by $\psi_n=m_{Q_n}$. The first several terms are:
\begin{gather*}
m_{2},\  m_{7},\  m_{9},\  m_{16},\  m_{57},\  m_{187},\  m_{431},\  m_{7945},\  m_{8376},\  m_{58201},\  m_{241180},\  m_{299381}
\end{gather*}
\item[\bf Result. ] For $Q_n\geq 5>n_0$,
\begin{equation*}
\psi_n\|\psi_n\theta\|\|\psi_n\theta^2\|
<
\frac{C_0}{Q_{n+1}}
\approx
\frac{1.72008}{Q_{n+1}}.
\end{equation*}
In Table~\ref{tab:LWseqex0.2.2}
we compute the bound $\frac{C_0}{Q_{n+1}}$ and compare it to the actual values of
$\psi_n\|\psi_n\theta\|\|\psi_n\theta^2\|$  for the first few $n$.
\end{itemize}
\end{itemize}

\begin{figure}[htb]
    \centering
        \includegraphics[width=\figwidth\textwidth
        ]
        {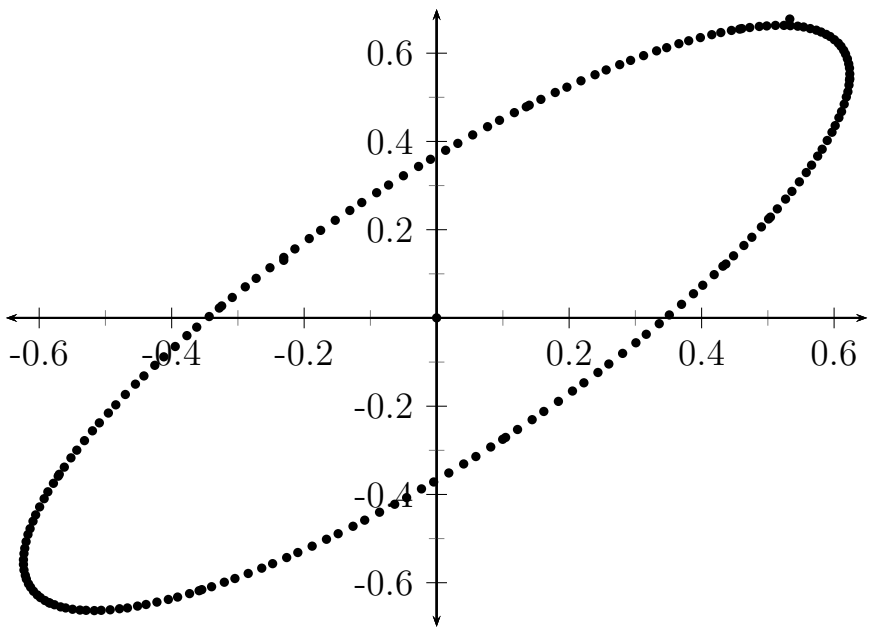}
    \caption{
    $(\sqrt{m_n}\langle m_n\theta\rangle,\sqrt{m_n}\langle m_n\theta^2\rangle)$ for
    $\theta^3=2 \theta+2$, $n\leq 200$.
    }
    \label{fig:LWseqex0.2.2}
\end{figure}

\begin{table}[h!tb]
\centering
\begin{tabular}{p{1.9 cm} p{5.3 cm} p{3.3 cm} p{2.8 cm}}\toprule
$Q_n$ & $\psi_n$ $(=m_{Q_n})$ & $\psi_n\|\psi_n\theta\|\|\psi_n\theta^2\|$ & $C_0/Q_{n+1}$\\
\midrule
2  &  1  &  $0.0300832$  &  $0.245725$\\
7  &  169  &  $0.064779$  &  $0.191120$\\
9  &  1296  &  $0.0281634$  &  $0.107505$\\
16  &  1618776  &  $0.0125768$  &  $0.0301768$\\
57  &  2219769241218582281661888  &  $0.00328181$  &  $0.00919827$\\
187  &  $714 \dots 835$     (82 digits)  &  $0.00165204$  &  $0.00399089$\\
431  &  $619 \dots 593$     (190 digits)  &  $0.000086617$  &  $0.000216498$\\
7945  &  $564 \dots 776$     (3514 digits)  &  $0.000074675$  &  $0.000205358$\\
8376  &  $258 \dots 036$     (3705 digits)  &  $0.0000120265$  &  $0.0000295541$\\
58201  &  $274 \dots 936$     (25746 digits)  &  $2.47362*10^{-6}$  &  $7.13192*10^{-6}$\\
241180  &  $436 \dots 500$     (106690 digits)  &  $2.13313*10^{-6}$  &  $5.74544*10^{-6}$\\
299381  &  $885 \dots 886$     (132436 digits)  &  $3.40429*10^{-7}$  &  $8.44223*10^{-7}$\\
2037466  &  $342 \dots 477$     (901311 digits)  &  $9.0588*10^{-8}$  &  $2.68268*10^{-7}$\\
6411779  &  $143 \dots 251$     (2836371 digits)  &  $6.8663*10^{-8}$  &  $2.03577*10^{-7}$\\
8449245  &  $362 \dots 526$     (3737682 digits)  &  $2.19251*10^{-8}$  &  $5.41594*10^{-8}$\\
31759514  &  $274 \dots 581$     (14049418 digits)  &  $2.88800*10^{-9}$  &  $7.45377*10^{-9}$\\
\bottomrule
\end{tabular}
\caption{$\psi_n\|\psi_n\theta\|\|\psi_n\theta^2\|$ and $C_0/Q_{n+1}$, where $\theta^3=2\theta+2$}
\label{tab:LWseqex0.2.2}
\end{table}

%% file: disEx1.-1.2.tex

\subsection{Example 10: $\alpha^3=\alpha^2-\alpha+2$}

Let $\alpha\approx 1.35321$ be the real root of $f(x)=x^3-x^2+x-2$, and consider the pair $(\alpha,\alpha^2)$. (The discriminant of $f$ is $-83$.) We have $A=1$, $B=-1$, $r_0=r_1=0$, $r_2=s=1$.

\begin{itemize}
  \item[]
\begin{itemize}
\item[\bf Step 1. ] We transform $f$ to
\begin{equation*}
g(x)
=
27 f\left(\frac{x+1}{3}\right)
=
x^3+6x-47,
\end{equation*}
which has $\theta:=3\alpha-1\approx 3.05963$ as its real root. Using PARI/GP, we find that $\O_K\subset \frac{1}{9}\Z[\theta]$ (so we put $d=9$).
\item[\bf Step 2. ] Using PARI/GP, we find that a fundamental unit is
\begin{equation*}
\ep_0=-\frac{2}{3}+\frac{1}{3}\theta\approx 0.35321.
\end{equation*}
Since $0<\ep_0<1$, we put
\begin{equation*}
\lambda
=
\ep_0^{-1}
=
\frac{10}{9}+\frac{2}{9}\theta+ \frac{1}{9}\theta^2
\approx
2.83118
\end{equation*}

\item[\bf Step 3. ] Define the sequences $\{a_n\}$, $\{b_n\}$, $\{c_n\}$ as before and define $\{m_n\}$ by
 $m_n=9A^2sd c_n=81 c_n$. (We plot the points $\{{m_n^{1/2}}\langle m_n\alpha\rangle,{m_n^{1/2}}\langle m_n\alpha^2\rangle\}$ for $n\leq 200$ in Figure~\ref{fig:LWseqex1.-1.2}.)
\item[\bf Step 4. ] Put
\begin{equation*}
\phi
=
\arctan\left(
\frac{\sqrt{3\theta^2+24}(2-\theta)}
{8-2\theta-\theta^2}
\right)
 \approx 0.796415,
\end{equation*}
The first few convergents of $\phi/\pi \approx 0.253507$ are:
\begin{equation*}
\frac{1}{3}, \frac{1}{4}, \frac{18}{71}, \frac{235}{927}, \frac{253}{998}, \frac{2259}{8911}, \frac{4771}{18820}, \frac{7030}{27731}, \frac{25861}{102013}, \frac{84613}{333770}, \frac{110474}{435783}, \frac{195087}{769553}
\end{equation*}
\item[\bf Step 5. ] We have

\begin{tabular}[t]{rcl rcl rcl}
$C_1$  &  $\approx$  &  $1.41421$
\qquad\qquad
& $C_2$  &  $\approx$  &  $0.513861$
\qquad\qquad
& $\widetilde{M_\theta}$  &  $\approx$  &  $9.81058$
\\
$\widetilde{M_\alpha}$  &  $\approx$  &  $88.2952$
& $\widetilde{M_{\beta_1}}$  &  $\approx$  &  $88.2952$
& $\widetilde{M_{\beta_2}}$  &  $\approx$  &  $27.1782$
\end{tabular}

and
\begin{align*}
\widetilde{N}
\approx
\max\{ 1.68261, 11.77, 25.46, 1.354, 117.7, 157.0, 0, 235.5, 117.7 \}
=
235.5
\end{align*}
\item[\bf Step 6. ] We have
\begin{align*}
C_0  & \approx 2399.70\\
n_0   & \approx 10.4959
\end{align*}
and we define $\{\psi_n\}$ by $\psi_n=m_{Q_n}$. The first several terms are:
\begin{gather*}
m_{3},\  m_{4},\  m_{71},\  m_{927},\  m_{998},\  m_{8911},\  m_{18820},\  m_{27731},\  m_{102013},\  m_{333770},\  m_{435783}
\end{gather*}
\item[\bf Result. ] For $Q_n\geq 11>n_0$,
\begin{equation*}
\psi_n\|\psi_n\alpha\|\|\psi_n\alpha^2\|
<
\frac{C_0}{Q_{n+1}}
\approx
\frac{2399.70}{Q_{n+1}}.
\end{equation*}
In Table~\ref{tab:LWseqex1.-1.2}
we compute the bound $\frac{C_0}{Q_{n+1}}$ and compare it to the actual values of
$\psi_n\|\psi_n\alpha\|\|\psi_n\alpha^2\|$  for the first few $n$.
\end{itemize}
\end{itemize}

\begin{figure}[h!tb]
    \centering
        \includegraphics[width=\figwidth\textwidth
        ]
        {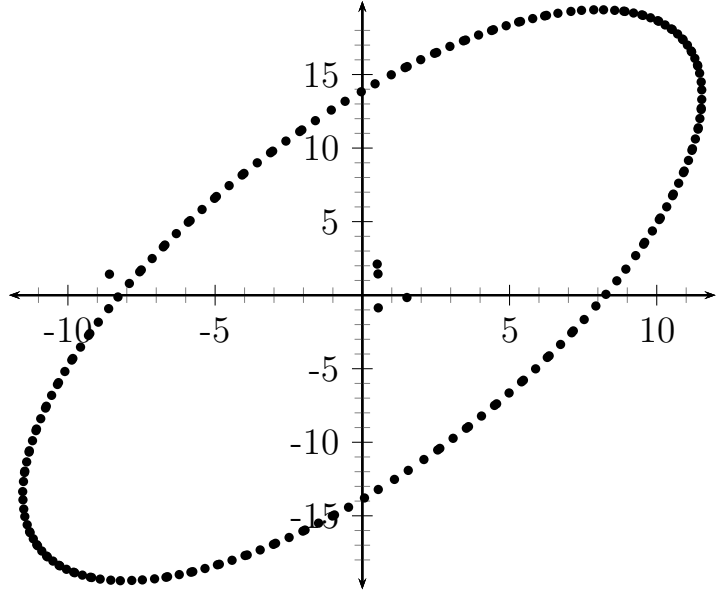}
    \caption{
    $(\sqrt{m_n}\langle m_n\alpha\rangle,\sqrt{m_n}\langle m_n\alpha^2\rangle)$ for
    $\alpha^3=\alpha^2-\alpha+2$, $n\leq 200$.
    }
    \label{fig:LWseqex1.-1.2}
\end{figure}

\begin{table}[h!tb]
\centering
\begin{tabular}{p{1.9 cm} p{5.3 cm} p{3.3 cm} p{2.8 cm}}\toprule
$Q_n$ & $\psi_n$ $(=m_{Q_n})$ & $\psi_n\|\psi_n\alpha\|\|\psi_n\alpha^2\|$ & $C_0/Q_{n+1}$\\
\midrule
3  &  54  &  $0.461096$  &  $599.925$\\
4  &  153  &  $1.07036$  &  $33.7986$\\
71  &  $292 \dots 560$     (33 digits)  &  $0.50606$  &  $2.58867$\\
927  &  $223 \dots 190$     (420 digits)  &  $0.454662$  &  $2.40451$\\
998  &  $274 \dots 634$     (452 digits)  &  $0.054208$  &  $0.269296$\\
8911  &  $715 \dots 385$     (4028 digits)  &  $0.0196042$  &  $0.127508$\\
18820  &  $248 \dots 673$     (8507 digits)  &  $0.0150207$  &  $0.0865350$\\
27731  &  $748 \dots 424$     (12534 digits)  &  $0.00457984$  &  $0.0235235$\\
102013  &  $777 \dots 457$     (46107 digits)  &  $0.00128298$  &  $0.00718968$\\
333770  &  $262 \dots 058$     (150854 digits)  &  $0.00073076$  &  $0.00550664$\\
435783  &  $859 \dots 070$     (196960 digits)  &  $0.00055223$  &  $0.00311830$\\
769553  &  $949 \dots 944$     (347813 digits)  &  $0.000178517$  &  $0.000874386$\\
2744442  &  $549 \dots 968$     (1240398 digits)  &  $0.0000166865$  &  $0.0000850536$\\
28213973  &  $411 \dots 832$     (12751787 digits)  &  $0.0000116516$  &  $0.0000775137$\\
30958415  &  $950 \dots 680$     (13992184 digits)  &  $5.0349*10^{-6}$  &  $0.0000266246$\\
\bottomrule
\end{tabular}
\caption{$\psi_n\|\psi_n\alpha\|\|\psi_n\alpha^2\|$ and $C_0/Q_{n+1}$, where $\alpha^3=\alpha^2-\alpha+2$}
\label{tab:LWseqex1.-1.2}
\end{table}

%% file: disEx0.1.2.tex

\subsection{Example 11: $\theta^3=\theta+2$}

Let $\theta$ be the real root of $f(x)=x^3-x-2$, and consider the pair $(\theta,\theta^2)$. (The discriminant of $f$ is $-104$.) We have $A=1$, $B=0$, $r_0=r_1=0$, $r_2=s=1$.

\begin{itemize}
  \item[]
\begin{itemize}
\item[\bf Step 1. ] We already have $f$ in the form $x^3-px-q$, with $p=1$ and $q=2$. Using PARI/GP, we find that $\O_K=\Z[\theta]$, so we put $d=1$.
\item[\bf Step 2. ] Using PARI/GP, we find that a fundamental unit is
\begin{equation*}
\ep_0=1+\theta+\theta^2\approx 4.83598.
\end{equation*}
Since $\ep_0>1$, we put $\lambda=\ep_0$.
\item[\bf Step 3. ] Define the sequences $\{a_n\}$, $\{b_n\}$, $\{c_n\}$ as before and define $\{m_n\}$ by $m_n=c_n$. (We plot the points $\{{m_n^{1/2}}\langle m_n\theta\rangle,{m_n^{1/2}}\langle
 m_n\theta^2\rangle\}$ for $n\leq 200$ in Figure~\ref{fig:LWseqex0.1.2}.)
\item[\bf Step 4. ] Put
\begin{equation*}
\phi
=
\arctan\left(
\frac{\sqrt{3\theta^2-4}(1-\theta)}
{4-\theta-\theta^2}
\right)
 \approx -1.38945,
\end{equation*}
The first few convergents of $\phi/\pi \approx -0.442276$ are:
\begin{equation*}
-\frac{1}{2}, -\frac{3}{7}, -\frac{4}{9}, -\frac{19}{43}, -\frac{23}{52}, -\frac{272}{615}, -\frac{11175}{25267}, -\frac{11447}{25882}, -\frac{22622}{51149}, -\frac{34069}{77031}, -\frac{90760}{205211}
\end{equation*}
\item[\bf Step 5. ] We have

\begin{tabular}[t]{rcl rcl rcl}
$C_1$  &  $\approx$  &  $1.41421$
\qquad\qquad
& $C_2$  &  $\approx$  &  $0.410174$
\qquad\qquad
& $\widetilde{M_\theta}$  &  $\approx$  &  $0.870111$
\\
$\widetilde{M_\alpha}$  &  $\approx$  &  $0.870111$
& $\widetilde{M_{\beta_1}}$  &  $\approx$  &  $0.870111$
& $\widetilde{M_{\beta_2}}$  &  $\approx$  &  $1.12656$
\end{tabular}

and
\begin{align*}
\widetilde{N}
\approx
\max\{ 2.19909, 2.629, 2.828, 1.256, 3.480, 0, 6.961, 3.480 \}
=
6.961
\end{align*}
\item[\bf Step 6. ] We have
\begin{align*}
C_0  & \approx 0.980236\\
n_0   & \approx 2.46219
\end{align*}
and we define $\{\psi_n\}$ by $\psi_n=m_{Q_n}$. The first several terms are:
\begin{gather*}
m_{2},\  m_{7},\  m_{9},\  m_{43},\  m_{52},\  m_{615},\  m_{25267},\  m_{25882},\  m_{51149},\  m_{77031},\  m_{205211},\  m_{282242}
\end{gather*}
\item[\bf Result. ] For $Q_n\geq 3>n_0$,
\begin{equation*}
\psi_n\|\psi_n\theta\|\|\psi_n\theta^2\|
<
\frac{C_0}{Q_{n+1}}
\approx
\frac{0.980236}{Q_{n+1}}.
\end{equation*}
In Table~\ref{tab:LWseqex0.1.2}
we compute the bound $\frac{C_0}{Q_{n+1}}$ and compare it to the actual values of
$\psi_n\|\psi_n\theta\|\|\psi_n\theta^2\|$  for the first few $n$.
\end{itemize}
\end{itemize}

\begin{figure}[h!tb]
    \centering
        \includegraphics[width=\figwidth\textwidth
        ]
        {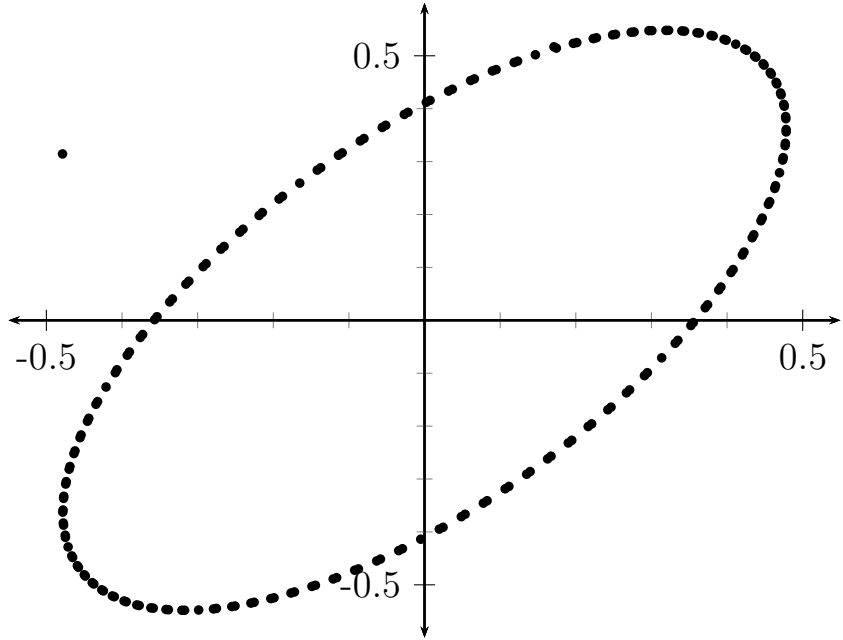}
    \caption{
    $(\sqrt{m_n}\langle m_n\theta\rangle,\sqrt{m_n}\langle m_n\theta^2\rangle)$ for
    $\theta^3=\theta+2$, $n\leq 200$.
    }
    \label{fig:LWseqex0.1.2}
\end{figure}

\begin{table}[h!tb]
\centering
\begin{tabular}{p{1.9 cm} p{5.3 cm} p{3.3 cm} p{2.8 cm}}\toprule
$Q_n$ & $\psi_n$ $(=m_{Q_n})$ & $\psi_n\|\psi_n\theta\|\|\psi_n\theta^2\|$ & $C_0/Q_{n+1}$\\
\midrule
2  &  4  &  $0.088387$  &  $0.140034$\\
7  &  10407  &  $0.0402688$  &  $0.108915$\\
9  &  243385  &  $0.0126439$  &  $0.0227962$\\
43  &  $455 \dots 491$     (29 digits)  &  $0.0104498$  &  $0.0188507$\\
52  &  $659 \dots 364$     (35 digits)  &  $0.00100429$  &  $0.00159388$\\
615  &  $152 \dots 879$     (421 digits)  &  $0.0000240102$  &  $0.0000387951$\\
25267  &  $122 \dots 779$     (17295 digits)  &  $0.0000152616$  &  $0.0000378733$\\
25882  &  $110 \dots 180$     (17716 digits)  &  $8.7519*10^{-6}$  &  $0.0000191643$\\
51149  &  $802 \dots 749$     (35010 digits)  &  $6.5084*10^{-6}$  &  $0.0000127252$\\
77031  &  $528 \dots 079$     (52726 digits)  &  $2.24402*10^{-6}$  &  $4.77672*10^{-6}$\\
205211  &  $791 \dots 131$     (140463 digits)  &  $2.02017*10^{-6}$  &  $3.47303*10^{-6}$\\
282242  &  $248 \dots 196$     (193190 digits)  &  $2.23897*10^{-7}$  &  $3.57048*10^{-7}$\\
2745389  &  $266 \dots 125$     (1879175 digits)  &  $5.0764*10^{-9}$  &  $8.09582*10^{-9}$\\
\bottomrule
\end{tabular}
\caption{$\psi_n\|\psi_n\theta\|\|\psi_n\theta^2\|$ and $C_0/Q_{n+1}$, where $\theta^3=\theta+2$}
\label{tab:LWseqex0.1.2}
\end{table}

%% file: disappendix.tex
We construct our sequences $\{c_n\}$ by taking powers of a unit $\lambda$ in $\Q[x]/(x^3-px-q)$, and taking their $\theta^2$-coordinates in the power basis $1,\theta,\theta^2$ (where $\theta$ is the real root of $x^3-px-q$). Two ways of computing the $c_n$'s are by working with polynomials modulo $x^3-px-q$ (for which we used PARI/GP~2.7.2) and by working with matrices (for which we used Mathematica~8.0). We list some of the specific commands and algorithms we used.

\section{Using Polynomials Modulo $x^3-px-q$ in PARI/GP}

Let $f=x^3-px-q$ (for example, \verb|f=x^3-2| or \verb|f=x^3-147*x-740|), and initialize the number field: \verb|k=bnfinit(f);|

\subsection{Finding $\lambda$}

\begin{enumerate}
\item Find a fundamental unit $\ep_0$: \verb|ep=k.fu[1]|
\item Find the real value of $\ep_0$ (get value of $\theta$ and plug it into polynomial for $\ep_0$): \\ \verb|t=polroots(f)[1]; subst(ep.pol,x,t)|
\item Depending on the value of $\ep_0$, let $\lambda$ be one of $\pm \ep_0^{\pm 1}$. For example: \verb|lam=ep^-1| or \verb|lam=-ep|
\end{enumerate}

\subsection{Computing $c_n$}
Assuming we've found a unit $\lambda>1$, we can compute $\lambda^n$ directly (as \verb|lam^n|), and $c_n$ by
\begin{equation*}
\verb|polcoeff((lam^n).pol,2)|
\end{equation*}
We can define a function
\begin{equation*}
\verb|cn(l,n)=polcoeff((l^n).pol,2)|
\end{equation*}
and then compute $c_n$ as \verb|cn(lam,n)|. If we also wanted to be able to find the $a_n$'s or $b_n$'s, we could define a function
\begin{equation*}
\verb|lcoeff(l,n,j)=polcoeff((l^n).pol,j)|
\end{equation*}
where \verb|lcoeff(lam,n,j)| finds the $\theta^j$ coordinate of $\lambda^n$ in the basis $1,\theta,\theta^2$. So \begin{equation*}
\verb|lcoeff(lam,n,0)|
\end{equation*}
would be $a_n$ and
\begin{equation*}
\verb|lcoeff(lam,n,1)|
\end{equation*}
would be $b_n$.\\

\subsection{Finding $m_n$ and Calculating $m_n\|m_n\alpha\|\|m_n\beta\|$}
To find $k_n$, we find an integral basis with \verb|k.zk| and let $d$ be the largest denominator in this integral basis. Then $k_n=d c_n$, and $m_n$ is some integer multiple of $k_n$.\\

To evaluate $n\|n\alpha\|\|n\beta\|$, we can first define $\lb x\rb$ by
\begin{equation*}
\verb|mod1(x)=x-round(x)|
\end{equation*}
and then define
\begin{equation*}
\verb|LWprod(n,a,b)=n*abs(mod1(n*a)*mod1(n*b))|
\end{equation*}
For example, we could find $k_n\|k_n\theta\|\|k_n\theta^2\|$ as
\begin{equation*}
\verb|LWprod(d*cn(lam,n),t,t^2)|
\end{equation*}

\subsection{Increasing Precision and Memory}
To be able to compute this for large $n$, we may need to increase the precision and available memory. To increase the precision to $N$ digits, we use
\begin{equation*}
\verb|default(realprecision,N)|
\end{equation*}
If we increase the precision, we need to re-calculate the values of $\theta$ or $\alpha$ or $\beta$. This may require increasing the available memory with \verb|allocatemem()| several times. For instance, calculating that
\begin{equation*}
c_n\|c_n\sqrt[3]{2}\|\|c_n\sqrt[3]{4}\|\approx 4.05112*10^{-7}
\end{equation*}
for $n=1059767$ requires around 1~million digits of precision (the default is 28 significant digits), and finding $\sqrt[3]{2}$ to this precision requires a memory stack size of around 1024MB (the default is 4MB).

\newpage

\section{Using Matrices}

Given $f(x)=x^3-px-q$ with $\theta$ as a root, we represent an element $\zeta=x+y\theta+z\theta^2$ of $K=\Q(\theta)$ by a vector\footnote{Mathematica treats \verb|{x,y,z}| as either a row vector or column vector, depending on the context.} $(x,y,z)\in\Q^3$ or as a matrix $M_\zeta:=xI+yT+zT^2\in\Q[T]$, where
\begin{equation*}
T=
\begin{pmatrix}
0 & 0 & q\\
1 & 0 & p\\
0 & 1 & 0
\end{pmatrix}
\end{equation*}
is the companion matrix of $f$. We add, subtract, and scalar-multiply componentwise. To multiply $\zeta_1$ and $\zeta_2$ we compute $M_{\zeta_1}\zeta_2$. To divide $\zeta_1/\zeta_2$, we compute $M_{\zeta_2}^{-1}\zeta_2$. (If $\zeta_2\neq 0$, then $M_{\zeta_2}$ is nonsingular.) To compute $\zeta^n$ for $n>0$, we can use repeated squaring on the matrix $M_\zeta$, and then \begin{equation*}
\zeta^n\longleftrightarrow M_\zeta^n\begin{pmatrix}1\\0\\0\end{pmatrix}.
\end{equation*}

We assume we already have $\alpha$ and $\beta$, and have transformed the minimal polynomial of $\alpha$ into the form $g(x)=x^3-px-q$. We initialize the field by
building the companion matrix $T$ for $g$, building the vector $(I,T,T^2)$ (which we call \verb|TPowers|), setting up a vector of roots $(\theta_1,\theta_2,\theta_3)$ of $g$, and putting $\theta=\theta_1$. We still use PARI/GP to find $d$ and a unit $\lambda>1$. We list some of the Mathematica functions we used.

\subsection{Computing in $K$}
\begin{enumerate}
\item \verb|KMat[v]| takes the vector \verb|v={x,y,z}| and returns the matrix \verb|v.TPowers| (which is $xI+yT+zT^2$).
\item \verb|KN[a_] := Det[KMat[a]];  (* Norm of a *)|
\item \verb|KTr[a_] := Tr[KMat[a]];  (* Trace of a *)|
\item \verb|KMul[a_, b_] := KMat[a].b   (* multiply a and b in K *)|
\item \verb|KInv[a_] := Inverse[KMat[a]].UnitVector[3, 1]; (* get a^{-1} *)|
\item \verb|KDiv[a_, b_] := Inverse[KMat[b]].a (* a/b in K *)|
\\

\begin{samepage}
\item \begin{verbatim}
(*  calculate a^n (n>0) in K by repeated squaring  *)
KExpPos[a_, n_] := Module[{digits, A, b, i},

  (* get bits of n, highest to lowest *)
  digits = IntegerDigits[n, 2];
  A = KMat[a];              (* matrix for a *)

  (*  repeated squaring step  *)
  b = a;   (* first bit is 1 *)
  For[i = 2, i <= Length[digits], i++,
   b = KMat[b].b;                        (* squaring *)
   If[digits[[i]] == 1,      (* if ith bit is 1, multiply by a *)
    b = A.b;
    ];
   ];
  b
  ]
\end{verbatim}
\end{samepage}
\item
\begin{verbatim}
(*  calculate a^n for any int n   *)
KExp[a_, n_Integer] :=
 If[n > 0, KExpPos[a, n],
  If[n == 0, UnitVector[3, 1],
   KExpPos[KInv[a], -n]
   ]
  ]
\end{verbatim}

\subsection{Calculating $\lb x \rb$, $\|x\|$, $n\|n\alpha\|\|n\beta\|$}

\item \verb|mod1[x_] := x - Round[x]| (this is $\lb x\rb$)
\item \verb|LWNorm[x_] := Abs[mod1[x]]| (this is $\|x\|$)
\item \verb|(*   find LW product to precision prec   *)|\\
\verb|LWProd[n_, a_, b_, prec_] := Abs[N[n mod1[n a] mod1[n b], prec]]|

\subsection{Calculating Sequences of Powers}
\begin{samepage}
\item
\begin{verbatim}
(*  get seq of a^i from m to n *)
KExpSeq[a_, m_, n_] := Module[{A, ai, i, seq},
  A = KMat[a];
  ai = KExp[a, m];
  seq = {ai};

  For[i = m, i < n, i++,
   ai = A.ai;
   seq = Append[seq, ai];
   ];
  seq
  ]
\end{verbatim}
\end{samepage}
\subsection{Calculating the Point $(n^{1/2}\lb n\alpha\rb,n^{1/2} \lb n\beta\rb)$}
\item
\begin{verbatim}
(*   (n^(1/2)<na>,n^(1/2)<nb>)   *)
SqrtMod1Pt[n_, list_, prec_] := N[Sqrt[n] mod1[n list], prec]
\end{verbatim}

\end{enumerate}

\subsection{Calculating the $m_n$'s}
We let \verb|lam = {a,b,c}| be the vector for $\lambda=a+b\theta+c\theta^2$ in the $1,\theta,\theta^2$ basis. We can compute individual $m_n$'s or compute a sequence of them. To get an individual $m_n$, we put
\begin{equation*}
\verb|mn = const KExp[lam,n]|
\end{equation*}
(where \verb|const| is some constant depending on $\alpha$ and $\beta$ and $d$). To find $m_n^{1/2}\lb m_n\alpha\rb$ and $m_n^{1/2}\lb m_n\beta\rb$ we used
\begin{alltt}
SqrtMod1[mn, \(\alpha\), \(\beta\)]
\end{alltt}
and to find $m_n\|m_n\alpha\|\|m_n\beta\|$ (to \verb|prec| significant digits) we used
\begin{alltt}
LWProd[mn,\(\alpha\),\(\beta\),prec]
\end{alltt}

\smallskip

To get a sequences of these, we put (for $m<n$)
\begin{equation*}
\verb|lamnseq = KExpSeq[lam, m, n];|
\end{equation*}
and
\begin{equation*}
\verb|mnseq = const #[[3]] & /@ lamnseq;|
\end{equation*}
This picks the last entry of each vector for $\lambda^n$ in \verb|lamnseq|, and scales by \verb|const|. For a sequence of points $(m_n^{1/2}\lb m_n\alpha\rb,m_n^{1/2}\lb m_n\beta^2\rb)$, we used
\begin{alltt}
mnptseq = Map[SqrtMod1Pt[#, \(\alpha\), \(\beta\), prec] &, mnseq];
\end{alltt}

\subsection{Increasing Precision}
Since $\{m_n\}$ grows exponentially, we soon need to increase \verb|$MaxExtraPrecision| to calculate $m_n^{1/2}\lb m_n\alpha\rb$ and $m_n^{1/2}\lb m_n\beta\rb$ with any accuracy. (In Mathematica~8.0, the default is 50.) To compute these for a few thousand $m_n$'s, \verb|$MaxExtraPrecision=10000| works. For some of the examples (like computing $m_{Q_n}\|m_{Q_n}\alpha\|\|m_{Q_n}\beta\|$ when $m_{Q_n}$ has tens of millions of digits), we used \verb|$MaxExtraPrecision=10^250000000| (250 million).